\documentclass[12pt]{amsart}

\usepackage{amsmath}
\usepackage{amsfonts,amssymb}
\usepackage{euscript}
\usepackage{amsthm}
\usepackage[matrix,arrow,curve]{xy}
\usepackage{array}
\usepackage{tikz}
\usetikzlibrary{arrows,automata}
\usetikzlibrary{topaths}

\numberwithin{equation}{section}
\theoremstyle{plain}
\newtheorem{theorem}{Theorem}[section]
\newtheorem{lemma}[theorem]{Lemma}
\newtheorem{predl}[theorem]{Proposition}
\newtheorem{corollary}[theorem]{Corollary}

\newtheorem{conjecture}[theorem]{Conjecture}
\theoremstyle{definition}
\newtheorem{definition}[theorem]{Definition}
\newtheorem{remark}[theorem]{Remark}
\newtheorem{example}[theorem]{Example}

\newcommand{\R}{\mathbb R}

\newcommand{\Z}{\mathbb Z}

\renewcommand{\P}{\mathbb P}

\newcommand{\D}{\mathcal D}

\newcommand{\EE}{\mathcal E}

\newcommand{\LL}{\mathcal L}

\newcommand{\TT}{\mathsf T}

\renewcommand{\O}{\mathcal O}

\renewcommand{\k}{\mathsf k}

\newcommand{\xra}{\xrightarrow}
\renewcommand{\le}{\leqslant}
\renewcommand{\ge}{\geqslant}

\DeclareMathOperator{\Hom}{\textup{Hom}}

\DeclareMathOperator{\Pic}{\mathrm{Pic}}

\DeclareMathOperator{\rank}{\mathrm{rank}}

\DeclareMathOperator{\coh}{\mathrm{coh}}

\begin{document}

\title{On exceptional collections of line bundles on weak del Pezzo surfaces}

\author{Alexey ELAGIN}
\thanks{The research was supported by RFBR grant 15-01-02158 and by the Simons Foundation}
\address{Institute for Information Transmission Problems (Kharkevich Institute), Moscow, RUSSIA\\
National Research University Higher School of Economics (HSE), Moscow, RUSSIA}
\email{alexelagin@rambler.ru}

\author{Junyan XU}
\thanks{Indiana University}
\address{Department of Mathematics, Indiana University, 831 E. Third St., Bloomington, IN 47405, USA}
\email{xu56@umail.iu.edu}

\author{Shizhuo ZHANG}
\thanks{Indiana University}
\address{Department of Mathematics, Indiana University, 831 E. Third St., Bloomington, IN 47405, USA}
\email{zhang398@umail.iu.edu}

\begin{abstract}We study full exceptional collections of line bundles on surfaces. We prove that any full strong exceptional collection  of line bundles on a weak del Pezzo surface of degree $\ge 3$ is an augmentation in the sense of L.\,Hille and M.\,Perling, while for some weak del Pezzo surfaces of degree $2$ the above is not true. We classify smooth projective surfaces possessing a cyclic strong exceptional collection of line bundles of maximal length: we prove that they are  weak del Pezzo surfaces and find all types of weak del Pezzo surfaces admitting such a collection. We find simple criteria of exceptionality/strong exceptionality for collections of line bundles on weak del Pezzo surfaces.
\end{abstract}

\maketitle

\tableofcontents

\section{Introduction}

Our paper is devoted to the study of exceptional collection of line bundles on surfaces. 
This study was started by Lutz Hille and Markus Perling in the paper~\cite{HP}. 
Among the questions that have been addressed in the above paper are the following
\begin{itemize}
\item Which surfaces admit full and exceptional/strong exceptional/cyclic strong exceptional collections of line bundles?
\item Can one describe or construct effectively all full and exceptional/strong exceptional collections of line bundles on a given surface?
\end{itemize}

It is believed that any variety with a full exceptional collection in the bounded derived category of coherent sheaves is rational. Though, there is no proof yet even in the case when collection is formed by line bundles. On the other hand, on any rational surface one can construct a full exceptional collection of line bundles, using a construction by Dmitry Orlov \cite{Or} and several mutations. 

For strong exceptional collections of line bundles the question is much more complicated.
It was conjectured by Alastair King \cite{Ki} that any smooth toric variety has a strong exceptional collection formed by line bundles. In \cite{HP2} Hille and Perling constructed a counterexample: a toric surface that admits no full strong exceptional collections of line bundles. Later in \cite[Theorem 8.2]{HP} they established a criterion for toric surfaces determining whether a toric surface  admits a full strong exceptional collection of line bundles or not. It was demonstrated  that such a collection exists if and only if the toric surface can be obtained from some Hirzebruch surface by two blow-up operations: on each step one can blow up several distinct points. Also it was 
explained that ``if'' part of the above statement holds for any rational surface, see Theorem 5.9 in \cite{HP}. Therefore one arrives at a reasonable

\begin{conjecture}
\label{conj_twosteps}
A rational surface $X$ has a full and strong exceptional collection of line bundles in the derived category if and only if $X$ can be obtained from some Hirzebruch surface by at most two steps of blowing up points (maybe several at each step).
\end{conjecture}

For the study of the above questions, Hille and Perling proposed two notions.
First, it is reasonable to pass from an exceptional collection 
$$(\O_X(D_1),\ldots,\O_X(D_n))$$ 
of line bundles (where $D_i$ are divisors on $X$) to the sequence
$A_1,\ldots,A_{n-1}$ of differences: $A_k=D_{k+1}-D_k$. It is useful also to complete the sequence $D_1,\ldots,D_n$ to an infinite helix $D_i, i\in \Z$, by the rule $D_{k+n}=D_k-K_X$ for all $k$, and
to add a term
$A_n=D_{n+1}-D_n=D_1-K_X-D_n$. Thus one gets a sequence
$$A=(A_1,\ldots,A_n),$$
which is naturally cyclic-ordered and which contains essentially the same information as the original collection 
$(\O_X(D_1),\ldots,\O_X(D_n))$. This sequence $A$ is called a \emph{toric system}. Elements $A_1,\ldots,A_n$ generate $\Pic X$ and have nice combinatorial properties:
\begin{itemize}
\item $A_i\cdot A_{i+1}=1$;
\item $A_i\cdot A_j=0$ if $j\ne i,i\pm 1$;
\item $\sum_i A_i=-K_X$.
\end{itemize}
These properties comprise a definition of a toric system, see Definition~\ref{def_ts}.
We say that a toric system has some property P (like full, exceptional etc) if the corresponding collection has property P.
The main characteristic feature of a toric system $A$ is the sequence 
$$A^2=(A_1^2,\ldots,A_n^2)$$
of self-intersections of its divisors. By an important theorem by Hille and Perling, for any toric system $A$ there exists a smooth projective toric surface $Y$ with torus-invariant divisors $T_1,\ldots,T_n$ (listed in the cyclic order) such that $A_i^2=T_i^2$ for all $i$. This theorem imposes strict constraints on possible sequences 
$A^2$, we call such sequences \emph{admissible}. 

The second discovery of \cite{HP} is a notion of augmentation. It enables, starting from a toric system $A=(A_1,\ldots,A_n)$ on some surface $X$, to obtain a toric system on the blow-up $X'$ of $X$ at a point. Namely, let $1\le m\le n+1$ be an index and let $E$ denote the exceptional divisor of the blow-up. Then the augmented toric system on $X'$ is
$${\rm augm}_m(A)=(A'_1,\ldots,A'_{m-2},A'_{m-1}-E,E,A'_m-E,A'_{m+1},\ldots,A'_n),$$
where $A'_i$ is the pull-back of $A_i$. One can start from some toric system on a Hirzebruch surface and perform several blow-ups, augmenting the toric system at each step. The resulting toric system is called a \emph{standard augmentation}. This construction gives us some (quite many) examples of toric systems on rational surfaces. 

For standard augmentations Hille and Perling proved the following: a toric system on a rational surface $X$ which is a standard augmentation along some sequence of blow-ups $X\to\ldots\to \mathbb F_n$ corresponds to a strong exceptional collection if and only if $X$ was obtained from the Hirzebruch surface $\mathbb F_n$ in at most two blow-ups (each time one can blow up several different points).   Therefore Conjecture~\ref{conj_twosteps} holds for any full strong exceptional collection  which comes from a standard augmentation. 

Therefore one would like to check that any full strong exceptional collection of line bundles corresponds to a toric system which is a standard augmentation. The above cannot be literally true, but is close to being true if we will not distinguish between exceptional collections which differ only by the ordering of line bundles inside completely orthogonal blocks of bundles. We refer to the operation of switching two orthogonal (or numerically orthogonal) line bundles which are neighbors in the collection 
as to \emph{permutation}:
\begin{multline*}
(\O_X(D_1),\ldots,\O_X(D_k),\O_X(D_{k+1}),\ldots,\O_X(D_n))\mapsto\\ \mapsto (\O_X(D_1),\ldots,\O_X(D_{k+1}),\O_X(D_{k}),\ldots,\O_X(D_n)).
\end{multline*}

Hille and Perling made the following 
\begin{conjecture}
\label{conj_augm}
Any strong exceptional toric system on a rational surface is (up to some permutations) a standard augmentation. 
\end{conjecture} 
They proved this conjecture for toric surfaces, thus proving for toric surfaces Conjecture~\ref{conj_twosteps}. Andreas Hochenegger and Nathan Ilten in \cite[Main Theorem 3]{HI} proved Conjecture~\ref{conj_augm} for toric surfaces of Picard rank $\le 4$ without strongness assumption. The first author and Valery Lunts in \cite[Theorem 1.4]{EL} proved  Conjecture~\ref{conj_augm} for del Pezzo surfaces and arbitrary toric systems (without exceptionality assumption). 

Main goal of this paper was to investigate Conjecture~\ref{conj_augm} for weak del Pezzo surfaces. Recall that a weak del Pezzo surface is a smooth projective surface $X$ such that $K_X^2>0$ and $-K_X$ is numerically effective. Alternatively, $X$ is the minimal resolution of singularities on a singular del Pezzo surface having only rational double points as singularities.

We have proved Conjecture~\ref{conj_augm} for all weak del Pezzo surfaces of degree $K_X^2\ge 3$. What is more interesting, we have found a counterexample for some weak del Pezzo surfaces of degree $2$. The proof relies on the classification of weak del Pezzo surfaces by the configuration of irreducible $(-2)$-curves. For small degrees ($3,4,5$) the proof also involves some computer calculations. Note here that our counterexamples do not contradict Conjecture~\ref{conj_twosteps}: any of the surfaces possessing a counterexample can be obtained from a Hirzebruch surface by two blow-ups.

\medskip
Of special interest are strong exceptional collections $(\O_X(D_1),\ldots,\O_X(D_n))$ that remain strong exceptional after ``cyclic shifts'': that is, any segment 
$(\O_X(D_{k+1}),\ldots,\O_X(D_{k+n}))$ in the helix is also strong exceptional. Such collections are called \emph{cyclic strong exceptional}.

In Hille and Perling's paper the following results concerning full cyclic strong exceptional collections of line bundles were obtained:
\begin{itemize}
\item if a rational surface $X$ has a full cyclic strong exceptional collection of line bundles then \\$\deg X\ge 3$;
\item any del Pezzo surface $X$ with $\deg X\ge 3$ has a full cyclic strong exceptional collection of line bundles;
\item a toric surface $X$ has a full cyclic strong exceptional collection of line bundles if and only if $-K_X$ is numerically effective.
\end{itemize}

In this paper we give a complete classification of surfaces admitting a full cyclic strong exceptional collection of line bundles. First, any such surface is a weak del Pezzo surface (in particular, it is rational). Second, for any type of weak del Pezzo surfaces we determine whether they possess a full cyclic strong exceptional collection of line bundles.

\begin{theorem}[Propositions~\ref{prop_cyclicwdp} and \ref{prop_classification}]
Let $X$ be a smooth projective surface having a cyclic strong exceptional collection of line bundles of maximal length. Then $X$ is a weak del Pezzo surface and $X$ is one of the surfaces from Table~\ref{table_yes}. Moreover, any weak del Pezzo surface from  Table~\ref{table_yes} possesses a full cyclic strong exceptional collection of line bundles.
\end{theorem}

Matthew Ballard and David Favero in \cite{BF} discovered a relation between cyclic strong exceptional collections in the derived category $\D^b(\coh X)$ of coherent sheaves on a smooth projective variety $X$ and dimension of the category $\D^b(\coh X)$ in the sense of Rouquier. They demonstrated, in particular, that a full strong exceptional collection $(\EE_1,\ldots,\EE_n)$ in $\D^b(\coh X)$ is cyclic strong if the generator $\oplus_i\EE_i$ has generation time equal to $\dim X$. 
For collections of line bundles on surfaces we are able to provide a sort of  converse statement.

\begin{theorem}[Proposition \ref{prop_timetwo}]
Let $X$ be a smooth projective surface with a full cyclic strong exceptional collection
$$(\LL_1,\ldots,\LL_n)$$
of line bundles. Then the generator $\oplus_i\LL_i$ of the category $\D^b(\coh X)$ has the generation time two. In particular, for any surface $X$ from Table~\ref{table_yes} 
the category $\D^b(\coh X)$ has dimension two.
\end{theorem}

The above theorem provides a new class of varieties $X$ for which dimension of $\D^b(\coh X)$ equals to the dimension of $X$. This confirms a conjecture by Orlov saying that $\dim \D^b(\coh X)=\dim X$ for all smooth projective varieties $X$.

\medskip
Augmentations can help to set up a relation between properties ``full'' and ``of maximal length'' of exceptional collections of line bundles. Recall that an exceptional collection in $\D^b(\coh X)$ is \emph{full} if its objects generate $\D^b(\coh X)$ as triangulated category and has \emph{maximal length} if its objects generate the Grothendieck group $K_0(X)$ of $X$ modulo numeric equivalence.
Clearly, any full exceptional collection has maximal length. The converse in not true: there are recent examples of surfaces of general type (classical Godeaux surface, \cite{BBS} or Barlow surface, \cite{BBKS}) possessing an exceptional collection of maximal length which is not full. But for rational surfaces there are no such examples known. It is natural to pose the following
\begin{conjecture}
\label{conj_full}
Let $X$ be a smooth rational projective surface. Then any exceptional collection in $\D^b(\coh X)$ having maximal length is full. 
\end{conjecture}
The above Conjecture is proved by Sergey Kuleshov and Dmitry Orlov for del Pezzo surfaces in \cite{KO}. Also, Conjecture~\ref{conj_full} is proved by Kuleshov 
for collections of vector bundles on rational surfaces $X$ with $-K_X$ without base components and $K_X^2>1$ (in particular, for
weak del Pezzo surfaces of degree $\ge 2$), see  \cite[Theorem 3.1.8]{Ku}.
For exceptional collections of line bundles on general rational surfaces the question is open. We point out that any standard augmentation corresponds to a FULL exceptional collection of line bundles. Therefore Conjecture~\ref{conj_augm} would imply Conjecture~\ref{conj_full} for strong exceptional collections of line bundles.

\bigskip
Let us go further into details and explain the structure of the text. Sections 2--7 contain preliminaries, they are mostly known and/or contained in a paper \cite{HP} by Hille an Perling.

We start from basics on divisors on rational surfaces in Section 2. The main definition here is the one of an $r$-class: this is a class in $\Pic X$ numerically corresponding  to a smooth   rational irreducible curve $C$  with $C^2=r$. Equivalently, one can say that a divisor $D$ on a rational surface $X$ is an $r$-class if the pair $(\O_X,\O_X(D))$ is numerically exceptional and $D^2=r$. A divisor $D$ is called \emph{left-orthogonal (resp. strong left-orthogonal)} if the pair $(\O_X,\O_X(D))$ is exceptional (resp. strong exceptional). Of special  interest to us are $(-1)$ and $(-2)$-classes. The set $R(X)$ of $(-2)$-classes on $X$ is defined by equations $C^2=-2, C\cdot K_X=0$, it  is a root system in (some sublattice of) $\Pic X$, which depends only on degree of $X$. In Section 3 we recall definitions and basic facts concerning weak del Pezzo surfaces. They are distinguished by the degree $K_X^2$ and by the configurations of irreducible $(-2)$-curves, which form
a root subsystem in $R(X)$. The set of irreducible $(-1)$-curves can be recovered as the set of all $(-1)$-classes $C$ on $X$ such that $C\cdot D\ge 0$ for all irreducible $D\in R(X)$. For example, on genuine del Pezzo surfaces $X$ there are no irreducible $(-2)$-curves and all $(-1)$-classes correspond to irreducible $(-1)$-curves. On the contrary, for complicated weak del Pezzo surfaces there are many irreducible $(-2)$-curves and few irreducible $(-1)$-curves. Also in Section 4 we prove criteria of left-orthogonality and strong left-orthogonality of $r$-classes on weak del Pezzo surfaces, see Proposition~\ref{prop_loslo}. These criteria operate only with $r$ and effectivity of some divisors. 

In Section 4 we introduce necessary notions concerning exceptional collections of line bundles and their relation to toric systems. Here we prove a criterion of toric systems on weak del Pezzo surfaces being exceptional/strong exceptional/cyclic strong exceptional, see Theorem~\ref{theorem_checkonlyminustwo}. This criterion expresses the above cohomological properties of toric systems in terms of effectivity of some divisors of the form $D=A_k+\ldots+A_l$ with $D^2\le -2$. Using this criterion, we construct examples of full cyclic strong exceptional collections of line bundles on certain surfaces, and counterexamples to Conjecture~\ref{conj_augm} on some surfaces of degree $2$.

In section 5 we introduce and study admissible sequences: they are the sequences arising as self-intersection indexes of torus-invariant divisors on smooth toric surfaces. Also, they come as sequences of the form $(A_1^2,\ldots,A_n^2)$ where $(A_1,\ldots,A_n)$ is a toric system of maximal length on some surface. We are interested in those admissible sequences that correspond to strong exceptional and cyclic strong exceptional toric systems. Hence we consider two kinds of admissible sequences. An admissible sequence $(a_1,\ldots,a_n)$ is said to be \emph{of the first kind} if $a_i\ge -2$ for all $i$ (they correspond to cyclic strong exceptional toric systems) and \emph{of the second kind} if $a_i\ge -2$ for  $1\le i\le n-1$ and $a_n<-2$ (they correspond to strong exceptional toric systems which are not cyclic strong exceptional). We classify admissible sequences of the first and of the second kind. The use of admissible sequences is crucial for our treatment of toric systems.

In Section 6 three operations with toric systems on a fixed surface are defined: permutations, cyclic shifts and symmetries.  They correspond to three operations with collections of line bundles which we describe next. Permutation can be performed if there are two numerically orthogonal line bundles $\O_X(D_k)$ and $\O_X(D_{k+1})$, then one just interchanges them. Cyclic shift and symmetry send collection
$$(\O_X(D_1),\ldots,\O_X(D_n))$$
to 
$$ (\O_X(D_2),\O_X(D_3),\ldots,\O_X(D_{n}),\O_X(D_{1}-K_X)\quad\text{and}\quad
(\O_X(D_{n}),\O_X(D_{n-1}),\ldots,\O_X(D_{1}))$$
respectively.
Cyclic shift preserves exceptional collection but in general does not preserve strong exceptional collections. Permutations preserve strong exceptional collections, but may spoil exceptional collections. We treat all three operations as essentially not changing the collection.

In Section 7 we define and discuss different variants of augmentation. Augmentation of a toric system under one blow-up of a point will be called \emph{elementary} augmentation. 
A toric system is called a \emph{standard} augmentation if it can be obtained by several elementary augmentations from a toric system on a Hirzebruch surface. 
A toric system is called an augmentation \emph{in the weak sense} if it can be obtained by several elementary augmentations, permutations and cyclic shifts from a toric system on a Hirzebruch surface. For a property P (like ``exceptional'', ``strong exceptional'', etc) it is said that a toric system is an augmentation \emph{with property P} if it can be obtained by several elementary augmentations, permutations and cyclic shifts from a toric system on a Hirzebruch surface, and any intermediate toric system has property P. Standard augmentations are the most natural but they do not exhaust all toric systems on most surfaces if we do not allow permutations. Augmentations in the weak sense
seem to be a good notion if we do not care about homological properties of collections like exceptionality. We argue that strong exceptional augmentations are suitable for accurate formulation of Conjecture~\ref{conj_augm}: any strong exceptional toric system is a strong exceptional augmentation. Note here that terminology in \cite{HP} is different: they say that a strong exceptional collection  has a \emph{normal form} that is a standard augmentation. Normal form of a collection is obtained from the original collection by permutations.

In Section 8 we introduce the Weyl group of a rational surface $X$. It is the group of isometries of $\Pic X$ generated by reflections in $(-2)$-classes. We prove that the 
natural action of the Weyl group on the set of toric systems $A$ with fixed sequence $A^2$ is free and transitive. This fact allows one to construct explicitly all toric systems with fixed $A^2$. It is not used in the paper but have been used for the computer search of counterexamples to Conjecture~\ref{conj_augm} on  surfaces of degree~$2$.

Sections 9--13 and 15 contain the main results of the paper. First, in Sections 9--12 we prove that some toric systems are augmentations in a certain sense. The proof is by induction in the Picard rank of the surface. To do the induction step, one has to perform some permutations to the given  toric system $A$ such that the new toric system $A'$ would contain some irreducible $(-1)$-curve $E$ as an element. Then one can blow $E$ down and see that  $A'$ is an elementary augmentation of some toric system $B$ on the blow down. All 
$(-1)$-classes arising as elements of toric systems $A'$ which can be obtained from $A$ by several permutations, form a subset, which we denote by $I(X,A)$, in the set of all $(-1)$-classes on $X$.  Subset $I(X,A)$ can be easily calculated as soon as the sequence $A^2=(A_1^2,\ldots,A_n^2)$ is known. Therefore, to perform the induction step one has to find an irreducible curve in the set $I(X,A)$. 
 
In Section 	9 we treat toric systems of the first kind: such toric systems $(A_1,\ldots,A_n)$ that $\chi(A_i)\ge 0$ for all $i$. For example, toric systems corresponding to cyclic strong exceptional collections are of the first kind. For  toric systems of the first kind we prove the following results.
\begin{theorem}[Theorem \ref{theorem_1kindweak}]
\label{theorem_1kindweak_}
Any toric system of the first kind on a smooth rational projective surface  is an augmentation in the weak sense.
\end{theorem}

\begin{theorem}[Theorem \ref{theorem_1kind}, Corollary~\ref{cor_cyclicaugm}]
\label{theorem_1kind_}
\begin{enumerate}
\item 
Any exceptional toric system of the first kind on a weak del Pezzo surface is an exceptional augmentation. 
\item
Any strong exceptional toric system of the first kind on a weak del Pezzo surface  is a strong exceptional augmentation. 
\item
Any cyclic strong exceptional toric system on a smooth rational projective surface  is a cyclic strong exceptional augmentation. 
\end{enumerate}
\end{theorem}

Theorem~\ref{theorem_1kindweak_} is mostly combinatorial and does not deal with geometry of surfaces or homological properties of line bundles.
Its proof  is based on the following combinatorial observation (see Lemma~\ref{lemma_IF}). 
Let $(A_1,\ldots,A_n)$ be a toric system of the first kind on a smooth rational surface $X$. Then the set $I(X,A)$ is equal to the set of all $(-1)$-classes on $X$.

Theorem~\ref{theorem_1kind_}  is a consequence of Theorem~\ref{theorem_1kindweak_}. Its proof (in contrast with the results of next sections) uses only general properties of weak del Pezzo surfaces and does not involve consideration of different types of surfaces. 
Theorem~\ref{theorem_1kind_} proves  Conjectures  \ref{conj_twosteps} and \ref{conj_augm} for full cyclic strong exceptional collections of line bundles. 

In Sections 10--12 we study exceptional toric systems of the second kind: such toric systems 
$A=(A_1,\ldots,A_n)$ that $\chi(A_i)\ge 0$ for all $1\le i \le n-1$ and $\chi(A_n)<0$. Toric systems corresponding to strong exceptional collections are either of the first or of the second kind.  We prove the following
\begin{theorem}[Theorem \ref{theorem_main}]
\label{theorem_2kind_}
Let $A$ be an exceptional toric system of the second kind on a weak del Pezzo surface of degree $\ge 3$. Then $A$ is an exceptional 	augmentation. Moreover, if $A$ is strong exceptional then $A$ is a strong exceptional augmentation.
\end{theorem}
Theorem \ref{theorem_2kind_} follows directly from the next fact.
\begin{predl}[See Proposition \ref{prop_main}]
\label{prop_main_}
Let $X$ be a weak del Pezzo surface of degree $\ge 3$ and $A=(A_1,\ldots,A_n)$ be a toric system of the second kind on $X$. Then either there is an irreducible curve in the set $I(X,A)$ (thus some permutation of $A$ is an elementary augmentation) or there exists a divisor of the form $D=A_{k}+\ldots +A_n+\ldots +A_l$ such that $-D\ge 0$ (and thus $A$ is not exceptional).
\end{predl}
We check the above Proposition by considering all admissible sequences $(A_1^2,\ldots,A_n^2)$ of the second kind and all types of  weak del Pezzo surfaces.
The arguments are different for different admissible sequences, we divide admissible sequences of the second kind into types II--VI.
For any fixed degree of $X$ (i.e. fixed length $n$) there is a finite number of sequences of types III to VI and several infinite series of sequences of type II.

We demonstrate the lines of the proof of Proposition \ref{prop_main_} on an example. Let us consider toric systems $A$ with 
$$A^2=(-1,-2,-2,0,0,-2,-2,-1,-5)$$
on  surfaces of degree $3$ (such sequence  $A^2$ is of type III).
We denote $H=A_1+A_2+A_3+A_4$ and check that $H=A_5+A_6+A_7+A_8$ and  $H$ is a $1$-class.
Assuming that there are no irreducible $(-1)$-curves in the set $I(X,A)$, we deduce that $H$ is ``good'' in the following sense: for any irreducible $(-1)$-curve $C$ one has $C\cdot H\ge 1$. By the definition of a toric system, $-A_9=2H+K_X$. 
To prove Proposition \ref{prop_main_} in this case we will show that for any good $1$-class~$H$ on a weak del Pezzo surface of degree $3$ one has $2H+K_X\ge 0$. This is done directly by means of computer calculations. For any weak del Pezzo surface of degree $3$ we find, using a computer, all good $1$-classes and check inequality $2H+K_X\ge 0$ for each of them.

For other types of admissible sequences the arguments are similar (though in some cases more complicated). They involve some amount of computer calculations made using \textbf{Sage} computer algebra system. All scripts can be found and run online on \textbf{CoCalc} server at \\ \texttt{https://cocalc.com/projects/2130c0a6-a36a-4fec-9ee8-66c0f5ed53dd/files} 

Theorems \ref{theorem_1kind_} and \ref{theorem_2kind_} imply that Conjectures \ref{conj_twosteps} and \ref{conj_augm}  hold for weak del Pezzo surfaces of degree $\ge 3$. 
In other words, we have the following:
\begin{theorem}[Corollaries \ref{cor_main} and \ref{cor_aug}]
Let $A$ be a strong exceptional toric system on a weak del Pezzo surface of degree $\ge 3$. Then $A$ is a strong exceptional 	augmentation and
 $X$ can be obtained by blowing up a Hirzebruch surface or $\P^2$ at most twice.
\end{theorem}
It is interesting that the strong exceptionality condition 
is used in the proof of the above Theorem  only in one place: to deduce that $\chi(A_i)\ge 0$ for $i=1\ldots n-1$ (and therefore the sequence $A^2$ is of the first or of the second kind). So it is natural to ask whether any exceptional toric system on a rational 	surface of degree $\ge 3$ is an exceptional augmentation.

In Section 13 we present a counterexample to Conjecture \ref{conj_augm}. This is a weak del Pezzo surface of degree $2$ and a full strong exceptional collection of line bundles on it which is not an augmentation in the weak (and thus in any reasonable) sense. 
This example was found using a computer. For all sequences $A^2$ of types III--VI and for some sequences $A^2$ of type II, for all weak del Pezzo surfaces $X$ of degree $2$ we checked whether there are any  counterexamples on~$X$ with the given $A^2$. All counterexamples found in degree $2$ have the same type:
$$A^2=(-1,-2,-2,-2,-1,-2,-2,-1,-2,-3)$$ 
or the symmetric one. All surfaces possessing a counterexample can be obtained from a Hirzebruch surface by two blow-ups. This gives some evidence that   Conjecture~\ref{conj_twosteps} should be true. It is interesting to note that all surfaces  where we found a counterexample have holes in the effective cone: such non-effective divisor classes which have a positive multiple being effective.
Discussion of counterexamples to Conjecture~\ref{conj_augm} comprises Section 14. Details about computer search of counterexamples and checking divisors effectivness is contained in Appendix A.

In Section 15 we classify smooth projective surfaces which have a cyclic strong exceptional collection of line bundles having maximal length. Using description of admissible sequences of the first type, we prove that any such surface is rational and has degree $\ge 3$. Further we prove that such surface is a weak del Pezzo surface. Then, for all types of weak del Pezzo surfaces we construct an example of a cyclic strong exceptional collection of line bundles or prove that it does not exist.

In Section 16 we give application to dimension of derived categories of coherent sheaves.

\medskip
{\bf Acknowledgements.} The authors would like to thank Valery Lunts, collaboration with whom initiated this project. We thank Michael Larsen for valuable remarks on the text. The first author is grateful to Indiana University for their hospitality. The third author would like to thank Li Tang for his support.

\section{Divisors on surfaces and $r$-classes}

Let $X$ be  a rational smooth projective surface over an algebraically closed field $\k$ of characteristic zero.  Let $K_X$ be a canonical divisor on $X$.
Let $d=K_X^2$ be the degree of $X$, further we always assume that $d>0$. 
The Picard group $\Pic(X)$ of $X$ is a finitely generated abelian group of rank $10-d$.
It is equipped with the intersection form $(D_1,D_2)\mapsto D_1\cdot D_2$ which has signature $(1,9-d)$.
For a divisor $D$ on $X$, we will use the following shorthand notations:
$$H^i(D):=H^i(X,\O_X(D)),\quad h^i(D)=\dim H^i(D),\quad \chi(D)=h^0(D)-h^1(D)+h^2(D).$$
By the Riemann-Roch formula, one has
$$\chi(D)=1+\frac{D\cdot (D-K_X)}2.$$
 
The following notions are introduced by Lutz Hille and Markus Perling in \cite[Definition 3.1]{HP}. 
\begin{definition}
A divisor $D$ on $X$ is \emph{numerically left-orthogonal} if $\chi(-D)=0$ (or equivalently $D^2+D\cdot K_X=-2$).
A divisor $D$ on $X$ is \emph{left-orthogonal} (or briefly \emph{lo}) if $h^i(-D)=0$ for all $i$. A divisor $D$ on $X$ is \emph{strong left-orthogonal} (or briefly \emph{slo}) if $h^i(-D)=0$ for all $i$ and $h^i(D)=0$ for $i\ne 0$.
\end{definition} 

\begin{definition}
We call $D$ an \emph{$r$-class} if $D$ is numerically left-orthogonal and $D^2=r$. \end{definition}

Motivation: if $C\subset X$ is a smooth rational irreducible curve, then the class of $C$ in $\Pic X$ is an $r$-class where $r=C^2$. 

If $C$ is an irreducible reduced curve on $X$ and $r=C^2$, it is said that $C$ is an \emph{$r$-curve}. An $r$-curve is \emph{negative} if $r<0$.

The next propositions are easy consequences of Riemann-Roch formula, see~\cite[Lemma 3.3]{HP} or \cite[Lemma 2.10, Lemma 2.11]{EL}.
\begin{predl}
Let $D$ be a numerically left-orthogonal divisor on $X$. Then 
$$\chi(D)=D^2+2=-D\cdot K_X.$$
\end{predl}
\begin{predl}
\label{prop_classes0}
Suppose $D_1,D_2$ are numerically left-orthogonal divisors on $X$. 
Then $D_1+D_2$ is numerically left-orthogonal if and only if $D_1D_2=1$. If that is the case, then 
$$\chi(D_1+D_2)=\chi(D_1)+\chi(D_2)\qquad \text{and}\qquad (D_1+D_2)^2=D_1^2+D_2^2+2.$$
\end{predl}

The following three lemmas are elementary exercises.
\begin{lemma}
\label{lemma_dprime}
Let $D$ be an $r$-class on a rational surface $X$ of degree $d$. Then $D'=-K_X-D$ is an $r'$-class where $r+r'=d-4$. 
\end{lemma}

\begin{lemma}
\label{lemma_classes1}
Let $D_0$ be an $r$-class on a rational surface $X$ and $D_1,\ldots,D_m$ be $(-1)$-classes such that $D_iD_j=0$ for $0\le i<j\le m$. Then $D_0-D_1-\ldots-D_m$ is an $(r-m)$-class on~$X$. 
\end{lemma}

\begin{lemma}
\label{lemma_classes2}
Let $D$ be an $r$-class on a rational surface $X$ and $F$  be a $0$-class  such that $DF=1$. Then $D+mF$ is an $(r+2m)$-class on $X$ for any $m\in \Z$. 
\end{lemma}

\medskip
Denote the set of $(-1)$-classes on $X$ by $I(X)$. We treat $I(X)$ as a graph: two vertices $D_1\ne D_2\in I(X)$ are connected by $m$ edges if $D_1\cdot D_2=m$. Note that the graph $I(X)$ depends only on $\deg(X)$. If one takes $X$ to be a del Pezzo surface, then any $(-1)$-class corresponds to an irreducible curve, and we have $D_1\cdot D_2\ge 0$. Therefore the above definition of a graph makes sense. Also note that $I(X)$ has no multiple edges for $\deg(X)\ge 3$. Denote by $I^{\rm irr}(X)\subset I(X)$ the full subgraph of irreducible $(-1)$-classes and by $I^{\rm red}(X)\subset I(X)$ the complement of $I^{\rm irr}(X)$. These subgraphs depend on the surface~$X$. Clearly, $I(X)=I^{\rm red}(X)\sqcup I^{\rm irr}(X)$.

Denote the set of $(-2)$-classes on $X$ by $R(X)$. It is a root system in some subspace in $N_X=(K_X)^{\perp}\subset \Pic(X)\otimes\R$ (see Yuri Manin's book \cite{Ma}) and depends only on $\deg(X)$. If $\deg(X)\le 6$ then $R(X)$ spans $N_X$. 

\begin{table}[h]
\caption{Root systems $R(X)$}
\begin{center}
\begin{tabular}{|c|c|c|c|c|c|c|c|}
		 \hline
			degree & 7 & 6& 5&4&3&2&1\\
		 \hline
			type & $A_1$ & $A_1+A_2$ & $A_4$ & $D_5$ & $E_6$ & $E_7$ & $E_8$\\
		 \hline	
			$|R(X)|$ & $2$ & $8$ & $20$ & $40$ & $72$ & $126$ & $240$\\
		 \hline	
			$|I(X)|$ & $3$ & $6$ & $10$ & $16$ & $27$ & $56$ & $240$\\
		 \hline	
\end{tabular}
\end{center}

\label{table_root}
\end{table}

Denote by $R^{\rm eff}(X)\subset R(X)$ the subset of effective $(-2)$-classes. Denote by $R^{\rm irr}(X)\subset R^{\rm eff}(X)$ the subset of irreducible $(-2)$-curves. 
Denote by $R^{\rm slo}(X)\subset R^{\rm lo}(X)\subset R(X)$ the subsets of strong left-orhogonal and left-orthogonal divisors respectively. Again, $R^{\rm eff}(X),R^{\rm irr}(X),R^{\rm lo}(X)$ and $R^{\rm slo}(X)$ depend on the surface $X$.

We finish this section with a lemma that will be used later.
\begin{lemma}
\label{lemma_RS}
Let $X$ be a rational surface of degree $d$, let $D$ be a $(d-2)$-class on~$X$. Then $D\cdot C\ge 0$ for any $(-1)$-class $C$ on $X$ unless $d=1$ and $C=D$.
\end{lemma}
\begin{proof}
By Lemma~\ref{lemma_dprime}, $D=-K_X+R$ where $R$ is a $(-2)$-class on $X$. Consider a sequence of $d-1$ (arbitrary) blow-ups $Y\to X$ such that $Y$ is a surface of degree $1$. There is an embedding $\Pic(X)\to \Pic(Y)$ and we will identify $\Pic(X)$ with its image in $\Pic(Y)$. By Lemma~\ref{lemma_dprime}, $C=-K_Y+S$ where $S$ is a $(-2)$-class on $Y$. We have $K_Y=K_X+E_1+\ldots+E_{d-1}$ where $E_i$ denote exceptional curves of the blow-ups. It follows that $R\cdot K_Y=0$ and $R$ is a $(-2)$-class on $Y$. Now we have
\begin{multline*}
D\cdot C=
(-K_X+R)(-K_Y+S)=(-K_Y+(E_1+\ldots+E_{d-1})+R)(-K_Y+S)=\\
=K_Y^2+(E_1+\ldots+E_{d-1})C+RS.
\end{multline*}
Note that $(E_1+\ldots+E_{d-1})C=0$ because $C\in \Pic(X)$. Further, $R,S\in K_Y^{\perp}\subset \Pic(Y)$ and the intersection form on $K_Y^{\perp}$ is negative definite. It follows that $|RS|\le \sqrt{|R^2||S^2|}=2$, therefore $RS\ge -2$.  We get
$$D\cdot C=K_Y^2+RS=1+RS\ge -1.$$ 
If $D\cdot C=-1$, then $RS=-2$. Consequently, $R=S$ and $K_Y=S-C=R-C\in\Pic(X)$ which implies that $X=Y$, $d=1$ and $C=D$. 
\end{proof}

\section{Weak del Pezzo surfaces}
\label{section_wdp}
By definition, a weak del Pezzo surface is a smooth connected projective rational surface~$X$ such that $K_X^2>0$ and $-K_X$ is nef. A del Pezzo surface is a  smooth connected projective rational surface $X$ such that $-K_X$ is ample. We refer to Igor Dolgachev~\cite[Chapter 8]{Do} or Ulrich Derenthal~\cite{De} for the main properties of weak del Pezzo surfaces. Every weak del Pezzo surface $X$ except for Hirzebruch surfaces $\mathbb F_0$ and $\mathbb F_2$ is a blow-up of $\P^2$ at several (maybe infinitesimal) points. That is, there exists a sequence 
$$X=X_n\xra{p_n} X_{n-1}\xra\ldots X_1\xra{p_1} X_0=\P^2$$
of $n$ blow-ups, where $p_k$ is the blow-up of point $P_k\in X_{k-1}$. 
Moreover, the surface $X_n$ as above is a weak del Pezzo surface if and only if $n\le 8$ and for any $k$ the point $P_k$ does not belong to a $(-2)$-curve on $X_{k-1}$.

Let $X$ be a blow-up of $\P^2$ at $n$ (maybe infinitesimal) points. The Picard group $\Pic X$ of $X$ has the standard basis $L,E_1,\ldots,E_n$. The  intersection form is given by 
$$L^2=1,E_i^2=-1,L\cdot E_i=0, E_i\cdot E_j=0 \quad \text{for}\quad i\ne j.$$
We will use the following shorthand notation:
$$E_{i_1\ldots i_k}=E_{i_1}+\ldots+E_{i_k},\quad L_{i_1\ldots i_k}=L-E_{i_1\ldots i_k}.$$

\begin{lemma}\label{neg curve}
Let $C\subset X$ be a negative curve on a weak del Pezzo surface $X$.  Then $C$ is a smooth rational curve and $C^2=-2$ or $-1$. If $X$ is a del Pezzo surface then $C^2=-1$.
\end{lemma} 
\begin{proof}
By the Riemann-Roch formula, one has 
$$\chi(\O_C)=\chi(\O_X)-\chi(\O_X(-C))=-\frac{C^2+CK_X}{2}.$$
Since $C$ is irreducible and reduced, one has $h^0(\O_C)=1$. Hence
$$0\le h^1(\O_C)=1+\frac{C^2+CK_X}2.$$
Note that $CK_X\le 0$ because $-K_X$ is nef and $C^2<0$ by assumptions. It follows that $h^1(\O_C)=0$, hence $C$ is rational and smooth. Moreover, $C^2+CK_X=-2$, therefore $C^2=-2$ or $-1$. If $X$ is a del Pezzo surface, then $CK_X<0$ and the variant $C^2=-2$ is impossible.
\end{proof}

A weak del Pezzo surface is a del Pezzo surface if and only if it has no $(-2)$-curves. 
Irreducible $(-2)$-curves on $X$ form a subset in $R(X)$ which is a set of simple roots of some root subsystem in $R(X)$. Corresponding positive roots are effective $(-2)$-classes. Two weak del Pezzo surfaces $X$ and $Y$ are said to have the same \emph{type} if there exists an isomorphism $\Pic (X)\to \Pic(Y)$ preserving the intersection form, the canonical class and identifying the sets of negative curves. Thus, weak del Pezzo surfaces are distinguished by the configuration of irreducible $(-2)$-curves. In most cases, the configuration of irreducible $(-2)$-curves determines the type uniquely. Surfaces of different types and with the same configuration of irreducible $(-2)$-curves are distinguished by the number of irreducible $(-1)$-curves. 
A surface of degree $d$ with configuration of $(-2)$-curves $\Gamma$ and with $m$ irreducible $(-1)$-curves is denoted by $X_{d,\Gamma,m}$. Number $m$ is omitted if $d$ and $\Gamma$ determine the type uniquely.
For example, $X_{4,A_2}$ denotes a weak del Pezzo surface of degree $4$ with irreducible $(-2)$-curves forming the diagram $A_2$ and $X_{4,A_3,5}$ denotes a weak del Pezzo surface of degree $4$ with irreducible $(-2)$-curves forming the diagram $A_3$ and with $5$ irreducible $(-1)$-curves.

\begin{predl}
\label{prop_loslo}
Let $D$ be an $r$-class on a weak del Pezzo surface~$X$ of degree $d$. Then Table~\ref{table_loslo} gives necessary and sufficient condition for $D$ being left-orthogonal and strong left-orthogonal. 
\begin{table}[h]
\label{table_loslo}
\caption{Criteria of left-orthogonality and strong left-orthogonality}
\begin{center}
\begin{tabular}{|c|c|c|}
		 \hline
			$r$ & $D$ is lo? & $D$ is slo?\\
		 \hline
			$\le -3$ & iff $h^0(-D)=0$ & no \\
		 \hline
			$-2$ & iff $h^0(-D)=0$ & iff $h^0(D)=h^0(-D)=0$ \\
		 \hline	
			$-1\le r \le d-3$ & yes & yes \\
		 \hline	
			$\ge d-2$ & iff $h^0(K_X+D)=0$ & iff $h^0(K_X+D)=0$ \\
		 \hline	
\end{tabular}
\end{center}
\end{table}
\end{predl}

For the proof we will need the following easy and useful statement.
\begin{lemma}
\label{lemma_easy}
Let $D$ be a divisor on a weak del Pezzo surface $X$ such that $D\cdot (-K_X)<0$. Then $h^0(D)=0$.
\end{lemma}
\begin{proof}
Assume the contrary: $h^0(D)>0$, then we may suppose that $D$ is effective. Since $-K_X$ is nef, it follows that $D\cdot (-K_X)\ge 0$, a contradiction.
\end{proof}

\begin{proof}[Proof of Proposition~\ref{prop_loslo}]
First, we find criteria of left-orthogonality. Since $D$ is numerically left-orthogonal, we have $h^0(-D)-h^1(-D)+h^2(-D)=0$ and $D\cdot K_X=-2-r$. Therefore $D$ is lo if and only if $h^0(-D)=h^2(-D)=0$. Suppose $r\ge -1$. Then $(-D)\cdot (-K_X)=-2-r<0$ and $h^0(-D)=0$ by Lemma~\ref{lemma_easy}. Suppose $r\le d-3$. Then $(K_X+D)\cdot (-K_X)=-d+2+r<0$ and $h^2(-D)=h^0(K_X+D)=0$ by Lemma~\ref{lemma_easy} and Serre duality.
Now the statement follows.

Now we prove criteria of strong left-orthogonality. 
For a slo divisor $D$ one has $D^2=\chi(D)-2=h^0(D)-2\ge -2$, therefore divisors with $r\le -3$ are not slo. 

We claim that the condition $h^0(-D)=0$ implies that $h^2(D)=0$. First, we note  that a general  anticanonical divisor $Z\in |-K_X|$ is irreducible. Recall that by  \cite[Theorem 8.3.2]{Do}, the linear system $|-K_X|$ has no base components. If $d>1$ then $\dim|-K_X|>1$ and $Z$ is irreducible by Bertini's theorem. For $d=1$ we argue as follows. Let $X\xra{f} \P^2$ be the blow-up of  points $P_1,\ldots,P_{8}$. If $Z\in |-K_X|=|3L-\sum E_i|$ is reducible then $f(Z)$ contains a line $l_Z$ on $\P^2$. Different $Z_1$ and $Z_2$ correspond to different lines $l_{Z_1}$ and $l_{Z_2}$: otherwise the linear system $|-K_X|$ has a base component. This  contradicts  \cite[Theorem 8.3.2]{Do}. Note that $l_Z$ contains at least two of points $P_i$: otherwise there exists a conic passing through $7$ of the points $P_i$ and $X$ is not a weak del Pezzo. But there exists only a finite number of lines passing through two of eight points. 

Let $Z\in |-K_X|$ be an irreducible divisor.
Consider the standard exact sequence
$$0\to \O_X(K_X)\to \O_X\to \O_Z\to 0.$$
Twisting it by $\O_X(-D)$, we get the exact sequence
\begin{equation}
\label{eq_loslo}
0\to \O_X(K_X-D)\to \O_X(-D)\to \O_Z(-D)\to 0.
\end{equation}
It follows that $h^2(D)=h^0(K_X-D)$ is a subspace in $h^0(-D)$ which vanishes.

Let $r=-2$. Suppose $D$ is left-orthogonal, then $h^2(D)=0$ by the above. Since $D$ is numerically left-orthogonal and $\chi(D)=D^2+2=0$, $D$ is slo iff $h^0(D)=0$. The statement for $r=-2$ follows. 

It remains to demonstrate that a lo divisor $D$ with $D^2\ge -1$ is slo. We know already that $h^2(D)=0$. The long exact sequence of cohomology associated with (\ref{eq_loslo}) yields that $h^1(\O_X(K_X-D))=h^0(\O_Z(-D))$. Since $(-D)\cdot Z=D\cdot K_X<0$ and $Z$ is irreducible, we have $h^0(\O_Z(-D))=0$ and by Serre Duality $h^1(D)=h^1(K_X-D)=0$.
\end{proof}

The following consequence of Proposition~\ref{prop_loslo} is very important.
\begin{corollary}
\label{cor_effslo}
On a weak Del Pezzo surface $X$ we have 
$$R=R^{\rm eff}\sqcup(-R^{\rm eff})\sqcup R^{\rm slo};\qquad R^{\rm lo}=R^{\rm eff}\sqcup R^{\rm slo}.$$
\end{corollary}

For future use we also prove the next three lemmas.
\begin{lemma}
\label{lemma_eff}
Let $D$ be an $r$-class on a weak del Pezzo surface~$X$ of degree $d$. Suppose $r\ge -1$. Then $D$ is effective. 
\end{lemma}
\begin{proof}
We have $h^2(D)=h^0(K_X-D)$. Since $(K_X-D)\cdot (-K_X)=-d-r-2<0$, Lemma~\ref{lemma_easy} implies that $h^0(K_X-D)=0$.
Hence $h^0(D)\ge h^0(D)-h^1(D)=\chi(D)=D^2+2>0$.
\end{proof}

\begin{lemma}
\label{lemma_ch0}
Let $X$ be a rational surface of degree $3\le d\le 7$ and $H$ be a $1$-class on $X$. Then $CH\ge 0$ for any $C\in I(X)$ and 
$$|\{C\in I(X)\mid CH=0\}|=9-d.$$
\end{lemma}
\begin{proof}
Since $\Pic(X)$ and its subsets of $r$-classes depend only on $\deg(X)$, we may assume that $X$ is a del Pezzo surface. Then the divisor $H$ is nef. Indeed: $H$ is effective by Lemma~\ref{lemma_eff}, therefore $H\cdot C\ge 0$ for any irreducible curve $C$ with $C^2\ge 0$.
By Lemma~\ref{lemma_RS}, $H\cdot C\ge 0$ for any $(-1)$-curve $C$, hence $H$ is nef. By~\cite[Lemma 1.7]{De} it follows that the linear system $|H|$ has no base points. Further, $H$ is slo by Proposition~\ref{prop_loslo} and we have $h^0(H)=H^2+2=3$.
Therefore
$|H|$
defines a regular map $\phi\colon X\to\P^2$.  The $(-1)$-classes $C$ such that $CH=0$ correspond to $(-1)$-curves on $X$ which are contracted by~$\phi$. The number of such classes equals to $\rank(\Pic X)-\rank(\Pic \P^2)$, hence the statement follows. 
\end{proof}

\begin{lemma}
\label{lemma_cscs0}
Let $X$ be a rational surface of degree $3\le d\le 7$, then for any $0$-class $S$ and $(-1)$-class $C$ one has $CS\ge 0$. If $S',S''$ are  $0$-classes on $X$ such that $S'S''=1$, then
$$|\{C\in I(X)\mid CS'=CS''=0\}|=8-d.$$
\end{lemma}
\begin{proof}
As in the above proof, we may assume that $X$ is a del Pezzo surface.
By Lemma~\ref{lemma_RS}, $H\cdot C\ge 0$ for any $(-1)$-curve $C$. Again, as in the proof of Lemma~\ref{lemma_ch0}, we check that the linear systems $|S|$ 
and $|S'|$ define regular maps $\phi',\phi''\colon X\to\P^1$. The product 
$$\phi'\times \phi''\colon X\to \P^1\times \P^1$$ 
is a birational morphism since $S'S''=1$. Then $(-1)$-classes $C$ such that $CS'=CS''=0$ correspond to $(-1)$-curves on $X$ which are contracted by $\phi'\times \phi''$. Their number equals  $\rank(\Pic X)-\rank(\Pic (\P^1\times \P^1))=8-d$.
\end{proof}

\section{Exceptional collections and toric systems}

Recall that an object $\EE$  in the derived category $\D^b(\coh(X))$ of coherent sheaves on $X$ is said to be \emph{exceptional} if $\Hom^i(\EE,\EE)=0$ for
 all $i\ne 0$ and  $\Hom(\EE,\EE)=\k$. An ordered collection $(\EE_1,\ldots,\EE_n)$ of exceptional objects is said to be \emph{exceptional} if $\Hom^i(\EE_l,\EE_k)=0$ for
 all $i$ and $k<l$. An exceptional collection $(\EE_1,\ldots,\EE_n)$  is said to be \emph{strong exceptional} if $\Hom^i(\EE_l,\EE_k)=0$ for
 all $i\ne 0$ and all $k,l$. A collection $(\EE_1,\ldots,\EE_n)$ is said to be \emph{full} if objects  $\EE_1,\ldots,\EE_n$ generate $\D^b(\coh X)$ as triangulated category. A collection $(\EE_1,\ldots,\EE_n)$ is said to be \emph{of maximal length } if classes of  objects  $\EE_1,\ldots,\EE_n$ generate the Grothendieck group $K_0(X)$ of $X$ modulo numeric equivalence. Recall that for a rational surface $X$ one has
$$\rank K_0(X)=\rank\Pic(X)+2=12-\deg(X).$$

Any line bundle on a rational surface is exceptional. Therefore an ordered collection of line bundles $(\O_X(D_1),\ldots,\O_X(D_n))$ on a rational surface $X$ is exceptional (resp. strong exceptional) if and only if the divisor $D_l-D_k$ is left-orthogonal (resp. strong left-orthogonal) for all $k<l$.

Clearly, any full exceptional collection has maximal length. In general, the converse in not true: there are examples of surfaces of general type (classical Godeaux surface, \cite{BBS} or Barlow surface, \cite{BBKS}) possessing an exceptional collection of maximal length which is not full. But for rational surfaces there are no such examples known. For weak del Pezzo surfaces of degree $\ge 2$, it follows from a result by Sergey Kuleshov \cite[Theorem 3.1.8]{Ku} that any exceptional collection of vector bundles of maximal length is full.

Next we recall the important notion of a \emph{toric system}, introduced by Hille and Perling in \cite{HP}.

For a sequence $(\O_X(D_1),\ldots,\O_X(D_n))$ of line bundles one can consider 
the infinite sequence (called a \emph{helix}) $(\O_X(D_i)), i\in \Z$, defined by the rule
$D_{k+n}=D_k-K_X$. From Serre duality it follows that the collection $(\O_X(D_1),\ldots,\O_X(D_n))$ is exceptional (resp. numerically exceptional) if and only if any collection of the form $(\O_X(D_{k+1}),\ldots,\O_X(D_{k+n}))$ is exceptional (resp. numerically exceptional). One can consider the $n$-periodic sequence $A_k=D_{k+1}-D_k$ of divisors on $X$. Following Hille and Perling, we will consider the finite sequence $(A_1,\ldots,A_n)$ with the cyclic order and will treat the index $k$ in $A_k$ as a residue modulo $n$.  Vice versa, for any sequence $(A_1,\ldots,A_n)$ one can construct the infinite sequence $(\O_X(D_i)), i\in \Z$, with the property $D_{k+1}-D_k=A_{k\mod n}$.

\begin{definition}[See {\cite[Definitions 3.4 and 2.6]{HP}}]
\label{def_ts}
A sequence $(A_1,\ldots,A_n)$ in $\Pic(X)$ is called a \emph{toric system} if $n=\rank K_0(X)$ and the following conditions are satisfied (where indexes are treated modulo $n$):
\begin{itemize}
\item $A_iA_{i+1}=1$;
\item $A_iA_{j}=0$ if $j\ne i,i\pm 1$;
\item $A_1+\ldots+A_n=-K_X$.
\end{itemize}
\end{definition}

Note that a cyclic shift $(A_k,A_{k+1},\ldots,A_n,A_1,\ldots,A_{k-1})$ of a toric system $(A_1,\ldots,A_n)$  is also a toric system. Also, note that by our definition any toric system has maximal length.

\begin{example}
Let $Y$ be a smooth projective toric surface. Its torus-invariant prime divisors form a cycle, denote them $T_1,\ldots,T_n$ in the cyclic order. Then $(T_1,\ldots,T_n)$ is a toric system on~$Y$.
\end{example}

The following is proved in \cite[Lemma 3.3]{HP}, see also \cite[Propositions 2.8 and 2.15]{EL}.
\begin{predl}
A sequence $(A_1,\ldots,A_n)$ in $\Pic(X)$ is a toric system if and only if $n=\rank K_0(X)$ and the corresponding collection $(\O_X(D_1),\ldots,\O_X(D_n))$ is numerically exceptional.
\end{predl}

For the future use we make the following remark, see the proof of Proposition 2.7 in~\cite{HP}.
\begin{remark}
\label{remark_genpic}
For any toric system $(A_1,\ldots,A_n)$ the elements $A_1,\ldots,A_n$ generate $\Pic X$ as abelian group.
\end{remark}

A toric system $(A_1,\ldots,A_n)$ is called \emph{exceptional} (resp. \emph{strong exceptional}) if the corresponding  collection $(\O_X(D_1),\ldots,\O_X(D_n))$ is  exceptional (resp. strong exceptional). Note that exceptional toric systems are stable under cyclic shifts while strong exceptional toric systems are not in general.

A toric system $(A_1,\ldots,A_n)$ is called \emph{cyclic strong exceptional} if the collection 
$$(\O_X(D_{k+1}),\ldots,\O_X(D_{k+n}))$$ 
is strong exceptional for any $k\in \Z$. Equivalently: if all cyclic shifts $$(A_k,A_{k+1},\ldots,A_n,A_1,\ldots,A_{k-1})$$
are strong exceptional.

Notation: for a toric system $(A_1,\ldots,A_n)$ denote 
$$A_{k,k+1,\ldots,l}=A_k+A_{k+1}+\ldots+A_l.$$
We allow $k>l$ and treat 
$[k,k+1,\ldots,l]\subset [1,\ldots, n]$ as a cyclic segment.
Note that $A_{k\ldots l}$ is a numerically left-orthogonal divisor with 
\begin{equation}
\label{eq_Akl2}
A_{k\ldots l}^2+2=\sum_{i=k}^l(A_i^2+2).
\end{equation}

\begin{remark}
If for a toric system $A$ one has $A_i^2\ge -2$ for all $i$, then one has $A_{k,\ldots,l}^2\ge -2$ for any cyclic segment $[k,k+1,\ldots,l]\subset [1,\ldots, n]$. 
\end{remark}

The next proposition is a straightforward consequence of definitions.
\begin{predl}
\label{prop_straightforward}
\begin{enumerate}
\item
A toric system $(A_1,\ldots,A_n)$ is  exceptional if and only if the divisor $A_{k\ldots l}$ is left-orthogonal for all $1\le k<l\le n-1$ and if and only if the divisor $A_{k\ldots l}$ is left-orthogonal for all $k,l\in [1\ldots n], l\ne k-1$.
\item
A toric system $(A_1,\ldots,A_n)$ is strong  exceptional if and only if the divisor $A_{k\ldots l}$ is strong left-orthogonal for all $1\le k<l\le n-1$.
\item
A toric system $(A_1,\ldots,A_n)$ is  cyclic strong exceptional if and only if the divisor $A_{k\ldots l}$ is strong  left-orthogonal for all $k,l\in [1\ldots n], l\ne k-1$.
\end{enumerate}
\end{predl}


\begin{theorem}
\label{theorem_checkonlyminustwo}
Let $A=(A_1,\ldots,A_n)$ be a toric system on a weak del Pezzo surface $X$. Suppose $A^2_i\ge -2$ for  $1\le i\le n-1$. 
Then: 
\begin{enumerate}
\item
$A$ is exceptional if and only if the following holds: 
for any cyclic segment $[k\ldots l]\subset [1\ldots n]$ such that one of the next conditions holds:
\begin{enumerate}
\item $1\le k\le l\le n-1$ and  $A_{k\ldots l}^2=-2$,
\item $k>l$ and $A_{k\ldots n\ldots l}^2=A_n^2\le -2$
\end{enumerate}
the divisor $A_{k\ldots l}$ is left-orthogonal (or equivalently: the divisor $A_{k\ldots l}$ is not anti-effective).
\item
$A$ is strong exceptional if and only if $A$ is exceptional and the following holds: for any $1\le k\le l\le n-1$ such that $A_{k\ldots l}^2=-2$, the divisor $A_{k\ldots l}$ is strong left-orthogonal (or equivalently: the divisor $A_{k\ldots l}$ is not effective nor anti-effective).
\end{enumerate}

Suppose moreover that $A^2_i\ge -2$ for  all $i$. Then
\begin{enumerate}
\item[(3)]
$A$ is exceptional if and only if the following holds: for any cyclic segment $[k\ldots l]\subset [1\ldots n]$ such that $A_{k\ldots l}^2=-2$, the divisor $A_{k\ldots l}$ is left-orthogonal (or equivalently: the divisor $A_{k\ldots l}$ is not anti-effective).
\item[(4)]
$A$ is cyclic strong exceptional if and only if the following holds: for any cyclic segment $[k\ldots l]\subset [1\ldots n]$ such that $A_{k\ldots l}^2=-2$, the divisor $A_{k\ldots l}$ is strong left-orthogonal (or equivalently: the divisor $A_{k\ldots l}$ is not effective nor anti-effective).
\end{enumerate}
\end{theorem} 
\begin{proof}


(1) ``Only if'' is trivial,  so we check ``if''. By Proposition~\ref{prop_straightforward},  $A$ is exceptional iff any divisor $D=A_{k\ldots l}$ is lo. We claim that $A$ is exceptional iff for any $D=A_{k\ldots l}$ with $D^2\le -2$ one has $h^0(-D)=0$. 
Denote $D'=-K_X-D=A_{l+1,\ldots,k-1}$, then  $D^2+(D')^2=d-4$ by Lemma~\ref{lemma_dprime}. By Serre duality, $D$ is lo iff $D'$ is lo. If $D^2<d-2$, then $D$ is lo by Proposition~\ref{prop_loslo} because $h^0(-D)=0$. If $D^2\ge d-2$ then $(D')^2\le -2$ and $D'$ is lo by Proposition~\ref{prop_loslo} because $h^0(-D')=0$.

Next, we observe that if $1\le k\le l\le n-1$ and $D^2\le -2$ then $D^2=-2$. 
Finally, suppose $D=A_{k\ldots n\ldots l}$ and $D^2<-2$. Recall that $D^2+2=\sum_{i=k}^l(A_i^2+2)$. One can decompose $D$ as
$$D=A_{k\ldots n\ldots l}=A_{k\ldots k_1-1}+A_{k_1\ldots n\ldots l_1}+A_{l_1+1\ldots l},$$
where $A_{k_1\ldots n\ldots l_1}^2=A_n^2$ and $A_{k\ldots k_1-1}^2,A_{l_1+1\ldots l}^2\ge -1$. Note that $A_{k\ldots k_1-1}^2,A_{l_1+1\ldots l}^2$ are effective by Lemma~\ref{lemma_eff} (they also can be zero). Therefore it suffices only to check $h^0(-D)=0$ for $D$ of the form $A_{k\ldots n\ldots l}$ where $D^2=A_n^2$ is the minimal possible.

(2) Suppose $1\le k\le l\le n-1$,  then by (1) $D=A_{k\ldots l}$ is lo. If $D^2=-2$,  then $D$ is slo by assumptions, if $D^2>-2$ then $D$ is slo by Proposition~\ref{prop_loslo}.

(3) follows from (1), (4) is similar to (2).
\end{proof}

\begin{corollary}
Let $A=(A_1,\ldots,A_n)$ be a strong exceptional toric system with $A_n^2\ge -1$. Then $A$ is cyclic strong exceptional.
\end{corollary}
\begin{proof}
Use Theorem~\ref{theorem_checkonlyminustwo}. Clearly, any cyclic segment $[k,\ldots,l]\subset [1,\ldots,n]$ such that $A_{k,\ldots,l}=-2$ lies in $[1,\ldots,n-1]$. Therefore $A_{k,\ldots,l}$ is slo by assumptions.
\end{proof}

\begin{corollary}
Any toric system $(A_1,\ldots,A_n)$  with $A_i^2\ge -2$ for all $i$ on a del Pezzo surface is cyclic strong exceptional. 
\end{corollary}
\begin{proof}
There are no effective and anti-effective $(-2)$-classes on a del Pezzo surface.
\end{proof}

We finish this section with a very important theorem which is due to Hille and Perling \cite[Theorem 3.5]{HP} for rational surfaces  and to Charles Vial  \cite[theorem 3.5]{Vi} for the general case.
\begin{theorem}
\label{theorem_HP}
Let $A=(A_1,\ldots,A_n)$ be a toric system of maximal length on a smooth projective surface $X$ such that $\chi(\O_X)=1$ ($X$ is not necessarily rational). Then there exists a smooth projective toric surface $Y$ with torus-invariant divisors $T_1,\ldots,T_n$ such that $A_i^2=T_i^2$ for all $i$.
\end{theorem}

\section{Admissible sequences}
\label{section_adm}

For a sequence $(a_1,\ldots,a_n)$ of integers we define the \emph{$m$-th elementary augmentation} as follows:

\begin{itemize}
\item ${\rm augm}_1(a_1,\ldots,a_n) = (-1,a_1-1,a_2,\ldots,a_{n-1},a_n-1)$;
\item ${\rm augm}_m(a_1,\ldots,a_n)=(a_1,\ldots,a_{m-2},a_{m-1}-1,-1,a_m-1,a_{m+1},\ldots,a_n)$ for $2\le m\le n$; 
\item ${\rm augm}_{n+1}(a_1,\ldots,a_n)=(a_1-1,a_2,a_3,\ldots,a_n-1,-1)$.
\end{itemize}

We call a sequence  \emph{admissible} if it can be obtained from the sequence $(0,k,0,-k)$ or $(k,0,-k,0)$ by applying several elementary augmentation.
It is not hard to see that admissible sequences are stable under cyclic shifts:
$${\rm sh}(a_1,\ldots a_n)=(a_2,\ldots,a_n,a_1)$$
and symmetries:
$${\rm sym}(a_1,\ldots a_n)=(a_{n-1},a_{n-2},\ldots,a_1,a_n).$$
Also note that for an admissible sequence $(a_1,\ldots,a_n)$ one has 
$$\sum_{i=1}^na_i=12-3n.$$
We call a sequence $a_1,\ldots,a_n$ of integers \emph{strong admissible} if it is admissible and $a_i\ge -2$ for $1\le i\le n-1$. 
We call a sequence $a_1,\ldots,a_n$ of integers \emph{cyclic strong admissible} if it is admissible and $a_i\ge -2$ for $1\le i\le n$. 
Note that cyclic strong admissible sequences are stable under cyclic shifts while strong admissible are not. Strong admissible and cyclic strong admissible sequences are stable under symmetries.

For a toric system $A=(A_1,\ldots,A_n)$ on $X$, we denote 
$$A^2=(A_1^2,\ldots,A_n^2).$$

The motivation for considering admissible sequences is the following. Let $Y$ be a toric surface with torus invariant divisors $T_1,\ldots,T_n$. Then the sequence $T^2=(T_1^2,\ldots,T_n^2)$ is admissible. Indeed, for $Y$ a Hirzebruch surface one has $n=4$ and the statement is clear. Otherwise $Y$ has a torus-invariant $(-1)$-curve $E$. Let $E=T_k$, and consider the blow-down $Y'$ of $E$. The torus-invariant divisors on $Y'$ are $T'_1,\ldots,T'_{k-1},T'_{k+1},\ldots,T'_n$, and $(T'_{k\pm 1})^2=T_{k\pm 1}^2+1$, $(T'_{i})^2=T_{i}^2$ otherwise. Therefore $T^2={\rm augm}_k((T')^2)$, and we proceed by induction.

Further, from Theorem~\ref{theorem_HP} it follows now that for any toric system $A$ the sequence $A^2$ is admissible. Suppose that $A$ is a strong exceptional toric system, then
for $1\le i\le n-1$ the divisor $A_i$ is slo, hence $A_i^2=\chi(A_i)-2\ge -2$ and the sequence $A^2$ is strong admissible. The same holds for cyclic strong. 

The next Proposition can be checked directly or deduced from Table 1 in \cite{HP}.

\begin{predl}
\label{prop_csadm}
All cyclic strong admissible sequences are, up to cyclic shifts and symmetries, in  Table~\ref{table_csadm}.
\begin{table}[h]
\caption{Cyclic strong admissible sequences}
\begin{center}
\begin{tabular}{|c|c|}
		 \hline
			type & sequence\\
		 \hline
			$\P^1\times \P^1$ & $(0,0,0,0)$ \\
		 \hline	
			$\mathbb F_1$ & $(0,1,0,-1)$ \\
		 \hline	
			$\mathbb F_2$ & $(0,2,0,-2)$ \\
		 \hline	
			5a & $(0,0,-1,-1,-1)$ \\
		 \hline	
			5b & $(0,-2,-1,-1,1)$ \\
		 \hline	
			6a & $(-1,-1,-1,-1,-1,-1)$ \\
		 \hline	
			6b & $(-1,-1,-2,-1,-1,0)$ \\
		 \hline	
			6c & $(-2,-1,-2,-1,0,0)$ \\
		 \hline	
			6d & $(-2,-1,-2,-2,0,1)$ \\
		 \hline	
			7a & $(-1,-1,-2,-1,-2,-1,-1)$ \\
		 \hline	
			7b & $(-2,-1,-2,-2,-1,-1,0)$ \\
		 \hline	
			8a & $(-2,-1,-2,-1,-2,-1,-2,-1)$ \\
		 \hline	
			8b & $(-2,-1,-1,-2,-1,-2,-2,-1)$ \\
		 \hline	
			8c & $(-2,-1,-2,-2,-2,-1,-2,0)$ \\
		 \hline	
			9 & $(-2,-2,-1,-2,-2,-1,-2,-2,-1)$ \\
		 \hline	
\end{tabular}
\end{center}
\label{table_csadm}
\end{table}

In particular, if a surface $X$ has a toric system $A=(A_1,\ldots,A_n)$ with $A_i^2\ge -2$ for all~$i$ then $n\le 9$ and $\deg(X)\ge 3$.
\end{predl}

\begin{predl}
\label{prop_ncsadm}
Any strong admissible sequence $(a_1,\ldots,a_n)$ which is not cyclic strong admissible is, up to a symmetry, in Table~\ref{table_ncsa}.
\begin{table}[h]
\caption{Strong admissible sequences which are not cyclic strong}
\begin{center}
\begin{tabular}{|c|c|}
\hline
type & sequence\\
\hline
IIa & $(b,c,d,e)$,  $c+e=4-n$; \\
IIb & $(-2,-1,-2,c,d,e)$, $c+e=5-n$; \\
IIc & $(-2,-1,-2,c,-2,-1,-2,e)$, $c+e=6-n$; \\
& where $c,e\in \Z$,  $c\ge -2$, $e\le -3$\\
& and $b$ and $d$ are sequences of the form $(0),(-1,-1)$ or $(-1,-2,-2,\ldots,-2,-1)$; \\
\hline	
IIIa & $(1,0,-1,-2,\ldots,-2,-1,4-n)$; \\
IIIb & $(-1,0,0,-2,\ldots,-2,-1,4-n)$; \\
IIIc & $(-1,-2,\ldots,-2,0,0,-2,\ldots,-2,1,4-n)$; \\
\hline	
IV & $(-2,0,1,-2,\ldots,-2,-1,4-n)$; \\
\hline	
V & $(-2,-1,-1,0,-2,\ldots,-2,-1,5-n)$; \\
\hline	
VI & $(-2,-2,-1,-2,0,\ldots,-2,-1,6-n)$. \\
\hline	
\end{tabular}
\end{center}
\label{table_ncsa}
\end{table}
\end{predl}

\begin{proof}
Let $a$ be an admissible sequence and $a'$ be its elementary augmentation which is strong admissible. Then $a$ is also strong admissible. Hence one has to check the following, see Table~\ref{table_ncsa}:
\begin{itemize}
\item Strong admissible sequences $(0,k,0,-k)$ with $k\ge 3$ are  in Table~\ref{table_ncsa}, type IIa;
\item If $a$ is cyclic strong admissible then $a'$ is either cyclic strong admissible  or in Table~\ref{table_ncsa};
\item If $a$ is of type II then $a'$ is of type II;
\item If $a$ is of type III then $a'$ is of types II, III or V;
\item If $a$ is of type IV then $a'$ is of types IV or V;
\item If $a$ is of type V then $a'$ is of type II, V or VI;
\item If $a$ is of type VI then $a'$ is of type VI.
\end{itemize}
\end{proof}

We will use the following terminology. Let $a=(a_1,\ldots,a_n)$ be a strong admissible sequence. We say that $a$ is of \emph{the first kind} if $a_n\ge -2$. We say that $a$ is of \emph{the second kind} if $a_n\le -3$. We refer to the Roman numeral in the first column of Table~\ref{table_ncsa} as to the \emph{type} of a sequence. We say that a toric system $(A_1,\ldots,A_n)$ is of \emph{the first/second kind} or of \emph{some type} if such is true for the sequence $(A_1^2,\ldots,A_n^2)$.
Note that  toric systems of types III-VI can be obtained by augmentations from the sequence $(1,1,1)$.

\section{Permutations, shifts, symmetries}

Let $(\O_X(D_1),\ldots,\O_X(D_n))$ be an exceptional collection. Recall that the line bundles 
$$O_X(D_k),\O_X(D_{k+1}),\ldots,\O_X(D_l)$$ 
\emph{form a block} if they are completely orthogonal to each other: $\Hom^i(\O_X(D_p),\O_X(D_q))=0$ for any $k\le p<q\le l$ and $i$.
One can reorder bundles in a block and get essentially the same exceptional collection. On the side of toric systems, this results in the next operation. 
\begin{definition}
Let $A=(A_1,\ldots,A_n)$ be a toric system. Suppose $A_k^2=-2$ for some $k$. Denote
$${\rm perm}_k(A)=(A_1,\ldots,A_{k-2},A_{k-1}+A_k,-A_k,A_{k}+A_{k+1},A_{k+2},\ldots,A_n).$$
\end{definition}
It is easy to see that ${\rm perm}_k(A)$ is also a toric system, which corresponds to the numerically left-orthogonal  collection 
$$
(\O_X(D_1),\ldots,\O_X(D_{k-1}),\O_X(D_{k+1}),\O_X(D_k),\O_X(D_{k+2}),\ldots,\O_X(D_n)).$$ 
The operation ${\rm perm}_k$ is called $k$-th \emph{permutation}. It is an involution and ${\rm perm}_k,\ldots,{\rm perm}_l$ define an action of the symmetric group $S_{l-k+2}$ on the set of toric systems $A$ satisfying $A_k^2=\ldots=A_l^2=-2$. Note also that
$$({\rm perm}_k(A))^2=A^2.$$
The following useful lemma is straightforward.
\begin{lemma}
\label{lemma_perm}
Suppose $A$ is a toric system and $A_k^2=-2$. 
\begin{enumerate}
\item If $A$ is exceptional then ${\rm perm}_k(A)$ is exceptional if and only if the divisor $A_k$ is strong left-orthogonal.
\item If $A=(A_1,\ldots,A_n)$ is strong exceptional and $k\ne n$ then ${\rm perm}_k(A)$ is also strong exceptional.
\item If $A=(A_1,\ldots,A_n)$ is cyclic strong exceptional  then ${\rm perm}_k(A)$ is also cyclic strong exceptional.
\end{enumerate}
\end{lemma}

We introduce also the two following natural operations:
\begin{enumerate}
\item symmetry: ${\rm sym}((A_1,\ldots,A_n))= (A_{n-1},A_{n-2},\ldots,A_2,A_1,A_n)$;
\item cyclic shift: ${\rm sh}((A_1,\ldots,A_n))= (A_2,A_3,\ldots,A_n,A_1)$.
\end{enumerate}
They correspond to following operations with collections:
\begin{enumerate}
\item symmetry: $(\O_X(D_1),\ldots,\O_X(D_n))\Longrightarrow (\O_X(-D_n),\ldots,\O_X(-D_1))$;
\item cyclic shift: $(\O_X(D_1),\ldots,\O_X(D_n))\Longrightarrow (\O_X(D_2),\ldots,\O_X(D_n),\O_X(D_{n+1}))$.
\end{enumerate}

Note that symmetry preserves both exceptional and strong exceptional toric systems. 
Cyclic shift preserves exceptional toric systems but may not preserve strong exceptional toric systems. Both operations preserve cyclic strong exceptional toric systems.

\section{Augmentations}

Following Hille and Perling \cite{HP}, we define augmentations. They provide a wide class of explicitly constructed toric systems.

Let $A'=(A'_1,\ldots,A'_n)$ be a toric system on a surface $X'$, and let $p\colon X\to X'$ be the blow up of a point with the exceptional divisor $E\subset X$. Denote $A_i=p^*A'_i$. Then one has the following toric systems on $X$:
\begin{align*}
{\rm augm}_{p,1}(A')=&(E,A_1-E,A_2,\ldots, A_{n-1},A_n-E);\\
{\rm augm}_{p,m}(A')=&(A_1,\ldots,A_{m-2}, A_{m-1}-E,E,A_{m}-E,A_{m+1},\ldots,A_n)\quad\text{for}\quad 2\le m\le n;\\
{\rm augm}_{p,n+1}(A')=&(A_1-E,A_2,\ldots, A_{n-1},A_n-E,E).
\end{align*}
Toric systems ${\rm augm}_{p,m}(A')$ ($1\le m\le n+1$) are called   \emph{elementary augmentations} of toric system  $(A'_1,\ldots,A'_n)$.

\begin{predl}[{See \cite[Proposition 3.3]{EL}}]
\label{prop_elemaugm}
In the above notation, let $A$ be a toric system on $X$ such that $A_m=E$ for some $m$. Then $A={\rm augm}_{p,m}(A')$ for some toric system $A'$ on $X'$.
\end{predl}

\begin{definition}
\label{def_augm1}
A toric system $A$ on $X$ is called a \emph{standard augmentation} if $X$ is a Hirzebruch surface or $A$ is an elementary augmentation of some standard augmentation. Equivalently: $A$ is a standard augmentation if there exists a chain of blow-ups
$$X=X_n\xra{p_n} X_{n-1}\to \ldots X_1\xra{p_1} X_0$$
where $X_0$ is the Hirzebruch surface and 
$$A={\rm augm}_{p_n,k_n}({\rm augm}_{p_{n-1},k_{n-1}}(\ldots {\rm augm}_{p_1,k_1}(A')\ldots ))$$
for some $k_1,\ldots,k_n$ and a toric system $A'$ on $X_0$. In this case we will say that $A$ is a \emph{standard augmentation along the chain $p_1,\ldots,p_n$}.
\end{definition}

\begin{remark}
To be more accurate, one  should add that (the unique) toric system on $\P^2$ is also considered as a standard augmentation. To simplify the forthcoming definitions and statements, we will ignore this issue.
\end{remark}

\begin{predl}[{See \cite{EL}, Proposition 2.21}]
\label{prop_augmexc}
Let $A={\rm augm}_k(A')$. Then
\begin{enumerate}
\item $A$ is exceptional if and only if $A'$ is exceptional;
\item if $A$ is strong exceptional then $A'$ is strong exceptional;
\item if $A$ is cyclic strong exceptional then $A'$ is cyclic strong exceptional.
\end{enumerate}
\end{predl}

\begin{definition}
\label{def_augm2}
A toric system $A$ on $X$ is called an \emph{augmentation in the weak sense} 
if~$A$ can be obtained from a toric system on a Hirzebruch surface by several permutations, cyclic shifts and elementary augmentations  (in any order). 
\end{definition}
\begin{remark}
In this definition, $A$ can be an exceptional or strong exceptional toric system, but the intermediate toric systems are not required to be exceptional.
\end{remark}

\begin{definition}
\label{def_augm4}
An exceptional (resp. strong exceptional, cyclic strong exceptional) toric system $A=(A_1,\ldots,A_n)$ on $X$ is called an \emph{exceptional (resp. strong exceptional, cyclic strong exceptional) augmentation} if $X$ is a Hirzebruch surface or $A$ can be obtained by a permutation, cyclic shift or an elementary augmentation from some exceptional (resp. strong exceptional, cyclic strong exceptional) toric system $B$.
\end{definition}

\begin{remark}
In \cite{HP} a different terminology is used. For an exceptional  augmentation $A$, they say that $A$ has a \emph{normal form} $B$ which is a standard augmentation. This normal form is obtained from $A$ by cyclic shifts and permutations preserving exceptionality.
\end{remark}

\begin{remark}
If a toric system is an augmentation, it is an augmentation in the weak sense. 
\end{remark}

\begin{lemma}
A toric system $A$ is an augmentation in the weak sense if and only if $sym(A)$ is. The same holds for standard, exceptional, strong exceptional and cyclic strong exceptional augmentations. 
\end{lemma}
\begin{proof}
This follows from the following relations on $A=(A_1,\ldots,A_n)$
$${\rm sym}\circ {\rm perm}_k={\rm perm}_{n-k}\circ {\rm sym},\quad {\rm sym}\circ {\rm sh}={\rm sh}^{-1}\circ {\rm sym},\quad {\rm sym}\circ {\rm augm}_k={\rm augm}_{n+1-k}\circ {\rm sym}$$
and the fact that ${\rm sym}$ preserves exceptional, strong exceptional and cyclic strong exceptional toric systems. 
\end{proof}

\begin{lemma}
\label{lemma_cseweak=mild}
A cyclic strong exceptional toric system $A$ is an augmentation in the weak sense if and only if it is a cyclic strong exceptional augmentation. 
\end{lemma}
\begin{proof}
``If'' is trivial. For ``only if'', suppose  $A=t_k\circ\ldots\circ t_1(A')$ where $A'$ is a toric system on a Hirzebruch surface and any $t_i$ is either ${\rm perm}_k, {\rm sh}$ or ${\rm augm}_k$. Recall that $perm$ and ${\rm sh}$ preserve cyclic strong exceptional toric systems. By Proposition~\ref{prop_augmexc}, if ${\rm augm}_k(B)$ is cyclic strong exceptional then so is $B$. Therefore all toric systems  $t_l\circ\ldots\circ t_1(A')$ where $0\le l\le k$ are cyclic strong exceptional and $A$ is a cyclic strong exceptional augmentation by definition. 
\end{proof}

Hille and Perling in their original paper \cite[Theorem 8.1]{HP} proved that any strong exceptional toric system on a toric surface is a strong exceptional augmentation. 
In \cite[Theorem 1.4]{EL} it is proved that any toric system on a del Pezzo surface is a standard augmentation.

\begin{remark}
In \cite{HP} a different terminology is used. For an exceptional  augmentation $A$, they say that $A$ has a \emph{normal form} $B$ which is a standard augmentation. This normal form is obtained from $A$ by cyclic shifts and permutations preserving exceptionality. 
\end{remark}

Here we give an example demonstrating that the use of cyclic shifts and permutations is necessary. In this example one cannot get a ``normal form'' only by reordering of line bundles in the collection.


\begin{example}
Let $X=X_{4,2A_1,8}$ be the weak del Pezzo surface of degree $4$ of type $2A_1$ with $8$ lines. Explicitly, let $P_1,P_2,P_3,P_4\in\P^2$ be points and $H\subset \P^2$ be a line such that $P_1,P_2,P_3\in H$, $P_4\notin H$. Let $X'$ be the blow-up of $P_1,P_2,P_3,P_4$ with exceptional divisors $E_1,E_2,E_3,E_4$. Let $P_5\in E_4$ be a general point and $X$ be the blow-up of $P_5$.  Consider the following toric system on $X$:
$$A=(A_1,\ldots,A_8)=(L-E_{145}, E_4, L-E_{234}, L-E_5, E_5-E_1, L-E_{35}, E_3-E_2, -L+E_{125}).
$$
Clearly, no $A_i$ is an irreducible curve, so $A$ is not an elementary augmentation.
There are $8$ irreducible $(-1)$-classes on $X$: 
\begin{multline*}
E_1=A_{678}, E_2=A_{812}, E_3=A_{7812}, E_5=A_{5678}, L-E_{14}=A_{56781}, L-E_{24}=A_{78123}, \\
L-E_{34}=A_{8123}, L-E_{45}=A_{6781}.
\end{multline*}
Therefore, there are no irreducible $(-1)$-classes of the form $A_{k\ldots l}$ where $1\le k\le l\le 7$. This means that there exist no elementary augmentation $B$ that can be sent to $A$ by permutations ${\rm perm}_1,\ldots,{\rm perm}_7$. On the other hand, let $B={\rm perm}_7{\rm perm}_8(A)$, then $B_6=E_1$ is an irreducible curve and $B$ is an elementary augmentation.

One can check (using Theorem~\ref{theorem_checkonlyminustwo}) that the toric system $A$ is cyclic strong exceptional.
\end{example}


\begin{remark}
\label{remark_full}
It follows from \cite[Propositions 2.19 and 2.21]{EL} that any exceptional augmentation is full. 
Therefore Conjecture~\ref{conj_augm} imply that any strong exceptional collection of line bundles of maximal length is full. 
\end{remark}

\section{Weyl group action}

Recall that $(-2)$-classes in $\Pic(X)$ belong to the orthogonal complement  $N_X=(K_X)^{\perp}\subset \Pic(X)\otimes \R$ and form a root system in some subspace of $N_X$. This subspace is all $N_X$ for $\deg X\le 6$. Orthogonal reflections in the roots generate a group which is called \emph{Weyl group} and will be denoted by $W(X)$. The Weyl group acts on $\Pic(X)$ preserving the intersection form and the canonical class. Therefore $W(X)$ acts on the set of $r$-classes for any $r$ and on the set $TS_a(X)$ of toric systems $A$ on $X$ with the fixed sequence $a=A^2$.
Recall the well-known
\begin{predl}[See {\cite[Theorem IV.1.9, Corollary IV.4.7]{Ma}}]
Suppose $1\le \deg(X)\le 6$. Then  
\begin{enumerate}
\item the group $W(X)$ equals to the whole group of isometries of $\Pic(X)$ preserving $K_X$;
\item the group $W(X)$ acts transitively on $I(X)$.
\end{enumerate}
\end{predl}

\begin{predl}\label{Weyl trans}
Let $X$ be a blow-up of $\P^2$ at $s\le 8$ points (maybe infinitesimal) and $a$ be an admissible sequence. Then the action of $W(X)$ on $TS_a$ is free and transitive.
\end{predl}

\begin{proof}
First we check that the action is free. Let $A$ be a toric system and $w\in W(X)$, and suppose $w(A)=A$. Then $w(A_i)=A_i$ for all $i$. Since $A_1,\ldots,A_n$ generate $\Pic(X)$ (see Remark~\ref{remark_genpic}), it follows that $w=e$.

Now we check that the action is transitive. We argue by induction in $s$. If $X=\P^2$ there is nothing to prove. 

If $X=\mathbb F_1$ then all toric systems on $X$ are (up to a cyclic shift) of the form 
$$A^{(c)}=(L-E_1, L+c(L-E_1),L-E_1, E_1-c(L-E_1)),$$
where $c\in\Z$, see \cite[Proposition 5.2]{HP}.
Note that $(A^{(c)})^2=(0,1+2c,0,-1-2c)$. It follows that $TS_{(0,k,0,-k)}(\mathbb F_1)$ is empty for even $k$ and has one element for odd $k$. Thus the statement is true. 

If $s=2$, then  all admissible toric systems  are (up to a cyclic shift and a symmetry) of the form $a^{(k)}=(-1,-1,k,0,-k-1)$. There are three $(-1)$-classes in $\Pic(X)$: $E_1,E_2,L_{12}$. We have $W\cong\Z/2\Z$. Classes $E_1,E_2$ are in one $W$-orbit and $L_{12}$ in another one. Note that $E_2$ is irreducible, let $f\colon X\to \mathbb F_1$ be the blow-down of $E_2$. Let $A$ be a toric system with $A^2=a^{(k)}$. Suppose $A_1=E_i$, then $A_2=L_{12}$. Up to an action of $W$ we may assume that $A_1=E_2$. Then $A={\rm augm}_{f,1}(A')$ where $A'$ is a toric system on $\mathbb F_1$. Note that $(A')^2=(0,k,0,-k)$, hence $k$ is odd and $A'$ is uniquely determined by $k$. Now suppose $A_1=L_{12}$, then $A_2=E_i$. Up to the action of $W$ we may assume that $A_2=E_2$, then $A={\rm augm}_{f,2}(A')$ where $A'$ is a toric system on $\mathbb F_1$. Note that now we have $(A')^2=(0,k+1,0,-k-1)$, hence $k$ is even and again $A'$ is uniquely determined by $k$. We see that for any $k$ there is only one orbit of the $W$-action  on $TS_{a^{(k)}}(X)$.

Suppose $s\ge 3$. Let $a=(a_1,\ldots,a_n)$ be an admissible sequence. Fix $m$ such that $a_m=-1$ and fix some irreducible $(-1)$-curve $E\subset X$. Let $f\colon X\to X'$ be the blow-down of $E$. Let $A,B$ be two toric systems with $A^2=B^2=a$. Since the action of $W$ on $I(X)$ is transitive, there exists $w_A,w_B\in W$ such that $w_A(A_m)=w_B(B_m)=E$. Hence $w_A(A)={\rm augm}_{f,m}(A')$ and  $w_B(B)={\rm augm}_{f,m}(B')$ for some toric systems $A',B'$ on $X'$. Clearly, $(A')^2=(B')^2$. Therefore by the induction hypothesis there exists $w\in W(X')$ such that $w(A')=B'$. Embedding $f^*\colon \Pic(X')\to \Pic(X)$ induces the embedding $W(X')\to W(X)$; we will denote the image of $w$ in $W(X)$ also by $w$. Now we have 
$$ww_A(A)=w({\rm augm}_{f,m}(A'))={\rm augm}_{f,m}(w(A'))={\rm augm}_{f,m}(B')=w_B(B)$$
because $w(E)=E$. It follows that $A$ and $B$ are in one $W(X)$-orbit. 
\end{proof}

\begin{remark}
The same should be true for any rational surface of degree $\ge 1$.
\end{remark}

\begin{corollary}
Let $X$ be a blow-up of $\P^2$ at $s\le 8$ points (maybe infinitesimal) and $A$ be a toric system on $X$. Then there exists $w\in W(X)$ such that $w(A)$ is  a standard augmentation.
\end{corollary}

\section{Toric systems of the first kind}
Below we prove that any toric system $A=(A_1,\ldots,A_n)$ of the first kind on a rational surface $X$ is an augmentation in the weak sense. Moreover, if $A$ is exceptional/strong  exceptional/cyclic strong exceptional then $A$ is an exceptional/strong  exceptional/cyclic strong exceptional augmentation in the sense of Definition~\ref{def_augm4}.

The proof is by induction in $n$. To do one step, it suffices to find a toric system $B$ obtained from $A$ by several permutations, such that $B$ is an elementary augmentation. For any toric system $A$ on a surface $X$, denote by $I(X,A)\subset I(X)$ the subset of $(-1)$-classes which are elements of toric systems equivalent to $A$:
$$I(X,A)=\{D\in I(X)\mid \exists B={\rm perm}_{k_r}\ldots {\rm perm}_{k_1}(A) \,\, \exists i\, D=B_i\}.$$
\begin{lemma}
\label{lemma_IFdef}
In the above notation one has
\begin{equation*}
\label{eq_IF}
I(X,A)=\{A_{k\ldots l}\mid \exists m,\, k\le m\le l,\, a_k=\ldots=a_{m-1}=a_{m+1}=\ldots=a_l=-2,\,  a_m=-1\},
\end{equation*}
where $a_i=A_i^2$. 
\end{lemma}
\begin{proof}
First, suppose $a_k=\ldots=a_{m-1}=a_{m+1}=\ldots=a_l=-2,\,  a_m=-1$. Then $A_{k\ldots l}^2=-1$ by (\ref{eq_Akl2}). Also, the toric system
$$B={\rm perm}_{m+1}\ldots {\rm perm}_{l-1}{\rm perm}_l ({\rm perm}_{m-1}\ldots {\rm perm}_{k+1}{\rm perm}_k(A))$$ 
is well-defined and $B_m=A_{k\ldots l}$.

Second, suppose for some toric system 
$A'={\rm perm}_{k_r}\ldots {\rm perm}_{k_1}(A)$,
$A'_m$ is a $(-1)$-class. Let $(\O_X(D_1),\ldots,\O_X(D_n))$ and $(\O_X(D'_1),\ldots,\O_X(D'_n))$ be the corresponding exceptional collections.
The sets $\{D_1,\ldots,D_n\}$ and $\{D'_1,\ldots,D'_n\}$ are the same, assume 
$D'_m=D_k$ and $D'_{m+1}=D_{l+1}$. Hence $A'_m=D'_{m+1}-D'_m=D_{l+1}-D_k=A_{k\ldots l}$. Further, the line bundles $D_k,D_{k+1},\ldots,D_{m}$ are in one block, and 
$D_{m+1},D_{m+2},\ldots,D_{l+1}$ are in another block. It follows that $a_k=a_{k+1}=\ldots=a_{m-1}=a_{m+1}=\ldots=a_{l}=-2$. Also $A_m^2=(A'_m)^2=-1$.
\end{proof}

The following lemma is very encouraging.
\begin{lemma}
\label{lemma_IF}
Let $A$ be a toric system of the first kind on a rational surface $X$ of degree $d\le 7$.
Then $I(X,A)=I(X)$.
\end{lemma}
\begin{proof}
By Lemma~\ref{lemma_IFdef}, the set  $I(X,A)$ is  determined by the sequence $A^2$. For any admissible sequence of the first kind we find the cardinality of $I(X,A)$ and check that it equals to the cardinality of $I(X)$, see Table~\ref{table_root}. Table~\ref{table_ruby} presents $I(X,A)$ for any admissible sequence $A^2$ up to a shift. 
Here we use D-notation: let
$$D_{k,l}=A_{k\ldots l-1}=A_k+\ldots+A_{l-1}.$$
Equivalently, if $(\O(D_1),\ldots,\O(D_n))$ is a numerically  exceptional collection  associated with the toric system $(A_1,\ldots,A_n)$ then $D_{kl}=D_l-D_k$.

\begin{table}[h]
\caption{$I(X,A)$ for toric systems of the first type}
\label{table_ruby}
\begin{center}
\begin{tabular}{|c|c|c|c|}
		 \hline
			type & $A^2$ & $I(X,A)$ & $|I(X)|$\\
		 \hline	
			5a & $(0,0,-1,-1,-1)$ & $A_3,A_4,A_5$ & $3$\\
		 \hline	
			5b & $(0,-2,-1,-1,1)$ & $A_{23},A_3,A_4$ & $3$\\
		 \hline	
			6a & $(-1,-1,-1,-1,-1,-1)$ & $A_1,A_2,A_3,A_4,A_5,A_6$ & $6$\\
		 \hline	
			6b & $(-1,-1,-2,-1,-1,0)$ & $A_1,A_2,A_{23},A_{34},A_4,A_5$ & $6$\\
		 \hline	
			6c & $(-2,-1,-2,-1,0,0)$ & $D_{13},D_{14},D_{23},D_{24},D_{35},D_{45}$ & $6$\\
		 \hline	
			6d & $(-2,-1,-2,-2,0,1)$ & $D_{13},D_{14},D_{15},D_{23},D_{24},D_{25}$ & $6$\\
		 \hline	
			7a & $(-1,-1,-2,-1,-2,-1,-1)$ & $A_1,D_{23},D_{24},D_{35},D_{45},D_{36},D_{46},D_{57},D_{67},A_7$ & $10$ \\
		 \hline	
			7b & $(-2,-1,-2,-2,-1,-1,0)$ & $D_{13},D_{14},D_{15},D_{23},D_{24},D_{25},D_{36},D_{46},D_{56},D_{67}$ & $10$\\
		 \hline	
			8a & $(-2,-1,-2,-1,-2,-1,-2,-1)$ & $D_{2k-1,2k+1},D_{2k,2k+1},D_{2k-1,2k+2},$ & $16$\\
			&& $D_{2k,2k+2}$ for $k=1\ldots 4$ & \\
		 \hline	
			8b & $(-2,-1,-1,-2,-1,-2,-2,-1)$ & $D_{13},D_{23},D_{34},D_{35},D_{46},D_{47},D_{48},D_{56},D_{57},D_{58},$ & $16$\\
			& & $D_{61},D_{62},D_{71},D_{72},D_{81},D_{82}$ & \\
		 \hline	
			8c & $(-2,-1,-2,-2,-2,-1,-2,0)$  & $D_{13},D_{14},D_{15},D_{16},D_{23},D_{24},D_{25},D_{26},$ & $16$\\
			& & $D_{37},D_{47},D_{57},D_{67},D_{38},D_{48},D_{58},D_{68}$ & \\
		 \hline	
			9 & $(-2,-2,-1,-2,-2,-1,-2,-2,-1)$ & $D_{ij}$ for $i\in\{1,2,3\},j\in\{4,5,6\}$,& $27$ \\
			&& $D_{ij}$ for $i\in\{4,5,6\},j\in\{7,8,9\}$, & \\
			&& $D_{ij}$ for $i\in\{7,8,9\},j\in\{1,2,3\}$ & \\
		 \hline	
\end{tabular}
\end{center}
\end{table}

\end{proof}

\begin{theorem}
\label{theorem_1kindweak}
Any toric system $A$ of the first kind on a smooth rational projective surface~$X$ is an augmentation in the weak sense.
\end{theorem}
\begin{proof}
By Lemma~\ref{lemma_IF}, $I(X,A)=I(X)$. If $X$ is not a Hirzebruch surface or $\P^2$ (which is a trivial case) then $X$ has an exceptional curve $E$. The class of $E$ belongs to $I(X,A)$, let $E=A_{k\ldots l}$. 
By definition of $I(X,A)$, there exists a toric system of the form $$B={\rm perm}_{k_r}\ldots {\rm perm}_{k_1}(A),$$ 
such that  $B_m=E$ for some $m$. Therefore $B={\rm augm}_m(C)$ for  some toric system $C$ on the blow down of $E$. Clearly, $B^2=A^2$ and $C$ is also of the first kind. By induction in $\rank \Pic(X)$ we may assume that $C$ is an augmentation in the weak sense, hence $B$ and $A$ are also such.
\end{proof}

\begin{remark}
\label{remark_anychain}
Actually we have proved that a toric system of the first kind on~$X$ is an augmentation in the weak sense along ANY given chain of blow-downs of $X$ to a Hirzebruch surface.
\end{remark}

\begin{corollary}
\label{cor_cyclicaugm}
Any cyclic strong exceptional toric system $A$ on a rational surface~$X$ is a cyclic strong exceptional augmentation along any given chain of blow-downs of $X$ to a Hirzebruch surface. 
\end{corollary}
\begin{proof}
Any cyclic strong exceptional toric system is of the first kind because $A_i^2=\chi(A_i)-2=h^0(A_i)-2\ge -2$ for all $i$. By Theorem~\ref{theorem_1kindweak}, $A$ is an augmentation in the weak sense. From Lemma~\ref{lemma_cseweak=mild} it follows that $A$ is a cyclic strong exceptional augmentation. It is clear that~$A$ is a cyclic strong exceptional augmentation along any chain of blow-ups, see the proof of Theorem~\ref{theorem_1kindweak}.
\end{proof}

\begin{corollary}
Any cyclic strong exceptional collection of line bundles of maximal length on a rational surface is full.
\end{corollary}
\begin{proof}
It follows from Corollary~\ref{cor_cyclicaugm} and Remark~\ref{remark_full}.
\end{proof}

\begin{corollary}
\label{cor_cyclicblowup}
Suppose $X$ is the blow-up of a point on $X'$ and $X$ possesses  a cyclic strong exceptional toric system. Then $X'$ also possesses such system.
\end{corollary}
\begin{proof}
Let $A$ be a cyclic strong exceptional toric system on $X$ and $E$ be the exceptional divisor of the blow-up. By the proof of Theorem~\ref{theorem_1kindweak}, there exists a cyclic strong exceptional toric system $B$ on $X$  such that $B_i=E$ for some $i$. 
By Proposition~\ref{prop_elemaugm}, $B={\rm augm}_i(B')$ for some toric system $B'$ on $X'$. By Proposition~\ref{prop_augmexc}, $B'$ is also cyclic strong exceptional. 
\end{proof}

\begin{corollary}
\label{corollary_f012}
A surface $X$ possessing a cyclic strong exceptional collection cannot be blown down to $\mathbb F_d$ for $d>2$.
\end{corollary}
\begin{proof}
Indeed, Hirzebruch surfaces $\mathbb F_d$ do not have cyclic strong exceptional collections for $d>2$, see \cite[Proposition 5.2]{HP} or Section~\ref{section_cyclicstrong}.
\end{proof}

\begin{theorem}
\label{theorem_1kind}
\begin{enumerate}
\item 
Any exceptional toric system $A$ of the first kind on a weak del Pezzo surface $X$ is an exceptional augmentation. 
\item
Any strong exceptional toric system $A$ of the first kind on a weak del Pezzo surface $X$ is a strong exceptional augmentation. 
\end{enumerate}
\end{theorem}
\begin{proof}
(1) In fact, all intermediate toric systems from the proof of Theorem~\ref{theorem_1kindweak} are also exceptional. We use notation from that proof. To make the arguments more clear, we will prove by induction in $\rank\Pic (X)$ that $A$ is an exceptional augmentation. For any fixed $\rank\Pic (X)$ we argue by induction in $l-k$ where $A_{k\ldots l}$ is an irreducible $(-1)$-class. For $l=k$ we have that $A_k=E$ is irreducible, $A={\rm augm}_k(C)$ and we may use induction in $\rank\Pic(X)$. For $l>k$ we may assume without loss of generality that $A_m^2=-1$ where $k\le m<l$. The $(-2)$-class $A_l$ is not effective. Indeed, otherwise $E=A_{k\ldots l}=A_{k\ldots l-1}+A_l$ and thus is reducible, because $A_{k\ldots l-1}$ is a $(-1)$-class and hence is effective by Lemma~\ref{lemma_eff}. Therefore $A_l\in R^{\rm lo}\setminus R^{\rm eff}=R^{\rm slo}$. It follows (see Lemma~\ref{lemma_perm}(1)) that $A'={\rm perm}_l(A)$ is also an exceptional toric system of the first kind and 
$(A')_{k\ldots l-1}=A_{k\ldots l}$. By the induction hypothesis, $A'$ is an exceptional augmentation.

(2) The above proof has to be modified to work in the strong exceptional setting.
The only subtle point is the following: the toric system $A'$ in the above proof may be not strong exceptional. If $l\ne n$ then $A_l$ is slo, ${\rm perm}_l(A)$ is strong exceptional by Lemma~\ref{lemma_perm}(2) and there are no problems. Now suppose $l=n$. We claim that the system $sh(A)=(A_2,\ldots,A_n,A_1)$ is  strong exceptional (recall that $sh(A)$ is automatically exceptional). Assume the contrary. Then there exists a divisor $D=A_{p\ldots q}$ with $2\le p\le q\le n$ which is not slo. If $q<n$ then $A_{p\ldots q}$ is slo by the definition of a strong exceptional toric system $A$. Hence, $q=n=l$. If $p\le m$ then $D$ is a lo $r$-class with $-1\le r\le \deg(X)-2$ and therefore is slo by Proposition~\ref{prop_loslo}. Hence, $m+1\le p$. It means that $D$ is a $(-2)$-class which is lo but not slo. Consequently, $D$ is effective. It follows that $A_{k\ldots l}=A_{k\ldots p-1}+D$. Both $D$ and $A_{k\ldots p-1}$ (as a $(-1)$-class) are effective, therefore $A_{k\ldots l}$ is not irreducible, a contradiction. Therefore the toric system
$sh(A)=(A_2,\ldots,A_n,A_1)$ is  strong exceptional. Now the toric system $A''={\rm perm}_{n-1}{\rm sh}(A)={\rm sh}(A')=(A_2,\ldots,A_{n-2},A_{n-1,n},-A_n,A_{1,n})$ is  also strong exceptional (by Lemma~\ref{lemma_perm}(2)) and $(A'')_{k-1,n-2}=A_{k\ldots n}$ is irreducible. Again we proceed by induction in $l-k$. 
\end{proof}

\begin{remark}
One can check that the statement of Theorem~\ref{theorem_1kind}(1) holds for all rational surfaces~$X$. Indeed, $\deg X\ge 3$ since $X$ has a toric system of type $1$. We have $h^2(-K_X)=h^0(2K_X)=0$ by Serre duality since $X$ is rational. By Riemann-Roch formula we have $\chi(-K_X)=h^0(-K_X)-h^1(-K_X)=K_X^2+1\ge 2$. It follows that 
 $h^0(-K_X)\ge 2$ and $-K_X>0$. 
Similar to the proof of Proposition~\ref{prop_loslo} one can check that a divisor $D\in R(X)$  is slo if it is not effective or anti-effective. Therefore all the arguments in the proof of Theorem~\ref{theorem_1kind}(1) will work.
\end{remark}

\section{Toric systems of the second kind}

Next we prove the following Theorem, which is one of the main results of the paper. 
\begin{theorem}
\label{theorem_main}
Let $A$ be an exceptional toric system of the second kind on a weak del Pezzo surface of degree $\ge 3$. Then $A$ is an exceptional 	augmentation. Moreover, if $A$ is strong exceptional then $A$ is a strong exceptional augmentation.
\end{theorem}

\begin{corollary}
\label{cor_main}
Let $A$ be a strong exceptional toric system  on a weak del Pezzo surface of degree $\ge 3$. Then $A$ is a strong exceptional 	augmentation. 
\end{corollary}

\begin{proof}
By Section \ref{section_adm}, the sequence $A^2$ is strong admissible, it is either of the first or of the second kind. If $A^2$ is of the first kind we use Theorem ~\ref{theorem_1kind}(2). If $A^2$ is of the second kind we use Theorem ~\ref{theorem_main}. 
\end{proof}

\begin{corollary}
\label{cor_aug}
Let $X\neq\mathbb{P}^2$ be a weak del Pezzo surface of degree  $\ge 3$ admitting a strong exceptional toric system $A$. Then $X$ can be obtained by blowing up a Hirzebruch surface at most twice.
\end{corollary}
\begin{proof} By Corollary~\ref{cor_main}, $A$ is a strong exceptional augmentation. Then by~\cite[Theorem 5.11]{HP} the result follows.
\end{proof}

\begin{remark}
For weak del Pezzo surfaces of degree $2$ the statement of Theorem~\ref{theorem_main} is false. A counterexample will be given in Section~\ref{section_ce}. 
\end{remark}

Our strategy of proving~\ref{theorem_main} is as follows. The main point is to demonstrate that there exists an irreducible divisor in the set $I(X,A)$ (which in this case is a proper subset of $I(X)$).
Provided that such irreducible divisor exists, we use the following lemma. 

\begin{lemma}
\label{lemma_1stepkind2}
Let $A$ be a toric system of the second kind. Suppose there exists an irreducible divisor $E$ in $I(X,A)$. Then there exists a  toric system $B={\rm perm}_{k_r}\ldots {\rm perm}_{k_1}(A)$ which is an elementary augmentation of some toric system $C$ on the blow-down of $E$. Moreover, if $A$ is exceptional (resp. strong exceptional) then $C,B$ and all toric systems ${\rm perm}_{k_j}\ldots {\rm perm}_{k_1}(A)$ for $1\le j\le r$ are also exceptional (resp. strong exceptional).
\end{lemma}

Note that the toric system $C$ in Lemma~\ref{lemma_1stepkind2} is of the first or of the second kind. Hence we can proceed by induction in $\deg(X)$ and thus prove Theorem~\ref{theorem_main}.

\begin{proof}[Proof of Lemma~\ref{lemma_1stepkind2}]
Divisor $E\in I(X,A)$ is of the form $A_{k\ldots l}$ where $1\le k\le l\le n-1$, because $A_n^2<-2$. The proof is similar to the proof of Theorem~\ref{theorem_1kind}. We argue by induction in $l-k$. Case $k=l$ is trivial. For $l>k$, we may assume that $A_l^2=-2$. Then $A_l$ is not effective: otherwise $E=A_{k\ldots l}=A_{k\ldots l-1}+A_l$ is reducible as a sum of $-1$ and $(-2)$-classes. Hence $A_l\in R^{\rm lo}\setminus R^{\rm eff}=R^{\rm slo}$. It follows that the toric system $A'={\rm perm}_l(A)$ is exceptional (resp. strong exceptional), see Lemma~\ref{lemma_perm}. Note that $A'_{k\ldots l-1}=A_{k\ldots l}=E$ so we can use the induction hypothesis.
\end{proof}

Therefore the crucial point in the proof of Theorem~\ref{theorem_main} is to find an irreducible divisor in $I(X,A)$. By Proposition~\ref{prop_main} below, one can always find an irreducible element in $I(X,A)$ for $X$ a weak del Pezzo surface of degree $\ge 3$
and $A$ an exceptional toric system of the second kind.
\begin{predl}
\label{prop_main}
Let $X$ be a weak del Pezzo surface of degree $\ge 3$ and $A$ be a toric system of the second kind. Suppose $I(X,A)\subset I^{\rm red}(X)$. Then there exists a divisor $D=A_{k\ldots n\ldots l}$ such that $-D\ge 0$ and 
$$A_k^2=A_{k+1}^2=\ldots=A_{n-1}^2=A_1^2=\ldots=A_l^2=-2.$$ 
In particular, $D$ is not left-orthogonal and $A$ is not exceptional.
\end{predl}

In the next sections we prove Proposition~\ref{prop_main} for types II-VI of toric systems of the second kind.

\section{Toric systems of the second kind: type II}

Here we prove Proposition~\ref{prop_main} for toric system $A$ of the second kind of type II.

Recall that a toric system $(A_1,\ldots,A_n)$ is of the second kind of type II if $A^2$ is (up to a symmetry) one of the following sequences   
\begin{align*}
&\text{type IIa} & & (b,c,d,e), \quad c+e=4-n; \\
&\text{type IIb} & & (-2,-1,-2,c,d,e), \quad c+e=5-n; \\
&\text{type IIc} & & (-2,-1,-2,c,-2,-1,-2,e), \quad c+e=6-n=-2, 
\end{align*}
where $c,e\in \Z$,  $c\ge -2$, $e\le -3$, and $b$ and $d$ are sequences of the form $(0),(-1,-1)$ or $(-1,-2,-2,\ldots,-2,-1)$.

For a toric system $A$ of type II we define divisors $S(A),D(A)$ and $D'(A)$ as follows.
For $A$ of type IIa we put $S'(A)=A_{1\ldots p}, S''(A)=A_{p+2,\ldots, n-1}$ where $b=(b_1,\ldots,b_p)$ has length~$p$. For $A$ of type IIb we put $S'(A)=A_{1223}, S''(A)=A_{5,\ldots, n-1}$. For $A$ of type IIc we put $S'(A)=A_{1223}, S''(A)=A_{5667}$.
Note that $S'(A)-S''(A)$ has zero intersection with every divisor $A_i$, $1\le i\le n$, which generate the Picard group of $X$. Hence $S'(A)=S''(A)$, we will denote it by $S(A)$.  
Note that $S(A)$ is a $0$-class.

For $A$ of type IIa we put $D(A)=A_{p+1}$ where $b=(b_1,\ldots,b_p)$ has length $p$.  For $A$ of type IIb we put $D(A)=A_4-A_2$. For $A$ of type IIc we put $D(A)=A_4-A_2-A_6$.
From Lemma~\ref{lemma_classes1} it follows that $D(A)$ is an $r$-class where $r=c$ for type IIa, $r=c-1$ for type IIb and $r=c-2$ for type IIc.
Note that $r=4-n-e\ge 7-n=\deg(X)-5$ because $e\le -3$. Also note that $D(A)\cdot S(A)=1$.

Suppose $D^2(A)\ge -1$. Then one can find $m\ge 0$ such that $D(A)^2-2m=0$ or $-1$. By Lemma~\ref{lemma_classes2}, the divisor 
$$D'(A)=D(A)-mS(A)$$
is a $k$-class with $k=0$ or $-1$ respectively. Note that $D(A)\ge D'(A)$.

Our definitions imply that 
$$-A_n=A_{1\ldots n-1}+K_X=S'(A)+D(A)+S''(A)+K_X=2S(A)+K_X+D(A)\ge 2S(A)+K_X+D'(A).$$

Our strategy is as follows. Assuming $I(X,A)\subset I^{\rm red}(X)$, we find possible values of $S(A)$. For $A$ of type IIa we demonstrate that $-A_n=2S(A)+K_X+D(A)\ge 2S(A)+K_X+D'(A)$ is effective and therefore $A_n$ is not lo. For $A$ of types IIb and IIc we also consider the toric system $A'={\rm perm}_1(A)$. One has $S(A')=S(A)$ and $D(A')\ne D(A)$. In this case  we demonstrate that either $-A_n=2S(A)+K_X+D(A)$ or $-A_{n1}=-A'_n$ is effective.

Denote by $G_k$ the disjoint union of $k$ segments. Then, as a graph, $I(X,A)$ has a subgraph $\~I(X,A)$ isomorphic to $G_{n-4}=G_{8-\deg(X)}$.  Indeed, for $A$ of type IIa with 
$$A^2=(b_1,\ldots,b_p,c,d_1,\ldots,d_q,e)$$
the subgraph $\~I(X,A)$ is
$$\{\xymatrix{D_{1i}\ar@{-}[r] & D_{i,p+1}}|2\le i\le p\}\cup \{\xymatrix{D_{p+2,i}\ar@{-}[r] & D_{i,n}}|p+3\le i\le n-1\}.$$
For $A$ of type IIb 
$\~I(X,A)$ is
$$\{\xymatrix{D_{13}\ar@{-}[r] & D_{24}}\}\cup \{\xymatrix{D_{14}\ar@{-}[r] & D_{23}}\}\cup \{\xymatrix{D_{5,i}\ar@{-}[r] & D_{i,n}}|6\le i\le n-1\}.$$
For $A$ of type IIc 
$\~I(X,A)$ is
$$\{\xymatrix{D_{13}\ar@{-}[r] & D_{24}}\}\cup \{\xymatrix{D_{14}\ar@{-}[r] & D_{23}}\}\cup \{\xymatrix{D_{57}\ar@{-}[r] & D_{68}}\}\cup \{\xymatrix{D_{58}\ar@{-}[r] & D_{67}}\}.$$
Note that if $c\ge 0$ then $\~I(X,A)$ equals $I(X,A)$.
Also note that the sum of endpoints of any segment of $\~I(X,A)$ equals $S(A)$.

To find $S(A)$, we use the next 
\begin{lemma}
\label{lemma_s}
Let $X$ be a weak del Pezzo surface of degree $d$. Suppose the graph $G_{8-d}$ is fully embedded into $I^{\rm red}(X)$. Then the sum of the two endpoints of any segment in $G_{8-d}$ is the same, denote it by $S$. For any $C\in I^{\rm irr}(X)$ one has $C\cdot S\ge 1$ and if $\deg(X)\ge 4$ then $C\cdot S=1$.
\end{lemma}
\begin{proof}
The first statement is Lemma~\ref{lemma_gembed}, (3).

Let $C\in I(X)$ be an irreducible class, let $V\subset I(X)$ denote the set of neighbors of $C$. Then $I(X)\setminus(V\cup\{C\})$ is isomorphic to $I(X_{d+1})$ where $X_{d+1}$ is the blow-down of $C$, and $\deg(X_{d+1})=d+1$. By Lemma~\ref{lemma_gembed} $G_{8-d}$ cannot be fully embedded into $I(X_{d+1})$. Therefore there exists a vertex $T\in G_{8-d}\cap(V\cup\{C\})$. Let $T'$ be the partner of $T$ in $G_{8-d}$. Since $T$ is reducible, $T\in G_{8-d}\cap V$. Thus we have $C\cdot S=C\cdot T+C\cdot T'\ge C\cdot T\ge 1$ because $T'\ne C$. If moreover $\deg(X)\ge 4$ then both $T$ and $T'$ cannot be neighbors of $C$ and $C\cdot S=1$.
\end{proof}

\begin{lemma}
\label{lemma_gembed}
Let $X_d$ be a rational surface of degree $d$, $1\le d\le 7$.
\begin{enumerate}
\item $G_{9-d}$ cannot be fully embedded into $I(X_d)$;
\item for any full embedding of $G_{8-d}$ into $I(X_{d})$ any vertex in $I(X_d)$ has a neighbor in $G_{8-d}$;
\item for any full embedding of $G_{8-d}$ into $I(X_{d})$ the sum of the two endpoints of any segment in $G_{8-d}$ is the same.
\end{enumerate}
\end{lemma}
\begin{proof}
We argue by decreasing induction in $d$. For $d=7$ and $6$ the statements are straightforward. Let $d\le 5$.

(1) Let $G_{9-d}=G_{8-d}\sqcup (\xymatrix{x\ar@{-}[r] & y})$. We may assume that $X_d$ is a del Pezzo surface and $x$ is an irreducible $(-1)$-curve. Let $X_{d+1}$ be the blow-down of $x$. Then $I(X_{d+1})$ is fully embedded into $I(X_d)$ and may be identified with the subset of non-neighbors of $x$. Therefore $G_{8-d}$ is fully embedded into $I(X_{d+1})$, a contradiction.

(2) Suppose $G_{8-d}\subset I(X_d)$ and the vertex $x\in I(X_d)$ has no neighbors in $G_{8-d}$. As in (1), we may consider the blow-down $X_{d+1}$ of $x$ and deduce that $G_{8-d}\subset I(X_{d+1})$, a contradiction.

(3) Let $\xymatrix{x\ar@{-}[r] & y}$ and $\xymatrix{z\ar@{-}[r] & w}$ be two segments in $I(X_d)$. Take some $v\in G_{8-d}\setminus\{x,y,z,w\}$ and consider the blow-down $X_{d+1}$ of $v$ as above. Then we have $x,y,z,w\in G_{7-d}\subset I(X_{d+1})$ and  by induction we get $x+y=z+w$.
\end{proof}

For the future use we introduce the following notation
\begin{definition}
Let $X$ be a rational surface. We say that a set $\{D_1,\ldots,D_k\}$ of divisors on $X$ is \emph{good} if the following conditions are satisfied:
\begin{itemize}
\item for any $C\in I^{\rm irr}(X)$ there exists some $D_i, 1\le i\le k$, such that $C\cdot D_i\ge 1$;
\item for any $i\ne j$ one has $D_iD_j=1$.
\end{itemize}
In particular, a divisor $D$ is \emph{good} if $C\cdot D\ge 1$ for any irreducible $(-1)$-curve $C$.
\end{definition}

We prove Proposition~\ref{prop_main} for weak del Pezzo surfaces of degrees $3,4,5,6,7$ separately. Recall that $L_{ij}$ denotes $L-E_{ij}$.

\subsection{Degree $7$}
If $A$ is a toric system of type II on a weak del Pezzo surface of degree $7$ then $A^2$ is of the form $(0,k,-1,-1,-k-1)$ where $k\ge 2$ (or the symmetric one).
It follows that $I(X,A)=\{A_3,A_4\}$. On the other hand, $|I^{\rm red}(X)|\le 1$, therefore $I(X,A)\not\subset I^{\rm red}(X)$.
\begin{table}[h]
\begin{center}
\caption{Weak del Pezzo surfaces of degree $7$ and their irreducible $(-1)$-classes}
\label{table_d7}
\begin{tabular}{|p{3cm}|p{4cm}|p{2cm}|p{5cm}|}
 \hline
 type & $R^{\rm irr}$ & $|I^{\rm irr}|$ & $I^{\rm irr}$ \\
 \hline
 {$\emptyset$}& $\emptyset$ & $3$ & $E_1,E_2,L_{12}$  \\
 \hline
 {$A_1$}& $E_1-E_2$ & $2$ & $E_2,L_{12}$  \\
 \hline
\end{tabular}
\end{center}
\end{table}

\subsection{Degree $6$}

Suppose $A$ is a toric system of type II and $I(X,A)\subset I^{\rm red}(X)$. We want to  demonstrate that $-A_6\ge 0$.
In Table~\ref{table_d6} we list weak del Pezzo surfaces of degree $6$, their negative curves (see Proposition 8.3 in \cite{CT}) and find all good $0$-classes $S$ (i.e. such $S$ that $S\cdot C=1$ for any $C\in I^{\rm irr}(X)$. 
By Lemma~\ref{lemma_s}, the divisor $S(A)$ equals one of divisors $S$ from the table.

Recall that any $0$-class in $\Pic(X)$ is of the form $L_i=L-E_i$.
\begin{table}[h]
\begin{center}
\caption{Weak del Pezzo surfaces of degree $6$ and their irreducible $(-1)$-classes}
\label{table_d6}
\begin{tabular}{|p{1.5cm}|p{4cm}|p{2cm}|p{4cm}|p{2cm}|}
 \hline
 type & $R^{\rm irr}$ & $|I^{\rm irr}|$ & $I^{\rm irr}$ & good $S$ \\
 \hline
 {$\emptyset$}& $\emptyset$ & $6$ & $E_1,E_2,E_3,L_{12},L_{13},L_{23}$ & \\
 \hline
 {$A_1,4$}& $E_1-E_2$ & $4$ & $E_2,E_3,L_{12},L_{13}$ & \\
 \hline	
 {$A_1,3$ }& $L_{123}$ & $3$ & $E_1,E_2,E_3$ &  \\
 \hline
  {$2A_1$ } & $E_1-E_2,L_{123}$ & $2$ & $E_2,E_3$  & \\
 \hline
   {$A_2$} & $E_1-E_2,E_2-E_3$ & $2$ & $E_3,L_{12}$ & $L_3$\\
 \hline
   {$A_1+A_2$} & $E_1-E_2,E_2-E_3,L_{123}$ & $1$ & $E_3$ & $L_3$\\
 \hline
\end{tabular}
\end{center}
\end{table}

It suffices to consider surfaces $X_{6,A_2}$ and $X_{6,A_1+A_2}$, we have $S(A)=L_3$.
Therefore
\begin{multline*}
-A_6=K_X+2S(A)+D(A)=-3L+E_{123}+2L_3+D(A)=D(A)-L_3+(E_1-E_3)+(E_2-E_3).
\end{multline*}
We claim that $-A_6$ is effective and therefore $A_6$ is not lo, we get a contradiction.
Indeed, $E_1-E_3,E_2-E_3\ge 0$ on both surfaces $X=X_{6,A_2}$ and $X_{6,A_1+A_2}$. Let $D(A)$ be a $k$-class. Then $D(A)-L_3=D(A)-S(A)$ is a $(k-2)$-class by Lemma~\ref{lemma_classes2} because $D(A)\cdot S(A)=1$. Since $k\ge 1$, the divisor $D(A)-L_3$ is effective.

Thus we have proved that for any exceptional toric system $A$ on a weak del Pezzo surface of degree $6$ the set $I(X,A)$ cannot be in $I^{\rm red}(X)$.

\subsection{Degree $5$}

In Table~\ref{table_d5} we present information about weak del Pezzo surfaces of degree $5$ (see  Proposition 8.4 in \cite{CT}). The last column lists good $0$-classes $S$.
Note that the $0$-classes in $\Pic(X)$ are $L_i$, $1\le i\le 4$, and $2L-E_{1234}$.

\begin{table}[h]
\begin{center}
\caption{Weak del Pezzo surfaces of degree $5$ and their irreducible $(-1)$-classes}
\label{table_d5}
\begin{tabular}{|p{1.5cm}|p{4.7cm}|p{0.8cm}|p{5cm}|p{2.5cm}|}
 \hline
 type & $R^{\rm irr}$ & $|I^{\rm irr}|$ & $I^{\rm irr}$ & good $S$ \\
 \hline
 {$\emptyset$}& $\emptyset$ &  $10$ & $I(X)$ & \\
 \hline
 {$A_1$ }& $E_1-E_2$ & $7$ & $E_2,E_3,E_4,L_{12},L_{13},L_{14},L_{34}$ & \\
 \hline
  {$2A_1$} & $E_1-E_2,E_3-E_4$ &  $5$ & $E_2,E_4,L_{12},L_{13},L_{34}$ & \\
 \hline
   {$A_2$} & $E_1-E_2,E_2-E_3$ & $4$ & $E_3,E_4,L_{12},L_{14}$ & \\
 \hline
   {$A_1+A_2$} & $E_1-E_2,E_2-E_3,L_{123}$ & $3$ & $E_3,E_4,L_{14}$ & \\
 \hline
   {$A_3$} & $E_1-E_2,E_2-E_3,E_3-E_4$ & $2$ & $E_4,L_{12}$ & $L_4$\\
 \hline
   {$A_4$} & $E_1-E_2,E_2-E_3,E_3-E_4,L_{123}$ & $1$ & $E_4$ & $L_4,2L{-}E_{1234}$\\
 \hline
\end{tabular}
\end{center}
\end{table}

We have to consider the surfaces $X_{5,A_3}$ and $X_{5,A_4}$.

Let $X=X_{5,A_3}$. Then $S(A)=L_4$ and we have 
$$-A_7=K_X+2S(A)+D(A)=-3L+E_{1234}+2L_4+D(A)\ge -L-E_4+E_{123}+D'(A),$$
where $D'(A)$ is a $0$-class or a $1$-class.
Therefore $D'(A)$ is one of the following: $2L-E_{1234},L_i, L,2L-E_{ijl}$. Consequently,
$-A_7\ge D'(A)-L-E_4+E_{123}\ge 0$.
Indeed, if $D'(A)=2L-E_{1234}$ or $2L-E_{ijl}$ then 
$$-A_7\ge 2L-E_{1234}-L-E_4+E_{123}=L-2E_4=L_{34}+(E_3-E_4)\ge 0.$$
If $D'(A)=L_i$ or $L$ then 
$$-A_7\ge L_i-L-E_4+E_{123}=(E_1-E_i)+E_2+(E_3-E_4)\ge 0.$$

Let $X=X_{5,A_4}$. If $S(A)=L_4$, we argue exactly as above. Suppose $S(A)=2L-E_{1234}$. Then 
$$-A_7=K_X+2S(A)+D(A)=-3L+E_{1234}+4L-2E_{1234}+D(A)\ge L_{1234}+D'(A),$$
where $D'(A)$ is a $0$-class or a $1$-class.
Divisor $D'(A)$ is one of the following: $2L-E_{1234},L_i, L,2L-E_{ijl}$. We claim that 
$-A_7\ge D'(A)+L_{1234}\ge 0$.
Indeed, if $D'(A)=2L-E_{1234}$ or $2L-E_{ijl}$ then 
$$-A_7\ge 2L-E_{1234}+L_{1234}=3L-2E_{1234}=L_{123}+L_{124}+L_{34}\ge 0.$$
If $D'(A)=L_i$ or $L$ then 
$$-A_7\ge L_i+L_{1234}=(E_1-E_i)+L_{123}+L_{14}\ge 0.$$

\subsection{Degree $4$}

In Table~\ref{table_d4} we present information on weak del Pezzo surfaces of degree $4$, see \cite[Proposition 6.1]{CT}.
For $X_4,X_{4,A_1},X_{4,2A_1,9}$ we have $|I^{\rm red}(X)|<8$ so $I(X,A)$ cannot be fully embedded to $I^{\rm red}(X)$.

\begin{table}[h]
\begin{center}
\caption{Weak del Pezzo surfaces of degree $4$ and their irreducible $(-1)$-classes}
\label{table_d4}
\begin{tabular}{|p{1.6cm}|p{4.7cm}|p{0.8cm}|p{5.8cm}|p{2cm}|}
 \hline
 type & $R^{\rm irr}$ & $|I^{\rm irr}|$ & $I^{\rm irr}$ & good $S$ \\
 \hline
 {$\emptyset$}& $\emptyset$ & $16$ & & \\
 \hline
 {$A_1$ }& $E_4-E_5$ & $12$ & &  \\
 \hline
 {$2A_1,9$ }& $E_2-E_3, E_4-E_5$ & $9$  & $E_1,E_3,E_5,L_{12},L_{14},L_{23},L_{45},L_{24},$&  \\
 &&& $Q$ & \\
 \hline
 {$2A_1,8$} & $L_{123},E_4-E_5$ & $8$ & $E_1,E_2,E_3,E_5,L_{14},L_{24},L_{34},L_{45}$ & $2L-E_{1235}$ \\
 \hline
 {$A_2$}& $E_3-E_4,E_4-E_5$ & $8$ & $E_1,E_2,E_5,L_{12},L_{13},L_{23},L_{34},Q$ &  \\
 \hline
 {$3A_1$}& $L_{123},E_2-E_3,E_4-E_5$ & $6$ & $E_1,E_3,E_5,L_{14},L_{24},L_{45}$ & $2L-E_{1235}$ \\
 \hline
 {$A_1+A_2$}& $E_1-E_2,E_3-E_4,E_4-E_5$ & $6$ & $E_2,E_5,L_{12},L_{13},L_{34},Q$ & \\
 \hline
 {$A_3,5$}& $E_2-E_3,E_3-E_4,E_4-E_5$ & $5$ & $E_1,E_5,L_{12},L_{23},Q$ & \\
 \hline
 {$A_3,4$} & $L_{123},E_3-E_4,E_4-E_5$ & $4$ & $E_1,E_2,E_5,L_{34}$ & $2L-E_{1235}$, $2L-E_{1245}$ \\
 \hline
 {$4A_1$ } & $E_1-E_2,E_4-E_5,$ & $4$ & $E_2,E_3,E_5,L_{14}$ & $2L-E_{2345}$, \\
 & $L_{123},L_{345}$ &&& $2L-E_{1235}$  \\
 \hline 
 {$2A_1+A_2$} & $E_1-E_2,E_2-E_3,$ & $4$ & $E_3,E_5,L_{14},L_{45}$ & $2L-E_{1235}$ \\
 & $E_4-E_5,L_{123}$ &&&\\
 \hline
 {$A_1+A_3$} & $E_1-E_2,E_3-E_4,$ & $3$ & $E_2,E_5,L_{34}$ & $2L-E_{1245}$,  \\
 & $E_4-E_5,L_{123}$ &&& $2L-E_{1235}$ \\
 \hline
 {$A_4$} & $E_1-E_2,E_2-E_3,$ &  $3$ & $E_5,L_{12},Q$ & $L_5$ \\
 & $E_3-E_4,E_4-E_5$ &&& \\
 \hline
  {$2A_1+A_3$} & $E_1-E_2,L_{345},E_3-E_4,$ &  $2$ & $E_2,E_5$ & $2L-E_{25ij}$ \\
 & $E_4-E_5,L_{123}$ &&& \\
 \hline
  {$D_4$} & $E_2-E_3,E_3-E_4,$ & $2$ & $E_1,E_5$ & $2L-E_{15ij}$ \\
 & $E_4-E_5,L_{123}$ &&& \\
 \hline
 {$D_5$} & $E_1-E_2,E_2-E_3,$ & $1$ & $E_5$ & $2L-E_{ijk5}$, \\
 & $E_3-E_4,E_4-E_5,L_{123}$ &&& $L_5$ \\
 \hline 
\end{tabular}
\end{center}
\end{table}

Let $A$ be an exceptional toric system of the second type on $X$. In most cases we demonstrate that the divisor $-A_8=K_X+2S(A)+D(A)$ is effective. In some special cases we can only show that $-A_8$ or $-A_{81}$ is effective. Note that 
$D(A)$ is a $k$-class with $k\ge \deg(X)-5=-1$, hence $D(A)\ge 0$. 

Suppose $S(A)$ has the form $2L-E_{12345}+E_i$.
Then we check that $K_X+2S(A)\ge 0$. We have 
$$K_X+2S(A)=-3L+E_{12345}+2(2L-E_{12345}+E_i)=L-E_{jklm}+E_i,$$
where $\{i,j,k,l,m\}=\{1,2,3,4,5\}$.

Let $X=X_{4,2A_1,8},X_{4,3A_1},X_{4,A_3,4},X_{4,4A_1},X_{4,2A_1+A_2},X_{4,A_1+A_3},X_{4,2A_1+A_3},X_{4,D_4},X_{4,D_5}$ and $i=4$. We have $L-E_{1235}+E_4=L_{123}+(E_4-E_5)\ge 0$.

Let $X=X_{4,A_3,4},X_{4,A_1+A_3},X_{4,2A_1+A_3},X_{4,D_4},X_{4,D_5}$ and $i=3$. We have $L-E_{1245}+E_3=L_{123}+2(E_3-E_4)+(E_4-E_5)\ge 0$.

Let $X=X_{4,4A_1},X_{4,2A_1+A_3},X_{4,D_4}$ and $i=1$. We have $L-E_{2345}+E_1=L_{345}+(E_1-E_2)\ge 0$. 

Let $X=X_{4,D_4},X_{4,D_5}$ and $i=2$, we have $L-E_{1345}+E_2=L_{134}+(E_2-E_5)\ge 0$.

Now suppose $S(A)$ is $L_5$ and $X=X_{4,A_4}$ or $X_{4,D_5}$. Then we will demonstrate that $-A_8=K_X+2S(A)+D'(A)\ge 0$ or $-A_{81}\ge 0$. Recall that $D'(A)$ is a $0$ or a $(-1)$-class such that $D'(A)\cdot S(A)=1$. 

Suppose $D'(A)$ is a $(-1)$-class, then $D'(A)$ can  be either $E_5$ or $L_{ij}$ with $1\le i<j\le 4$. Suppose $D'(A)=L_{ij}$ with $1\le i<j\le 4$. Then 
$$K_X+2S(A)+D'(A)=-3L+E_{12345}+2L-2E_5+L_{ij}=E_{1234}-E_{ij5}=(E_k-E_5)+E_l\ge 0,$$
where $\{1,2,3,4\}=\{i,j,k,l\}$. Now suppose $D'(A)=E_5$. Then either $D(A)\ge D'(A)+L_5=L$ or $D(A)=D'(A)$. In the first case 
$$K_X+2S(A)+D(A)\ge -3L+E_{12345}+2L-2E_5+L=E_{1234}-E_{5}=(E_4-E_5)+E_{123}\ge 0.$$
In the second case $D(A)=D'(A)$ we see that $D(A)=E_5$ is irreducible. If the toric system~$A$ is of type IIa then $D(A)=A_i$ for some $i$ and $E_5\in I^{\rm irr}(X)\cap I(X,A)$, a contradiction.  If $A$ is of type IIb or IIc then $A_1$ is a $(-2)$-class.  Consider the toric system $A'={\rm perm}_1(A)$. We have $S(A')=S(A)$ and
$-A'_8=K_X+2S(A')+D'(A')=K_X+2S(A)+D'(A)$. Note that $A'_2\ne A_2$ and therefore $D'(A')\ne D'(A)=E_5$. It follows that $D'(A')=L_{ij}$ with $1\le i<j\le 4$ and $-A_{81}=-A'_8\ge 0$ as above.

Suppose now that $D'(A)$ is a $0$-class, then $D'(A)$ is $L_i$ with $1\le i\le 4$. We get 
$$-A_8\ge K_X+2S(A)+D'(A)=-3L+E_{12345}+2L_5+L_{i}=E_{jkl}-E_5=(E_j-E_5)+E_{kl}\ge 0,$$
where $\{i,j,k,l\}=\{1,2,3,4\}$.

\subsection{Degree $3$, general discussion}
In Tables~\ref{table_d3} and~\ref{table_d3bis} we present information about weak del Pezzo surfaces of degree $3$.

Recall that in $\Pic(X)$ there are the following $(-2)$-classes: $E_i-E_j, L_{ijk},Z:=2L-E_{123456}$ and the following $(-1)$-classes: $E_i,L_{ij},Q_i:=2L-E_{123456}+E_i$. 

Also we will deal with $0$- and $1$-classes on $X$. There are the following $0$-classes: $L_i, 2L-E_{ijkl}, C_i:=3L-E_{123456}-E_i$. All $1$-classes have the form $-K_X+C$, where $C$ is a $(-2)$-class.

\begin{table}[h]
\begin{center}
\caption{Weak del Pezzo surfaces of degree $3$, part 1}
\label{table_d3}
\begin{tabular}{|p{1.5cm}|p{4cm}|p{1cm}|p{6cm}|p{2cm}|}
 \hline
 type & $R^{\rm irr}$ & $|I^{\rm irr}|$ & $I^{\rm irr}$ & $|\{\text{good}\, S\}|$ \\
 \hline
 {$\emptyset$}& $\emptyset$ & $27$ & & $0$ \\
 \hline
 {$A_1$ }& $Z$ & $21$ & &  $0$ \\
 \hline
 {$2A_1$ }& $E_1-E_2, E_3-E_4$ & $16$  &&  $0$\\
 \hline
 {$A_2$}& $E_1-E_2,E_3-E_4$ & $15$ & $E_3,E_4,E_5,E_6,L_{12},L_{14},L_{15},L_{16}$,&$0$ \\

 &&& $L_{45},L_{46},L_{56},Q_3,Q_4,Q_5,Q_6$ & \\
 \hline
 {$3A_1$}& $E_1-E_2,E_3-E_4,$ & $12$ & $E_2,E_4,E_6,L_{12},L_{34},L_{56},L_{13},L_{15}$, & $0$ \\
 & $E_5-E_6$ && $L_{35},Q_2,Q_4,Q_6$ & \\
 \hline
 {$A_1{+}A_2$}& $E_4-E_5,E_1-E_2,$ & $11$ & $E_3,E_5,E_6,L_{12},L_{14},L_{16},L_{45},L_{46}$, & $1$\\
 & $E_2-E_3$  && $Q_3,Q_5,Q_6$ & \\
 \hline
 {$A_3$}& $E_1-E_2,E_2-E_3,$ & $10$ & $E_4,E_5,E_6,L_{12},L_{15},L_{16},L_{56},Q_4$, & $0$\\
 & $E_3-E_4$ && $Q_5,Q_6$ & \\
 \hline
 {$4A_1$} & $E_1-E_2,E_3-E_4,$ & $9$ & $E_2,E_4,E_6,L_{12},L_{34},L_{56},L_{13},L_{15}$, & $0$ \\
 & $E_5-E_6,Z$ && $L_{35}$ & \\
 \hline 
 {$2A_1{+}A_2$} & $E_4-E_5,L_{123},E_1-E_2,$ & $8$ & $E_3,E_5,E_6,L_{14},L_{16},L_{45},L_{46},Q_3$ &  $2$ \\
 & $E_2-E_3$ &&& \\
 \hline
 {$A_1{+}A_3$} & $E_5-E_6,E_1-E_2,$ & $7$ & $E_4,E_6,L_{12},L_{15},L_{56},Q_4,Q_6$ & $2$ \\
 & $E_2-E_3,E_3-E_4$ &&& \\
 \hline
 {$2A_2$} & $E_1-E_2,E_2-E_3,$ & $7$ & $E_3,E_6,L_{12},L_{14},L_{45},Q_3,Q_6$ & $3$ \\
 & $E_4-E_5,E_5-E_6$ &&& \\
 \hline
\end{tabular}
\end{center}
\end{table}

\begin{table}[h]
\begin{center}
\caption{Weak del Pezzo surfaces of degree $3$, part 2}
\label{table_d3bis}
\begin{tabular}{|p{2cm}|p{5cm}|p{1cm}|p{5cm}|p{2cm}|}
\hline 
 type & $R^{\rm irr}$ & $|I^{\rm irr}|$ & $I^{\rm irr}$ & $|\{\text{good}\, S\}|$ \\
 \hline
 {$A_4$} & $E_1-E_2,E_2-E_3,$ &  $6$ & $E_5,E_6,L_{12},L_{16},Q_5,Q_6$ & $3$\\
 & $E_3-E_4,E_4-E_5$ &&& \\
 \hline
 {$D_4$} & $E_1-E_2,E_3-E_4,$ &  $6$ & $E_2,E_4,E_6,L_{12},L_{34},L_{56}$ & $0$ \\
 & $E_5-E_6,L_{135}$ &&& \\
 \hline
 {$2A_1+A_3$} & $E_5-E_6,Z,$ &  $5$ & $E_4,E_6,L_{12},L_{15},L_{56}$ & $4$ \\
 & $E_1-E_2,E_2-E_3,E_3-E_4$  &&& \\
 \hline
 {$A_1+2A_2$} & $L_{123},E_1-E_2,E_2-E_3,$ & $5$ & $E_3,E_6,L_{14},L_{45},Q_3$ &  $5$ \\
 & $E_4-E_5,E_5-E_6$ &&& \\
 \hline
 {$A_1+A_4$} & $Z,E_1-E_2,E_2-E_3,$ &  $4$ & $E_5,E_6,L_{12},L_{16}$ & $6$\\
 & $E_3-E_4,E_4-E_5$ &&& \\
 \hline
 {$A_5$} & $E_1-E_2,E_2-E_3,E_3-E_4,$ &  $3$ & $E_6,L_{12},Q_6$ & $9$\\
 & $E_4-E_5,E_5-E_6$ &&& \\
 \hline
 {$D_5$} & $E_1-E_2,E_2-E_3,E_3-E_4,$ & $3$ & $E_5,E_6,Q_6$ & $7$\\
 & $E_4-E_5,L_{126}$  &&& \\
 \hline 
 {$3A_2$} & $E_1-E_2,E_2-E_3,E_4-E_5,$ &  $3$ & $E_3,E_6,L_{14}$ & $9$\\
 & $E_5-E_6,L_{123},L_{456}$ &&& \\
 \hline
 {$A_1+A_5$} & $Z,E_1-E_2,E_2-E_3,$ &  $2$ & $E_6,L_{12}$ & $12$\\
 & $E_3-E_4,E_4-E_5,E_5-E_6$ &&& \\
 \hline
 {$E_6$} & $L_{123},E_1-E_2,E_2-E_3,$ &  $1$ & $E_6$ & $17$\\
 & $E_3-E_4,E_4-E_5,E_5-E_6$ &&& \\
 \hline
\end{tabular}
\end{center}
\end{table}

Suppose $A$ is a toric system of type II such that $I(X,A)\subset I^{\rm red}(X)$. We will show that either $-A_9$ or $-A_{91}$ is effective and thus $A$ is not exceptional.
As in lower degrees, we will use good $0$-classes. Also we will need to use good pairs of $0$-classes, which are numerous. We prefer to use machine computation; therefore we do not present good $0$-classes and good pairs of $0$-classes in Tables~\ref{table_d3} and ~\ref{table_d3bis} but only give their numbers.

Recall that the divisor $D(A)$ is an $r$-class where $r\ge -2$. We start with the case $r\ge -1$ which is easier to handle and then consider the case $r=-2$.

\subsection{Degree $3$, $D^2(A)\ge -1$}

We have $-A_9=2S(A)+D(A)+K_X$. Since $D(A)$ is an $r$-class with $r\ge -1$, $D(A)$ is effective. Except for several cases we will check (using a computer) that $2S+K_X\ge 0$ for any good $0$-class $S$. Then we will handle exceptions by hand.
\begin{predl}
Let $S$ be a good $0$-class on a weak del Pezzo surface $X$ of degree $3$. Then $2S+K_X\ge 0$ except for the following cases:
$$S=L_6, X=X_{3,A_5}, X_{3,A_1+A_5}, X_{E_6}\quad\text{or}\quad S=C_6, X=X_{3,A_5}.$$
\end{predl}
\begin{proof}
See the script for SAGE at\\ \texttt{https://cocalc.com/projects/2130c0a6-a36a-4fec-9ee8-66c0f5ed53dd/files/}.
\end{proof}

Now we deal with exceptions. 

Let $X$ be $X_{3,A_5}, X_{3,A_1+A_5}$, or $X_{E_6}$ and $S(A)=L_6$. We have $2S(A)+K_X=-L+E_{12345}-E_6$. We aim to check that $2S(A)+K_X+D'(A)\ge 0$. Therefore we find all possible  $D'(A)$: such $0$ or $(-1)$-classes that $D'(A)\cdot L_6=1$. Clearly, $D'(A)$ can be one of $E_6,L_{ij}$ with $i,j\ne 6$, $Q_i$ with $i\ne 6$ ($(-1)$-classes) or $L_i$ with $i\ne 6$, $2L-E_{ijk6}$, $C_6$ ($0$-classes). 

Suppose $D'(A)=L_{ij}$, $D'(A)=L_i=L_{ij}+E_j$ or $D'(A)=2L-E_{ijk6}=L_{ij}+L_{k6}$ where $i,j,k\ne 6$, then
$$2S(A)+K_X+D'(A)\ge -L+E_{12345}-E_6 + L_{ij}=E_{klm}-E_6=(E_k-E_6)+E_{lm}\ge 0,$$
where $\{i,j,k,l,m\}=\{1,2,3,4,5\}$.

Suppose $D'(A)=Q_i$ or $D'(A)=C_6=Q_i+L_{i6}$  where $i\ne 6$, then
$$2S(A)+K_X+D'(A)\ge -L+E_{12345}-E_6 + 2L-E_{123456}+E_i=L+E_i-2E_6=L_6+(E_i-E_6)\ge 0.$$

Now suppose $D'(A)=E_6$. Consider two cases: $A$ is of type IIa and $A$ is of type IIb.
If $A$ is type IIa and $D(A)=D'(A)$ then $E_6=D(A)=A_{k}$ for some $k$, and $I(X,A)$ has an irreducible element $E_6$, a contradiction. If $A$ is type IIa and $D(A)\ne D'(A)$ then
$D(A)\ge D'(A)+S(A)=L\ge L_1$ and by the above computations we get $-A_9\ge 0$.
If $A$ is of type IIb then $A_1$ is a $(-2)$-class. Consider the toric system $A'={\rm perm}_1(A)$. One has $S(A')=S(A)=L_6$ and $D'(A')\ne D'(A)=E_6$. Applying the above arguments for $A'$, we get that $-A_{91}=-A'_9\ge 0$.

It remains to consider the case $X=X_{3,A_5}$  and $S(A)=C_6$. We have $2S(A)+K_X=3L-E_{12345}-3E_6$. We aim to check that $2S(A)+K_X+D'(A)\ge 0$. All possible  $D'(A)$ are: $E_i,L_{ij}$ with $i,j\ne 6$, $Q_6$ ($(-1)$-classes) and $L_6$, $2L-E_{ijk6}$, $C_i$  with $i\ne 6$ ($0$-classes). 

Suppose $D'(A)=E_i$ or $D'(A)=L_6=E_i+L_{i6}$, where $i\ne 6$, then 
$$2S(A)+K_X+D'(A)\ge 3L-E_{12345}-3E_6 + E_i=Q_i+L_{56}+(E_5-E_6)\ge 0.$$ 

Suppose $D'(A)=L_{ij}$, $D'(A)=2L-E_{ijk6}=L_{ij}+L_{k6}$ or $D'(A)=C_i=L_{ij}+Q_j$, where $i,j,k\ne 6$, then 
$$2S(A)+K_X+D'(A)\ge 3L-E_{12345}-3E_6 + L_{ij}=Q_l+Q_m+(E_k-E_6)\ge 0,$$ 
where $\{i,j,k,l,m\}=\{1,2,3,4,5\}$.

Now suppose $D'(A)=Q_6$. Consider two cases: $A$ is of type IIa and $A$ is of type IIb.
If $A$ is type IIa and $D(A)=D'(A)$ then $Q_6=D(A)=A_{k}$ for some $k$, and $I(X,A)$ has an irreducible element $Q_6$, a contradiction. If $A$ is type IIa and $D(A)\ne D'(A)$ then
$D(A)\ge D'(A)+S(A)=C_1+Q_1\ge C_1$ and by the above computations we get $-A_9\ge 0$.
If $A$ is of type IIb then $A_1$ is a $(-2)$-class. Consider the toric system $A'={\rm perm}_1(A)$. One has $S(A')=S(A)=L_6$ and $D'(A')\ne D'(A)=C_6$. Applying the above arguments for $A'$, we get that $-A_{91}=-A'_9\ge 0$.
 
Proposition~\ref{prop_main} for weak del Pezzo surfaces of degree $3$ and toric systems $A$ of type II with $D^2(A)\ge -1$ is now proven.

\subsection{Degree $3$, $D^2(A)=-2$, type IIa}
Let $A$ be a toric system of type IIa on $X$, such that $D^2(A)=-2$. Then $A^2=(b,-2,d,-3)$ where $b=(b_1,\ldots,b_p)$ and $d$ are of the form $(0)$ or $(-1,-2,\ldots,-2,-1)$. Assuming  that $I(X,A)\subset I^{\rm red}(X)$, we will demonstrate t/hat $-A_9\ge 0$. By Lemma~\ref{lemma_s}, $S(A)=A_{1\ldots p}=A_{p+2,\ldots, 8}$ is a good $0$-class. Consider the toric system $A'={\rm perm}_{p+1}(A)$. Clearly, $I(X,A')=I(X,A)\subset I^{\rm red}(X)$, therefore $S(A')$ is also a good $0$-class. We have $S(A')=A'_{1\ldots p}=A_{1\ldots p+1}$, $S(A)S(A')=1$. Therefore
$$-A_9=A_{1\ldots 8}+K_X=S(A)+S(A')+K_X.$$
Now the effectivity of $-A_9$ follows from the next proposition, which we check using a computer.
\begin{predl}
Let $S,S'$ be good $0$-classes on a weak del Pezzo surface of degree $3$ such that $S\cdot S'=1$. Then $S+S'+K_X\ge 0$.
\end{predl}
\begin{proof}
See the script for SAGE at\\ \texttt{https://cocalc.com/projects/2130c0a6-a36a-4fec-9ee8-66c0f5ed53dd/files/}.
\end{proof}

\subsection{Degree $3$, $D^2(A)=-2$, type IIb}
Let $A$ be a toric system of type IIb on $X$, such that $D^2(A)=-2$. Then $A^2=(-2,-1,-2,-1,-1,-2,-2,-1,-3)$. Assuming  that $I(X,A)\subset I^{\rm red}(X)$, we will demonstrate that $-A_9\ge 0$ or $-A_{91}\ge 0$. By Lemma~\ref{lemma_s}, $S=S(A)=A_{1223}=A_{5678}$ is a good $0$-class. One has $D(A)=A_4-A_2$.
Consider the toric system $A'={\rm perm}_{1}(A)$. One has $S(A')=S(A)$ and $D(A')=A_4-A_{12}$.
Denote $S_1=S(A)+D(A)=A_{1234}=A_{45678}-A_2$ and $S_2=S(A')+D(A')=A_{234}=A_{45678}-A_{12}$, $S_1$ and $S_2$ are $0$-classes.
Then we have $SS_1=SS_2=S_1S_2=1$. 
\begin{lemma}
In the above assumptions and notation, $(S_1,S_2)$ is a good pair.
\end{lemma}
\begin{proof}
Clearly, $0$-classes $S_1$ and $S_2$ satisfy $S_1S_2=1$.
By Lemma~\ref{lemma_cscs0} it then follows that 
$$|\{C\in I(X)\mid CS_1=CS_2=0\}|=5.$$

It is easy to see that  the $(-1)$-classes $C=A_4,A_{34},A_8,A_{78},A_{678}$ satisfy $CS_1=CS_2=0$. Therefore
$$\{C\in I(X)\mid CS_1=CS_2=0\}=\{A_4,A_{34},A_8,A_{78},A_{678}\}\subset I(X,A)\subset I^{\rm red}(X).$$
It follows that for any $C\in I^{\rm irr}(X)$ either $CS_1$ or $CS_2\ge 1$, hence $(S_1,S_2)$ is good.
\end{proof}

Further,
$$-A_9=A_{1\ldots 8}+K_X=S+S_1+K_X\quad\text{and}\quad -A'_9=-A_{91}=S+S_2+K_X.$$
Now the effectivity of $-A_9$ or $-A_{91}$ follows from the next proposition, which we check using a computer.
\begin{predl}
Let $X$ be a weak del Pezzo surface of degree $3$, let $S$ be a good $0$-class on $X$. Let $(S',S'')$ be a good pair of $0$-classes such that $SS'=SS''=1$. 
Then $S+S'+K_X\ge 0$ or $S+S''+K_X\ge 0$.
\end{predl}
\begin{proof}
See script for SAGE at\\ \texttt{https://cocalc.com/projects/2130c0a6-a36a-4fec-9ee8-66c0f5ed53dd/files/}.
\end{proof}

\section{Toric systems of the second kind: types III-VI}
In this section we prove Proposition~\ref{prop_main} for toric systems $A$ of types III-VI.

\subsection{Type III}
Here we prove Proposition~\ref{prop_main} for toric system $A$ of type III.

Let $A=(A_1,\ldots,A_n)$ be an exceptional toric system of type III. This means that (up to a symmetry) $A^2$
is one of following sequences: 
\begin{align*}
&\text{type IIIa} & &(1,0,-2,\ldots,-2,-1,4-n),\\ 
&\text{type IIIb} & &(-1,0,0,-2,\ldots,-2,-1,4-n),\\
&\text{type IIIc} &  &(-1,-2,\ldots,-2,0,0,-2,\ldots,-2,-1,4-n).
\end{align*}

Supposing that $I(X,A)\subset I^{\rm red}(X)$ we will demonstrate that $-A_n$ is effective.
Recall that $4-n=d-8$. We have to consider surfaces of degrees $3,4,5$.

For $A$ of type IIIa denote $H=A_1$. Note that $H=A_{2\ldots n-1}$: indeed, the difference $A_1-A_{2\ldots n-1}$ has zero intersection with $A_i$, $1\le i\le n$, and therefore vanishes. For $A$ of type IIIb denote $H=A_{12}=A_{3\ldots n-1}$.
For $A$ of type IIIc denote $H=A_{1\ldots m}=A_{m+1\ldots n-1}$, where $a_m=a_{m+1}=0$. Clearly, $H$ is a $1$-class and $A_{1\ldots n-1}=2H$.

\begin{lemma}
\label{lemma_mainIII}
Suppose $I(X,A)\subset I^{\rm red}(A)$ and $\deg(X)\ge 3$. Then $H$ is a good $1$-class.
\end{lemma}
\begin{proof}
By Lemma~\ref{lemma_ch0}, we have 
$$|\{C\in I(X)|CH=0\}|=n-3.$$

Denote by $P$ the set of  elements $C\in I$ such that $CH=0$. One can check that $I(X,A)\subset P$ and $|I(X,A)|=n-3$. It follows that $I(X,A)=P$. By our assumptions,  $I(X,A)\subset I^{\rm red}(X)$. Consequently, for $C\in I^{\rm irr}$ one has $CH\ge 1$.
\end{proof}

\begin{predl}
\label{prop_III}
Let $X$ be a weak del Pezzo surface with $3\le \deg (X)\le 5$. Let $H$ be a good $1$-class on $X$. Then $2H+K_X\ge 0$.
\end{predl}
\begin{proof}
See the scripts for SAGE at\\ \texttt{https://cocalc.com/projects/2130c0a6-a36a-4fec-9ee8-66c0f5ed53dd/files/}.
\end{proof}

To prove Proposition~\ref{prop_main} it remains to note that $-A_n=A_{1\ldots n-1}+K_X=2H+K_X$.

\subsection{Type IV}

Here we prove Proposition~\ref{prop_main} for toric system $A$ of type IV. 

Let $A=(A_1,\ldots,A_n)$ be an exceptional toric system of type IV. This means that (up to a symmetry) $A^2$
is the following sequence: 
$$(-2,0,1,-2,\ldots,-2,-1,4-n).$$ We will demonstrate that either $-A_n$ or $-A_{n1}$ is effective provided that $I(X,A)\subset I^{\rm red}(X)$. Note that we have to consider surfaces of degrees $3,4,5$.

Denote $S_1=A_{3\ldots n-1}-A_2$. Note that $S_1=A_{12}$: the difference $A_{122}-A_{3\ldots n-1}$ has zero intersection with any $A_i$, $1\le i\le n$, and therefore vanishes. Hence, $S_1$ is a $0$-class. 
Similarly, denote $S_2=A_{3\ldots n-1}-A_{12}=A_2$, it is also a $0$-class.  

\begin{lemma}
\label{lemma_mainIV}
Suppose $I(X,A)\subset I^{\rm red}(X)$. Then the pair $(S_1,S_2)$ is good.
\end{lemma}
\begin{proof}
Clearly, $S_1S_2=1$. By Lemma~\ref{lemma_cscs0}, there exist exactly $8-d$ elements $C$ in $I(X)$ such that $CS_1=CS_2=0$.
Note that $I(X,A)=\{A_{k\ldots n-1}\mid k=4\ldots n-1\}$ and all $n-4$ elements $C$ of $I(X,A)$ satisfy $CS_1=CS_2=0$. Since $8-d=n-4$, it follows that $\{C\in I(X)\mid CS_1=CS_2=0\}=I(X,A)$. By the assumption, $I(X,A)\subset I^{\rm red}(X)$. Therefore for any $C\in I^{\rm irr}(X)$ one has $CS_1\ge 1$ or $CS_2\ge 1$, i.e. $(S_1,S_2)$ is a good pair.
%
%
\end{proof}

\begin{predl}
\label{prop_IV}
Let $X$ be a weak del Pezzo surface with $3\le \deg(X)\le 5$ and $(S_1,S_2)$ be a good pair of $0$-classes  on $X$. Then either $2S_1+S_2+K_X$ or $2S_2+S_1+K_X$ is effective.
\end{predl}
\begin{proof}
See the scripts for SAGE at\\ \texttt{https://cocalc.com/projects/2130c0a6-a36a-4fec-9ee8-66c0f5ed53dd/files/}.
\end{proof}

To prove Proposition~\ref{prop_main}, it remains to note that 
$$-A_n=A_{1\ldots n-1}+K=A_{12}+(A_{3\ldots n-1}-A_2)+A_2+K=2S_1+S_2+K$$ and 
$$-A_{n1}=A_{2\ldots n-1}+K=A_{2}+(A_{3\ldots n-1}-A_{12})+A_{12}+K=2S_2+S_1+K.$$

\subsection{Type V}
Here we prove Proposition~\ref{prop_main} for toric system $A$ of type V.

Let $A=(A_1,\ldots,A_n)$ be an exceptional toric system of type V. This means that (up to a symmetry) $A^2$
is the following sequence: 
$$(-2,-1,-1,0,-2,\ldots,-2,-1,5-n).$$ 
We assume that $\deg(X)=3$ or $4$ (then $n=9$ or $8$): for greater degrees the toric system $A$ is of the first kind.
We will demonstrate that either $-A_n$ or $-A_{n1}$ is effective provided that $I(X,A)\subset I^{\rm red}(X)$.

First, note that $A_{1223}=A_{4\ldots n-1}$. Indeed, the difference $A_{1223}-A_{4\ldots n-1}$ has zero intersection with every $A_i$ and thus vanishes. Consider the $(-1)$-classes 
$$C_1=A_2=A_{4\ldots n-1}-A_{123},\quad C_0=A_3,\quad C_2=A_{12}=A_{4\ldots n-1}-A_{23}.$$ 
We have  $C_0C_1=C_0C_2=1, C_1C_2=0$. Denote 
$$H=C_1+C_0+C_2;$$
it is a $1$-class by Proposition~\ref{prop_classes0}.

\begin{lemma}
\label{lemma_mainV}
Suppose $I(X,A)\subset I^{\rm red}(X)$. Then in the above notations the $1$-class $H$ is good.
\end{lemma}
\begin{proof}
By Lemma~\ref{lemma_ch0},
$$|\{C\in I(X)|CH=0\}|=n-3.$$
Note that the following $n-3$ divisors $C=A_2,A_{12},A_{5\ldots n-1},A_{6\ldots n-2},\ldots, A_{n-2,n-1},A_{n-1}$ are $(-1)$-classes and satisfy $CH=0$. It follows that 
$$\{C\in I(X)|CH=0\}=\{A_2,A_{12},A_{5\ldots n-1},A_{6\ldots n-2},\ldots, A_{n-2,n-1},A_{n-1}\}\subset I(X,A)\subset I^{\rm red}(X).$$
Therefore, for any $C\in I^{\rm irr}(X)$ one has $CH\ge 1$.
\end{proof}

\begin{predl}
\label{prop_V}
Let $X$ be a weak del Pezzo surface with $3\le \deg(X)\le 4$ and $C,C',C''$ be $(-1)$-classes such that $CC'=CC''=1, C'C''=0$, $C\in I^{\rm red}(X)$ and the $1$-class $H=C+C'+C''$ is good. Then either $K_X+2H-C'$ or $K_X+2H-C''$ is effective.
\end{predl}
\begin{proof}
See the scripts for SAGE at\\ \texttt{https://cocalc.com/projects/2130c0a6-a36a-4fec-9ee8-66c0f5ed53dd/files/}.
\end{proof}

To prove Proposition~\ref{prop_main}, it remains to note that $A_{4\ldots n-1}=C_0+C_1+C_2=H$, the $(-1)$-class $C_0=A_3$ is in $I^{\rm red}(X)$,
$$-A_n=K_X+A_{1\ldots n-1}=K_X+A_{123}+H=K_X+2H-C_1$$ and 
$$-A_{n1}=K_X+A_{2\ldots n-1}=K_X+A_{23}+H=K_X+2H-C_2.$$
Therefore one can apply Proposition~\ref{prop_V} for $C=C_0, C'=C_1,C''=C_2$.

\subsection{Type VI}
Here we prove Proposition~\ref{prop_main} for the toric system $A$ of type VI.

Let $A=(A_1,\ldots,A_n)$ be an exceptional toric system of type VI. This means that  (up to a symmetry) $A^2$
is the following sequence: $(-2,-2,-1,-2,0,-2,\ldots,-2,-1,6-n)$. As we consider only surfaces of degree $\ge 3$, we may assume that $\deg(X)=3$: for greater degrees the toric system $A$ is of the first kind. Hence, $n=3$ and 
$$A^2=(-2,-2,-1,-2,0,-2,-2,-1,-3).$$
We will demonstrate that either $-A_9$, $-A_{91}$ or $-A_{912}$ is effective provided that $I(X,A)\subset I^{\rm red}(X)$.

First, note that $A_{5678}=A_{1223334}$. Indeed, the difference $A_{1223334}-A_{5678}$ has zero intersection with every $A_i$ and thus vanishes. Consider the following divisors
$$S_1=A_{5678}-A_3=A_{122334},\quad S_2=A_{5678}-A_{23}=A_{12334},\quad S_3=A_{5678}-A_{123}=A_{2334}.$$ 
We have  $S_1S_2=S_1S_3=S_2S_3=1$. 
It follows from Lemma~\ref{lemma_classes1} that $S_1,S_2,S_3$ are $0$-classes.

\begin{lemma}
\label{lemma_mainVI}
Suppose $I(X,A)\subset I^{\rm red}(X)$. Then in the above notations the pairs 
$(S_1,S_2),(S_1,S_3),(S_2,S_3)$ are good. 
\end{lemma}
\begin{proof}
Let us check that the pair $(S_1,S_2)$ is good, for two other pairs the arguments are similar. By Lemma~\ref{lemma_cscs0},
$$|\{C\in I(X)\mid CS_1=CS_2=0\}|=5.$$
One can see that the five divisors $C=A_{6},A_{67},A_{678},A_{123},A_{1234}$ satisfy $CS_1=CS_2=0$. Hence
$$\{C\in I(X)\mid CS_1=CS_2=0\}=\{A_{6},A_{67},A_{678},A_{123},A_{1234}\}\subset I(X,A)\subset I^{\rm red}(A).$$
Therefore for any $C\in I^{\rm irr}(A)$ one has $CS_1\ge 1$ or $CS_2\ge 1$.
\end{proof}

\begin{predl}
\label{prop_VI}
Let $X$ be a weak del Pezzo surface of degree $3$ and $S,S',S''$ be $0$-classes such that
the pairs $(S,S')$, $(S,S'')$, $(S',S'')$ are good. Then at least one of the divisors $K_X+S+S', K_X+S+S'', K_X+S'+S''$ is effective.
\end{predl}
\begin{proof}
See the script for SAGE at\\ \texttt{https://cocalc.com/projects/2130c0a6-a36a-4fec-9ee8-66c0f5ed53dd/files/}.
\end{proof}

To prove Proposition~\ref{prop_main}, it remains to note that 
\begin{gather*}
-A_9=K_X+A_{1\ldots 8}=K_X+S_1+S_2,\\
-A_{91}=K_X+A_{1\ldots 8}=K_X+S_1+S_3,\\
-A_{912}=K_X+A_{1\ldots 8}=K_X+S_2+S_3.
\end{gather*}

\section{A counterexample in degree $2$}
\label{section_ce}
Here we give an example of a strong exceptional toric system $A$ on a weak del Pezzo surface~$X$ of degree $2$ such that $A$ is not an augmentation in the weak sense. Thus, the statement of Theorem~\ref{theorem_main} is not true for surfaces of degree $\le 2$.

Let $X$ be a weak del Pezzo surface of degree $2$ of type $A_1+2A_3$. Explicitly, and let $C$ be a smooth conic on $\P^2$, let $P_1,P_4\in C$ be two distinct points. Let $H\subset \P^2$ be the tangent line to $C$ at $P_1$. Let $X'$ be the blow-up of $P_1$ and $P_4$, let $E_1,E_4\subset X_7$ be the exceptional divisors and $H',C'\subset X_7$ be the strict transforms of $H$ and $C$. Let $P_2=C'\cap H'\cap E_1$, $P_5=C'\cap E_4$. Denote by $X''$ the blow-up of $P_2$ and $P_5$ on $X'$. Let $E_2,E_5$ be the exceptional divisors and $H'',C''$ be the strict transforms of $H,C$ on $X''$. Denote $P_3=E_2\cap H''$ and $P_6=E_5\cap C''$. Let $X'''$ be the blow-up of $P_3$ and $P_6$. Let $E_3,E_6$ be the exceptional divisors and $H''',C'''$ be the strict transforms of $H,C$ on $X'''$. Finally, let $X=X''''$ be the blow-up of the point $P_7=E_6\cap C'''$ on $X'''$. One can check that there are $7$   $(-2)$-classes on $X$ which correspond to irreducible curves: 
$$L_{123}; E_1-E_2, E_2-E_3, 2L-E_{124567}; E_4-E_5, E_5-E_6, E_6-E_7.$$
One has $$I^{\rm irr}(X)=\{E_3,E_7,L_{14},L_{45}\}.$$

Consider the sequence
$$A=(L_{25}, L_{137}, E_3-E_4, L_{236}, L_{15}, E_1-E_7, -L_{567}, 3L-E_{12345567}, -L_{345}, -2L+E_{12257}).$$
One can check that $A$ is a toric system and 
$$A^2=(-1,-2,-2,-2,-1,-2,-2,-1,-2,-3).$$
That is, $A$ is a toric system of the second kind of type IIb. 

Denote $C_i=3L-E_{1234567}-E_i$ and $Q_{ij}=2L-E_{1234567}+E_{ij}$. Then 
\begin{multline*}
I(X,A)=\\
=\{D_{1j}\mid 2\le j\le 5\}\cup \{D_{ij}\mid 2\le i\le 5, 6\le j\le 8\}\cup
\{D_{ij}\mid 6\le i\le 8, 9\le j\le 10\}=\\
=\{L_{25}, Q_{46}, Q_{36}, C_2\}\cup \{L_{15}, Q_{47},Q_{37},C_1, L_{57}, Q_{14}, Q_{13}, C_7, E_6,L_{23}, L_{24}, Q_{56}\}\cup \\
\cup \{C_5,Q_{67}, Q_{16}, Q_{34}, L_{12}, L_{27}\}.
\end{multline*}
It follows that $I(X,A)\subset I^{\rm red}(X)$ and therefore $A$ is not an augmentation in any sense.

On the other hand, $A$ is strong exceptional. To check this, we use the next 
\begin{lemma}
Let $X$ be a weak del Pezzo surface of degree $2$. Let $A$ be a toric system on~$X$ with 
$$A^2=(-1,-2,-2,-2,-1,-2,-2,-1,-2,-3).$$
Then $A$ is strong exceptional if and only if all $(-2)$-classes $A_{k\ldots l}$ with $1\le k\le l\le 9$ are neither effective nor anti-effective and the $(-3)$-classes $A_{9,10}$ and $A_{10}$ are not anti-effective.
\end{lemma}
\begin{proof}
Direct consequence of Theorem~\ref{theorem_checkonlyminustwo}.

\end{proof}

By the above lemma, we have to check that the $(-2)$-classes
\begin{multline*}
A_2=L_{137}, A_3=E_3-E_4, A_4=L_{236}, A_{23}=L_{147}, A_{34}=L_{246}, A_{234}=2L-E_{123467}, \\
A_6=E_1-E_7, A_7=-L_{567}, A_{67}=-L_{156}, A_9=-L_{345}
\end{multline*}
are neither effective nor anti-effective and that the $(-3)$-classes 
$$A_{10}=-(2L-E_{12257}), A_{9,10}=-(3L-E_{12234557})$$
are not anti-effective.
The first claim is straightforward, for the second the argument is based on the following trivial fact.
\begin{lemma}
Suppose  $D$ is an effective divisor on a surface and $C$ is an irreducible curve with $C^2<0$. If $D\cdot C<0$ then $D-C$ is also an effective divisor.
\end{lemma}
Suppose now that $D=-A_{10}=2L-E_{12257}\ge 0$. Since $D\cdot (2L-E_{124567})=-1$ and $2L-E_{124567}\in R^{\rm irr}(X)$ it follows that 
$D-(2L-E_{124567})=E_{46}-E_2\ge 0$. Since $(E_{46}-E_2)\cdot (E_6-E_7)=-1$ and $E_6-E_7\in R^{\rm irr}(X)$, it follows that $(E_{46}-E_2)-(E_6-E_7)=E_{47}-E_2\ge 0$.
Since $(E_{47}-E_2)\cdot E_7=-1$ and $E_7\in I^{\rm irr}(X)$, it follows that $(E_{47}-E_2)-E_7=E_{4}-E_2\ge 0$. The latter is clearly not true. 
Similarly one checks that $D=-A_{9,10}=3L-E_{12234557}$ is not effective. The sequence of negative curves which are subtracted from $D$ is 
$$E_1-E_2, E_4-E_5, L_{123}, 2L-E_{124567}, E_6-E_7,E_7,$$
Finally we get $E_2-E_4$ which is not effective.
Therefore, $A$ is a strong exceptional toric system.

\begin{remark} 
Note that the surface $X$ can be obtained from $\P^2$ in only two blow-ups.
Indeed, the four $(-1)$-curves from $I^{\rm irr}(X)$ do not intersect each other and can be blown down simultaneously. The images of the $(-2)$-curves $E_1-E_2, E_2-E_3, 2L-E_{124567}, E_6-E_7, E_5-E_6, E_4-E_5$ on the blow-down are irreducible $(-1)$-curves and form a cycle. Blowing down three of them we get $\P^2$. It confirms 
a conjecture by Hille and Perling  saying that any  surface with a strong exceptional toric system can be obtained from some  Hirzebruch surface in two blow-ups.
\end{remark}

\section{Some remarks about augmentations in degrees $2$ and $1$}
\label{section_remarks}
The counterexample from Section~\ref{section_ce} was found using a computer. Moreover, for all toric systems of types IIb and III-VI on weak del Pezzo surfaces of degrees $2$ and $1$ we checked whether they are augmentations or not. Surprisingly, all  exceptional toric systems of types III-VI on weak del Pezzo surfaces of degrees $2$ and $1$ are elementary augmentations (after several permutations). 
Note that an elementary augmentation of a toric system of type II is also of type II (or is not of the second kind). Consequently, any exceptional toric system of types III-VI on any weak del Pezzo surface of degrees $2$ or $1$ is an exceptional augmentation. The same holds for strong exceptional.

For type IIb and degree $2$ there exist some strong exceptional toric systems which are not augmentations, but there are not many of them. 
We briefly call such toric systems \emph{counterexamples} (to the conjecture of Hille and Perling). 
Recall that a toric system $A$ of type IIb (up to symmetry) on a surface of degree $2$ has
$$A^2=(-1,-2,-2,-2,-1,-2,-2,-1,-2,-3).$$
The total number of toric systems $A$ with $A^2$ as above on any particular surface of degree $2$  is equal to the order of the Weyl group $W(E_7)$, which equals $2903040$.  
The counterexamples only occupy a tiny fraction ($<0.3\%$) of possible systems.
Two toric systems on a weak del Pezzo surface are regarded to be \emph{essentially the same} if they differ by a Weyl group element that stabilizes the set of $(-2)$-curves (and thus the set of $(-1)$-curves). Two toric systems which are essentially the same are either both counterexamples or both satisfy the Hille-Perling conjecture.
Searching over all elements of the Weyl group, we found a total of $390$ essentially different counterexamples of type IIb with this $A^2$
 in degree $2$, which spread over five different root subsystems in $E_7$ (one of which is realizable by a surface only in characteristic $2$, type $7A_1$). These root subsystems all contain five disjoint $(-2)$-curves (i.e. orthogonal simple roots), and do not contain $D_4$. Corresponding surfaces contain holes in their effective cones, i.e. there are non-effective divisors that have a positive multiple which is effective. In fact, one can say more. For any strong exceptional toric system $A$ with 
$A^2=(-1,-2,-2,-2,-1,-2,-2,-1,-2,-3)$ such that $A$ is not an augmentation, at least one  $(-3)$-class $D$ among  $A_{10}$ and $A_{9,10}$ is a hole in the anti-effective cone. I.e., $-D$ is not effective but some positive multiple of $-D$ is effective.

\begin{table}[h]
\begin{center}
\caption{Strong exceptional toric systems of type IIb which are not augmentations on weak del Pezzo surfaces of degree $2$}
\label{table_ce1}
\begin{tabular}{|p{4.7cm}|p{1.6cm}|p{1.6cm}|p{1.6cm}|p{1.6cm}|p{1.6cm}|}
        \hline
        \text{root subsystem} & $7A_1$ & $6A_1$ & $5A_1$ & $A_3+3A_1$ & $A_1+2A_3$  \\
		 \hline
		 \text{number of essentially} & 48 & 90 & 36 & 144 & 72\\
		 \text{different counterexamples}&&&&&\\
		 \hline
		 \text{order of stabilizer of} & 168 & 48 & 32 & 4 & 4\\
		 \text{set of }$(-2)$\text{-curves}&&&&&\\
		 \hline
		 \text{total number of} & 8064 &
		 4320 & 1152 & 576 & 288\\
		 \text{counterexamples}&&&&&\\
		 \hline	
\end{tabular}
\end{center}
\end{table}

\begin{table}[h]
\begin{center}
\caption{Exceptional toric systems of type IIb which are not augmentations on weak del Pezzo surfaces of degree $2$}
\label{table_ce2}
\begin{tabular}{|p{4.7cm}|p{1cm}|p{1cm}|p{1cm}|p{1.6cm}|p{1.6cm}|p{1.6cm}|p{1.6cm}|}
        \hline
        \text{root subsystem}  & $7A_1$ & $6A_1$ & $5A_1$ & $A_3+3A_1$ & $A_1+2A_3$ & $D_4+2A_1$ & $D_4+3A_1$ \\
		 \hline
		 \text{number of essentially} & 90 & 126 & 36  & 144 & 72 & 9 & 177\\
		 \text{different counterexamples}&&&&&&&\\
		 \hline
		 \text{order of stabilizer of} & 168 & 48 & 32 & 4 & 4 & 4 & 6\\
		 \text{set of }$(-2)$\text{-curves}&&&&&&&\\
		 \hline
		 \text{total number of} & 15120 &
		 6048 & 1152 & 576 & 288 & 36 & 1062\\
	     \text{counterexamples}&&&&&&&\\
		 \hline	
\end{tabular}
\end{center}
\end{table}

If we also care about exceptional toric systems which are not augmentations and which are not necessarily strong, there are $36$ and $42$ more essentially different examples on surfaces of types $6A_1$ and $7A_1$ respectively. In addition, there are $177$ and $9$ essentially different examples on surfaces of types $D_4+3A_1$ and $D_4+2A_1$ respectively. It is still true for all these examples that at least one of $A_{10}$ and $A_{9,10}$ is a hole in the anti-effective cone. The seven root subsystems appearing above together with $D_6+A_1$ exhaust all root subsystems of $E_8$ containing five mutually orthogonal simple roots.

The number of essentially different counterexamples for each root subsystem is summarized in Tables~\ref{table_ce1} and~\ref{table_ce2}.

See appendix A for the details about the computer search for counterexamples.

\section{Classification of surfaces with cyclic strong exceptional toric systems}
\label{section_cyclicstrong}

In this section we prove that cyclic strong exceptional collections of line bundles having maximal length can exist only on weak del Pezzo surfaces and  determine which surfaces possess such collections and which do not.

First we demonstrate that a surface with a cyclic strong exceptional collection of line bundles having maximal length must be rational. It is proven by Morgan Brown and Ian Shipman in \cite[Theorem 4.4]{BS} that a surface with a full strong exceptional collection of line bundles is rational. It seems that fullness is not really needed in the proof in \cite{BS}, nevertheless, we prefer to give a separate simple proof for the case of cyclic strong exceptional  collections.

\begin{lemma}
\label{lemma_rational}
Let $X$ be a smooth projective surface admitting a cyclic strong exceptional toric system of maximal length. Then $X$ is a rational surface.
\end{lemma}

\begin{proof}
Assume that $A_1,\ldots,A_n$ is a cyclic strong exceptional toric system of maximal length on $X$.  By Theorem~\ref{theorem_HP}, the sequence $A^2$ is admissible. 
We have  $A_i^2\ge -2$ for all $i$, therefore, $A^2$ is cyclic strong admissible, see Table~\ref{table_csadm}. Then we can group $A_1+\ldots+A_n$ into two groups $D_1:=A_1+\ldots+A_m, D_2:=A_{m+1}+\ldots+A_n$ for some $m$ such that $D_1,D_2$ are both non-zero strong left-orthogonal divisors and $D_1^2\ge -1, D_2^2\ge -1$. Thus $h^0(D_i)=D_i^2+2>0$ for $i=1,2$. Then $D_1,D_2$ are both effective, hence $-K_X=A_1+\ldots+A_n$ is effective and $-2K_X$ is effective: $-2K_X>0$. It follows immediately that $h^0(2K_X)=0$: if $2K_X$ is effective then $0=2K_X+(-2K_X)>0$. Recall that $h^1(\O_X)=0$ since line bundles on $X$ are exceptional. By Castelnuovo's rationality criterion, we conclude that $X$ is a rational surface.
\end{proof}

\begin{predl} 
\label{prop_cyclicwdp}
Let $X$ be a smooth projective surface with a cyclic strong exceptional toric system. Then~$X$ is a weak del Pezzo surface. 
\end{predl}
\begin{proof}
By lemma~\ref{lemma_rational}, $X$ must be a rational surface.

We claim that $X$ is either $\mathbb F_0,\mathbb F_1$ or a blow-up of $\P^2$. Indeed, by Corollary~\ref{corollary_f012} $X$ can be blown down to $\mathbb F_0,\mathbb F_1$ or $\mathbb F_2$. If $X$ is not $\mathbb F_i$ then $X$ can be blown down to $X'$ which is a blow up of $\mathbb F_0,\mathbb F_1$ or $\mathbb F_2$ at one point.
Note that such $X$ can be always blown down to $\mathbb F_1$ or $\mathbb F_3$. The latter case is impossible by Corollary~\ref{corollary_f012} so $X'$ can be blown down to $\mathbb F_1$ and further to $\P^2$.

Now assume $X$ is a blow-up of $\P^2$. By Dolgachev, it suffices to show that this blow-up is in \emph{almost general position}. It means that there is a sequence of $s\le 8$ blow-ups
$$X=X_s\to X_{s-1}\to\ldots \to X_1\to X_0=\P^2$$
where $X_k\to X_{k-1}$ is a blow-up of one point $P_k\in X_{k-1}$ which does not lie on a smooth rational $(-2)$-curve on $X_{k-1}$.

Note that in our case $s\le 6$ because $X$ possesses a cyclic strong exceptional toric system. Denote by $L,E_1,\ldots,E_s$ the standard basis in $\Pic X$. Then all $(-2)$-classes in $\Pic X$ are $E_i-E_j$ and $\pm(L-E_i-E_j-E_k)$. Suppose that for some $l$ the point $P_l$ belongs to an irreducible $(-2)$-curve $R\subset X_{l-1}$ which is of the form $E_i-E_j$ or $L-E_i-E_j-E_k$. In the first case, consider the divisor $D=L_{jl}=L-E_j-E_l=(L-E_i)+(E_i-E_j-E_l)$. This divisor is a $(-1)$-class, so by Lemma~\ref{lemma_IF} $D\in I(X,A)$. It follows that $D$ is slo. On the other hand, $C=E_i-E_j-E_l$ is an irreducible $-3$-curve and it is a component of $D$. We will use a criterion from 
\cite[Theorem 1.2]{E}: 
for a slo effective divisor $D$ and its connected subdivisor $D'$ one has $p_a(D')\le 1+D\cdot D'$. This gives a contradiction: we have $p_a(C)=0$ since $C$ is a smooth rational curve but $C\cdot D=-2$.

In the second case $R=L-E_i-E_j-E_k$ we have that $C=L-E_i-E_j-E_k-E_l$ is an irreducible $-3$-curve. We can assume that $(i,j,k,l)=(1,2,3,4)$ and other blow-ups are after these four. Note that $X_4$ also has a cyclic strong exceptional toric system $(A_1,\ldots,A_7)$. We claim that there exists a cyclic segment $[k\ldots l]\subset [1\ldots 7]$ such that  $A_{k\ldots l}=2L-E_1-E_2-E_3-E_4$. It would follow then that the divisor $D=2L-E_1-E_2-E_3-E_4$ is slo. But $D=L+C$. As above, we have $p_a(C)=0$ and $D\cdot C=-2$, it gives a contradiction.

Now we prove the claim. Note that $D$ is a $0$-class in $\Pic X_4$. 
Actually we will check that all $0$-classes in  $\Pic X_4$ are of the form $A_{k\ldots l}$. There are totally  $5$ such classes: $2L-E_1-E_2-E_3-E_4$ and $L-E_i$, $i=1,2,3,4$.
By Proposition~\ref{prop_csadm}, the sequence $A_1^2,\ldots,A_7^2$ can be either 
$-1,-1,-2,-1,-2,-1,-1$ or $-1,-1,0,-2,-1,-2,-2$. 
In any of the two cases one can check that the following 5-tuples of divisors are pairwise different $0$-classes:
$A_{12},A_{123}, A_{234}, A_{2345}, A_{71}$ and $A_{12}, A_{712}, A_{6712}, A_3, A_{34}$ respectively.
\end{proof}

\begin{predl}
\label{prop_classification}
Toric systems from Table~\ref{table_yes} are cyclic strong exceptional. Therefore weak del Pezzo surfaces from Table~\ref{table_yes} admit cyclic strong exceptional toric systems. Weak del Pezzo surfaces $X$ from Table~\ref{table_no} do not admit cyclic strong exceptional toric systems. 
\end{predl}
\begin{table}[h]
\begin{center}
\caption{Cyclic strong exceptional toric systems}
\label{table_yes}
\begin{tabular}{|p{2cm}|p{2cm}|p{10cm}|}
 \hline
 degree & type & toric system \\
 \hline
 $9$ & $\P^2$ & $L,L,L$  \\
 \hline
 $8$ & $\mathbb F_0$ & $H_1,H_2,H_1,H_2$  \\
 \hline
 $8$ & $\mathbb F_1$ & $L_1,E_1,L_{1},L$  \\
 \hline
 $8$ & $\mathbb F_2$ & $F,S-F,F,S-F$ (where $F^2=0, S^2=2, FS=0$)  \\
 \hline
 $7$ & any & $L_{1},E_1,L_{12},E_2,L_{2}$  \\
 \hline
 $6$ & any & $L_{13},E_1,L_{12},E_2,L_{23},E_3$  \\
 \hline
 $5$ &  $\emptyset$ & ${L_{134}},E_4,{E_1-E_4},L_{12},E_2,L_{23},E_3$\\
 \hline
 $5$ &  $A_1$ &  the above \\  \hline
 $5$ &  $2A_1$ & the above \\  \hline
 $5$ &  $A_2$ & the above  \\  \hline
 $5$ &  $A_1+A_2$ & the above \\  \hline
 $4$ & $\emptyset$ & ${L_{134}},E_4,{E_1-E_4},L_{12},{E_2-E_5},E_5,{L_{235}},E_3$  \\
 \hline
 $4$ & $A_1$ & the above  \\  \hline
 $4$ & $2A_1,9$ & the above  \\  \hline
 $4$ & $2A_1,8$ & the above  \\  \hline
 $4$ & $A_2$ & the above  \\  \hline
 $4$ & $3A_1$ & the above  \\  \hline
 $4$ & $A_1+A_2$ & the above  \\  \hline
 $4$ & $A_3,4$ & the above  \\  \hline
 $4$ & $4A_1$ & the above  \\  \hline
 $4$ & $2A_1+A_2$ & the above  \\  \hline
 $4$ & $A_1+A_3$ & the above  \\  \hline
 $4$ & $2A_1+A_3$ & the above  \\  \hline
 $3$ & $\emptyset$ & ${E_2-E_4,L_{125}},E_5,{E_1-E_5,L_{136}},E_6,{E_3-E_6,L_{234}},E_4$   \\  \hline
 $3$ & $A_1$ & the above  \\  \hline
 $3$ & $2A_1$ & the above  \\  \hline
 $3$ & $A_2$ & the above  \\  \hline
 $3$ & $3A_1$ & the above  \\  \hline
 $3$ & $A_1+A_2$ & the above  \\  \hline
 $3$ & $4A_1$ & the above  \\  \hline
 $3$ & $2A_1+A_2$ & the above  \\  \hline
 $3$ & $2A_2$ & the above  \\  \hline
 $3$ & $A_1+2A_2$ & the above  \\  \hline
 $3$ & $3A_2$ & the above  \\  \hline
\end{tabular}
\end{center}
\end{table}

\begin{table}[h]
\begin{center}
\caption{Surfaces  with no cyclic strong exceptional toric systems}
\label{table_no}
\begin{tabular}{|p{2cm}|p{4cm}||p{3cm}|p{3cm}|}
 \hline
 $\deg(X)$ & type of $X$ & type of $X'$ & $P\in X'$ \\
 \hline
 $5$ & $A_3$ & & \\  \hline
 $5$ & $A_4$ & &\\  \hline
 $4$ & $A_3,5$ & $A_3$ & general \\  \hline
 $4$ & $A_4$ & $A_4$ & general \\  \hline
 $4$ & $D_4$ & $A_3$ & general on $L_{12}$\\  \hline
 $4$ & $D_5$ & $A_4$ & general on $E_4$ \\  \hline
 $3$ & $A_3$ & $A_3,5$ & general \\  \hline
 $3$ & $A_1+A_3$ & $A_3,5$ & general on $E_1$ \\  \hline

 $3$ & $A_4$ & $A_4$ & general \\  \hline
 $3$ & $D_4$ & $D_4$ & general \\ \hline
 $3$ & $2A_1+A_3$ & $A_3,5$ & $E_1\cap Q$  \\  \hline
 $3$ & $A_1+A_4$ & $A_4$ & general on $Q$ \\  \hline
 $3$ & $A_5$ & $A_4$ & general on $E_5$ \\  \hline
 $3$ & $D_5$ & $D_5$ & general \\  \hline
 $3$ & $A_1+A_5$ & $A_4$ & $E_5\cap Q$ \\  \hline
 $3$ & $E_6$ & $D_5$ & general on $E_5$\\  \hline
\end{tabular}
\end{center}
\end{table}

\begin{proof}
First note that the sequences from Table~\ref{table_yes} are toric systems. Indeed, it is easy to see that any of them is a standard augmentation. Next one needs to check that they are cyclic strong exceptional. Any toric system $A$ in the Table is of the first kind. According to Theorem~\ref{theorem_checkonlyminustwo}, we have to check that all divisors of the form $A_{k\ldots l}$ where $[k\ldots l]\subset [1\ldots n]$ is a cyclic segment such that $A_k^2=A_{k+1}^2=\ldots=A_l^2=-2$ are not effective nor anti-effective. 
We see that for degrees $\ge 6$ there are no such divisors. For degrees $5,4,3$ we have the following divisors:
\begin{align*}
\deg(X)=5\colon & L_{134},E_1-E_4; \\
\deg(X)=4\colon & L_{134},E_1-E_4, E_2-E_5, L_{235};\\
\deg(X)=3\colon & E_2-E_4,L_{125},L_{145},E_1-E_5,L_{136},L_{356},E_3-E_6,L_{234},L_{246}.
\end{align*}
We refer to Tables~\ref{table_d5}, \ref{table_d4}, \ref{table_d3}, \ref{table_d3bis}  for $R^{\rm eff}$. Using these Tables one has to check that the divisors listed above are not in $\pm R^{\rm eff}$.
It may be useful to note that $R^{\rm eff}(X_{5,A_1+A_2})$ is maximal among all surfaces of degree $5$ from Table~\ref{table_yes}, therefore we need to check the above claim  only for $X_{5,A_1+A_3}$. Likewise, for degree $4$ the maximal $R^{\rm eff}$ is owned by the surfaces $X_{4,2A_1+A_3}$ and $X_{4,2A_1+A_2}$. For degree $3$ the surfaces with maximal $R^{\rm eff}$ in Table~\ref{table_yes} are $X_{3,3A_2}$ and $X_{3,4A_1}$.

Now we prove the negative part of the Proposition.
First, we check that the surfaces $X_{5,A_3}$ and $X_{5,A_4}$ have no cyclic strong exceptional toric system. Assume the contrary, $A$ is such system. Then up to a shift and a symmetry, $A^2$ is $(-1,-1,-2,-1,-2,-1,-1)$ or $(-1,-1,0,-2,-1,-2,-2)$. In both cases we see that there exist two $(-2)$-classes $D_1=A_i$ and $D_2=A_{i+2}$ which are slo and such that $D_1\cdot D_2=0$. On the other hand, we have $R^{\rm slo}(X_{5,A_3})=\pm \{L_{123},L_{124},L_{134},L_{234}\}$ and $R^{\rm slo}(X_{5,A_4})=\emptyset$. It easy to see that such $D_1$ and $D_2$ cannot exist.

We claim that all other surfaces $X$ from Table~\ref{table_no} are blow-ups of $X_{5,A_3}$ or $X_{5,A_4}$. It would follow then from Corollary~\ref{cor_cyclicblowup} that they admit no cyclic strong exceptional toric systems. In the right two columns of Table~\ref{table_no} we present a surface $X'$ and a point $P\in X'$ such that the blow-up of $X'$ at $P$  has the same type as $X$. We refer to Tables~\ref{table_d5} and~\ref{table_d4} or to diagrams from \cite{CT}, Propositions 6.1 and 8.4, for the verifying.
\end{proof}

\section{Applications to dimension of $\D^b(\coh X)$}

The paper \cite{BF} by Matthew Ballard and David Favero initiated the study of the  relation between the dimension of triangulated categories in the sense of Rafael Rouquier \cite{Rou} and full cyclic strong exceptional collections. Recall an object 
$T$ of a triangulated category $\TT$ is called a classical generator if $\TT$ is the smallest full strict triangulated idempotent closed subcategory in $\TT$ containing $T$. Generator $T$ is said to have generation time $n$ if $n$ the minimal number of cones required to construct any object in $\TT$ starting from $T$. Rouquier defined dimension of a triangulated category $\TT$ as the minimal possible generation time for all  classical generators in~$\TT$. 
\begin{conjecture}[Orlov]
For a smooth projective variety $X$ one has 
 $$\dim \D^b(\coh(X))=\dim X.$$
\end{conjecture}
The bound $\D^b(\coh(X))\ge \dim X$ was proven by Rouquier in \cite{Rou}, while the equality is known only for some special varieties. A useful tool for proving the conjecture was found by Ballard and Favero in \cite{BF}:

\begin{theorem}[See {\cite[Theorem 3.4]{BF}}]
\label{theorem_BF}
Let $X$ be a smooth proper variety over a perfect field $\k$, let $T\in \D^b(\coh(X))$ be a classical generator such that $\Hom^i(T,T)=0$ for $i\ne 0$. Denote 
$$i_0=\max\{i\mid \Hom^i(T,T\otimes\O_X(-K_X))\ne 0\}.$$
Then the generation time of $T$ is $\dim X+i_0$.
\end{theorem}

\begin{corollary}
Let $X$ be a smooth proper variety over a perfect field $\k$, let $(\EE_1, \ldots,\EE_n)$ be a full strong exceptional collection on $X$. Suppose
the generation time of the generator $T=\oplus_i \EE_i$ is equal to $\dim X$. Then the collection $(\EE_1, \ldots,\EE_n)$ is cyclic strong exceptional.
\end{corollary}
 
We are able to prove the converse statement for collections of line bundles on surfaces. 

\begin{predl}
\label{prop_timetwo}
Let $X$ be a smooth projective surface with a full cyclic strong exceptional collection 
$$(\O_X(D_1),\ldots,\O_X(D_n))$$ 
of line bundles. Then the generator $T=\oplus_{i=1}^n \O_X(D_i)$ of the category $\D^b(\coh (X))$ has generation time two. Therefore, dimension of $\D^b(\coh(X))$ is equal to $2$.
\end{predl}
\begin{proof}
By Proposition~\ref{prop_cyclicwdp} and Proposition~\ref{prop_csadm}, $X$ is a weak del Pezzo surface and $\deg X\ge 3$. By Theorem~\ref{theorem_BF}, we have to check that $H^i(D-K_X)=0$ for any $i>0$ and any $D$ of the form $D_k-D_l$. There are three possible cases:
$$
D=\begin{cases}
D_k-D_l, k<l,\\
0,\\
D_k-D_l, k>l.
\end{cases}$$
For $D=D_k-D_l$ where $k<l$ we have 
$$H^i(D_k-D_l-K_X)=\Hom^i(D_l,D_k-K_X)=0$$
for any $i>0$ because the collection
\begin{multline*}
(\O_X(D_l),\O_X(D_{l+1}),\ldots,\O_X(D_n),\\ \O_X(D_1-K_X),\O_X(D_{2}-K_X),\ldots,\O_X(D_k-K_X),\ldots, \O_X(D_{l-1}-K_X))
\end{multline*}
is strong exceptional by assumption.

For $D=D_k-D_l$ where $k>l$ we borrow the arguments from Lemma 3.9 of \cite{BF}. By [Do, Theorem 8.3.2], the linear system $|-K_X|$ has no base points. By Bertini's Theorem, the general divisor $Z\in|-K_X|$ is a smooth irreducible curve. By the adjunction formula, $Z$ is a curve of genus $1$.
Consider the standard exact sequence
\begin{equation}
\label{eq_DD}
0\to \O_X(D)\to \O_X(D-K_X)\to \O_Z(D-K_X)\to 0.
\end{equation}
One has $H^i(\O_X(D))=0$ for $i>0$ by strong exceptionality of the collection 
$(\O_X(D_1),\ldots,\O_X(D_n))$. Also, 
$$Z\cdot (D-K_X)=K_X(K_X-D)=d-DK_X=d+2+D^2\ge d\ge 3$$
because $D$ is a slo divisor and hence $D^2\ge -2$.
It follows that 
$$H^i(X,\O_Z(D-K_X))=H^i(Z,\LL)$$
where $\LL$ is a line bundle on $Z$ of degree $\ge 3$. Since $Z$ has genus $1$, we get $H^i(Z,\LL)=0$ for $i>0$.
The long exact sequence of cohomology associated to (\ref{eq_DD}) provides that $H^i(X,\O_X(D-K_X))=0$ for $i>0$.

The above arguments work also for $D=0$.
\end{proof}

\begin{corollary}
Dimension of $\D^b(\coh(X))$ is equal to $2$ for any weak del Pezzo surface~$X$ from Table~\ref{table_yes}.
\end{corollary}

\appendix
\section{The counterexample searching algorithm and the effectiveness test}

\subsection{The counterexample searching algorithm and its implementation}

This section describes the algorithms and their implementation.


The counterexample searching program consists of two separate parts.\\

The preparation, or pre-computation, is done in Mathematica \cite{Wo}: the inputs are the $(-2)$-curves for all root subsystems, 
all the strongly admissible sequences of Types III--VI (and some of Type II),
and an initial toric system for each such sequence 
(which generates all toric systems with the same self-intersection sequence via the Weyl group action by Theorem \ref{Weyl trans}). 

The Mathematica program does the following things automatically:
\begin{enumerate}
\item For each root subsystem, verify that the intersection products between the corresponding $(-2)$-curves match the ADE description (the Dynkin diagram);
\item For each strongly admissible sequence in the input, verify that it matches the intersection numbers between terms in the corresponding initial toric system;
\item For each root subsystem, generate all effective $(-2)$-classes and all  $(-1)$-curves;
\item For each strongly admissible sequence, generate all $r$-classes ($r\le-1$) arising from sums of all terms in cyclic and non-cyclic intervals in the initial toric system;
\item Prepare input data for the GAP 3 program.
\end{enumerate}

The reasons for working with Mathematica are its easiness to interact with and the ability to save frequently used functions into notebooks so that they can be conveniently called; it is ideal for plentiful small tasks that are not heavy on computation.\\

The second part, which involves Weyl group elements and heavy but boring computations, is done in GAP 3 \cite{S+} using the package CHEVIE \cite{CHEVIE}.
We find GAP 3 the ideal environment to work with individual Weyl group elements, as CHEVIE is able to go through all $696729600$ elements of $W(E_8)$ within merely 9 minutes on a personal computer.

Let $A$ be a numerical exceptional toric system. Denote by $R^c(A)$ the cyclic $(-2)$-classes in toric system $A$. Denote by $R(A)$ the non-cyclic $(-2)$-classes. Denote by $T_0$ the initial input toric system. Denote by $w\in W$ an element of the Weyl group of $X$.

The counter examples searching algorithms proceed as follows:\begin{enumerate}
\item If one of $(-2)$-class of $T_i$ in $R^c(A)$ is in $-R^{\rm eff}$, then $i:=i+1$, $w:=w_i, T_i=:w(T_{i-1})$,  perform $(1)$.
    \item Otherwise, if one of the $(-2)$-classes of $T_i$ in $R(A)$ is in $R^{\rm eff}$ or $-R^{\rm eff}$, then $i:=i+1$, $w:=w_i, T_i=:w(T_{i-1})$,  perform $(1)$.
    \item Otherwise, if one of the $(-1)$-classes in $I^F(X,A)$ is in $I^{\rm irr}$, then $i:=i+1$, $w:=w_i, T_i=:w(T_{i-1})$, perform $(1)$,
    \item Otherwise, if one of the $r$-classes $A_{k,\ldots,l}$ with $r\leq -3$, such that $-A_{k,\ldots,l}$ is effective divisor, then $i:=i+1$, $w:=w_i, T_i=:w(T_{i-1})$, perform $(1)$.
    \item Otherwise, print "COUNTER EXAMPLE" and output $T$. Then $i:=i+1$, $w:=w_i, T_i=:w(T_{i-1})$, perform $(1)$.
\end{enumerate}

We implement this algorithm using the computer in the following way: for each Weyl group element $w$ and each strongly admissible sequence, transform the initial toric system $T_0$ by the Weyl group element. Then, for each root subsystem, apply four tests to check whether the transformed toric system $w(T_0)$ is a counterexample to Hille and Perling's conjecture, i.e. whether it is strongly exceptional and does not come from a standard augmentation. 
Tests are applied in the following order: 
\begin{enumerate}
\item no cyclic $(-2)$-class is anti-effective;
\item no non-cyclic $(-2)$-class is effective or anti-effective;
\item all non-cyclic $(-1)$-classes are reducible (and slo);
\item no (cyclic) $r$-class with $r\le-3$ is anti-effective.
\end{enumerate}

Tests (1)--(3) are done by looking up true/false tables generated using Mathematica, and Test (4) is done using the effectiveness test detailed in the next subsection.

In the degree $\ge2$ case, all $(-1)$-classes are slo, but when the degree is $1$, there can be reducible $(-1)$-classes that are not lo, though the irreducible ones must still be slo. 
In the degree 1 case, there is a 1-1 correspondence between the $(-1)$-classes and the $(-2)$-classes given by $C\mapsto C+K$, and $C$ is slo, or equivalently lo, exactly when $C+K$ is not effective (cf. Table \ref{table_loslo}). 
We no longer need the extra computation required in the degree 2 case to find out the permutation action on the $(-1)$-classes. (There are 126 $(-2)$-classes and 56 $(-1)$-classes, and CHEVIE represents Weyl group elements as permutations on the $(-2)$-classes).

Moreover, in any toric system on a degree 1 weak del Pezzo surface, there is a 1-1 correspondence between non-cyclic $(-1)$-classes and cyclic $(-2)$-classes given by $C\mapsto -K-C$. If Test (3) passes, then all $(-1)$-classes $C$ are slo, so no $C+K$ is effective, i.e. no cyclic $(-2)$-class $-K-C$ is anti-effective, so Test (1) must also pass. Thus we see that Test (3) rendered Test (1) redundant in the degree 1 case.\\

If Test (1), (2) or (4) fails, the toric system is not strongly exceptional, while if Test (3) fails, it comes from a standard augmentation. Therefore, for a toric system to be a counterexample, all four tests must pass. Conversely, if all tests pass, the following lemmas guarantee that the toric system is indeed a counterexample if $d\ge2$. (See following theorem $18.2$ and corollary $18.3$ )
When a essentially new toric system passes all tests, the program prints it out on the screen and records it in a list which is output in the end. No toric system in degree 1 that we examined passed all tests, so we verified Hille and Perling's conjecture for these toric systems (including those of Types III--VI) on degree 1 weak del Pezzo surfaces.

\begin{remark}
We can modify the counter example searching algorithms above by not testing whether the non-cyclic $(-2)$-classes is  effective to make it work for testing Hille-Perling's conjecture for all exceptional toric systems(not necessarily strong exceptional) of the second kind.
\end{remark}

\begin{theorem}
Let $T$ be a second kind toric system on a weak del Pezzo surface of degree $1\leq d\leq 2$ of type III-VI with given initial value $T=T_0$. Running the algorithm above, $T$ will be output if and only if $T$ is a strong exceptional toric system of maximal length which is not an augmentation in any sense. In particular, if $\deg(X)=2$, $T$ is output iff $T$ is a full strong exceptional toric system which is not an augmentation in any sense
\end{theorem}

\begin{proof}
It follows from Theorem \ref{theorem_checkonlyminustwo} noting that the strong exceptional collection of maximal length is full on weak del Pezzo surfaces of $\deg(X)\geq 2$.
\end{proof}

\begin{corollary}
Assume that after running the algorithm, $T$ is not output, then the full strong exceptional toric system of second kind of type III-VI is an augmentation in the mild sense
\end{corollary}

\begin{proof}
By Proposition \ref{Weyl trans}, the Weyl group $W$ acts on $TS_{a}(X)$ transitively. Then any full strong exceptional toric system of $A^2=a$ is of the form $wT_0$, where $w\in W$ and $T_0$ is the initial toric system of $T^2=a$. By theorem $18.2$, $A$ is a full strong exceptional toric system which is an augmentation in mild sense.
\end{proof}

Now, we apply the algorithm above to toric systems on weak del Pezzo surfaces of degree $2$ and degree $1$. Note that, the first kind toric systems do not exist on weak del Pezzo surfaces of degree less than $3$ and all strong admissible sequences which are not cyclic strong can be obtained by augmentations from $(1,1,1)$ on $\mathbb{P}^2$. 

\subsubsection{Degree $2$, type III-VI}

Up to symmetry, $A^2$ of type III-VI toric systems on weak del Pezzo surface of degree $2$ are one of the following sequences:
\begin{enumerate}
    \item $(-2,-2,-1,-2,0,-2,-2,-2,-1,-4)$
    \item $(-2,-1,-1,0,-2,-2,-2,-2,-1,-5)$
    \item $(-2,0,1,-2,-2,-2,-2,-2,-1,-6)$
    \item $(-1,-2,-2,-2,0,0,-2,-2,-1,-6)$
    \item $(-1,-2,-2,-2,-2,0,0,-2,-1,-6)$
    \item $(-1,-2,-2,-2,-2,-2,0,0,-1,-6)$
    \item $(-1, -2, -2, -2, -2, -2, -2, 0, 1, -6)$
\end{enumerate}

We give the corresponding initial toric systems $T_0$ of type III-VI as follows:
\begin{enumerate}
    \item $E_2-E_3,L_{127},E_7,E_{17},L-E_1,L_{234},E_4-E_5,E_5-E_6,E_6,E_3-E_4-E_5-E_6$
    \item $E_2-E_3,L_{12},E_1,L-E_1,L_{234},E_4-E_5,E_5-E_6,E_6-E_7,E_7,E_3-E_4-E_5-E_6-E_7$
    \item $E_1-E_2,L-E_1,L,L_{123},E_3-E_4,E_4-E_5,E_5-E_6,E_6-E_7,E_7,E_2-E_3-E_4-E_5-E_6-E_7$
    \item $E_7,E_5-E_7,E_4-E_5,E_3-E_4,L-E_3,L-E_1,E_1-E_2,E_2-E_6,E_6,L_{1234567}$
    \item $E_7,E_5-E_7,E_4-E_5,E_3-E_4,E_2-E_3,L-E_2,L-E_1,E_1-E_6,E_6,L_{1234567}$
    \item $E_7,E_5-E_7,E_4-E_5,E_3-E_4,E_2-E_3,E_1-E_2,L-E_1,L-E_6,E_6,L_{1234567}$
    \item $E_7, E_6-E_7, E_5-E_6, E_4-E_5, E_3-E_4, E_2-E_3, E_1-E_2, L-E_1, L, L_{1234567}$
\end{enumerate}

Applying the algorithm to $T_0$ for $(1)-(7)$, there is no $T_i$ output, by corollary $18.3$, all the strong exceptional toric systems of the second kind of type III-VI is an augmentation in the mild sense.

\subsubsection{Degree $1$, type III-VI}

Up to symmetry, $A^2$ of type III-VI toric system on weak del Pezzo surface of degree $1$ is one of the following sequences:
\begin{enumerate}
    \item $(-2,-2,-1,-2,0,-2,-2,-2,-2,-1,-5)$
    \item $(-2,-1,-1,0,-2,-2,-2,-2,-2,-1,-6)$
    \item $(-2,0,1,-2,-2,-2,-2,-2,-2,-1,-7)$
    \item $(-1,-2,-2,-2,-2,0,0,-2,-2,-1,-7)$
    \item $(-1,-2,-2,-2,0,0,-2,-2,-2,-1,-7)$
    \item $(-1,-2,-2,-2,-2,-2,0,0,-2,-1,-7)$
    \item $(-1,-2,-2,-2,-2,-2,-2,0,0,-1,-7)$
    \item $(-1, -2, -2, -2, -2, -2, -2, -2, 0, 1, -7)$
\end{enumerate}

We give the correspondent initial toric systems $T_0$ as follows:
\begin{enumerate}
    \item $E_2-E_3,L_{127},E_7,L_{17},E_1-E_7,L-E_1,L_{234},E_4-E_5,E_5-E_6,E_6-E_8,E_8,E_3-E_4-E_5-E_6-E_8$
    \item $E_2-E_3,L_{12},E_1,L-E_1,L_{234},E_4-E_5,E_6-E_7,E_7-E_8,E_8,E_3-E_4-E_5-E_6-E_7-E_8$
    \item $E_1-E_2,L-E_1,L,L_{123},E_3-E_4,E_4-E_5,E_5-E_6,E_6-E_7,E_7-E_8,E_8,E_2-E_3-E_4-E_5-E_6-E_7-E_8$
    \item $E_8,E_7-E_8,E_5-E_7,E_4-E_5,E_3-E_4,L-E_3,L-E_1,E_1-E_2,E_2-E_6,E_6,L_{12345678}$
    \item $E_7,E_5-E_7,E_4-E_5,E_3-E_4,L-E_3,L-E_1,E_1-E_2,E_2-E_6,E_6-E_8,E_8,L_{12345678}$
    \item $E_8,E_7-E_8,E_5-E_7,E_4-E_5,E_3-E_4,E_2-E_3,L-E_2,L-E_1,E_1-E_6,E_6,L_{12345678}$
    \item $E_8,E_7-E_8,E_5-E_7,E_4-E_5,E_3-E_4,E_2-E_3,E_1-E_2,L-E_1,L-E_6,E_6,L_{12345678}$
    \item $E_8,E_7-E_8, E_6-E_7, E_5-E_6, E_4-E_5, E_3-E_4, E_2-E_3, E_1-E_2, L-E_1, L, L_{12345678}$
\end{enumerate}

Applying the algorithm  to $T_0$ for $(1)-(8)$, there is no $T_i$ output, so by corollary $18.3$ again, all the strong exceptional toric systems of the second kind of type III-VI are augmentations in the mild sense.

For the degree 2 case (root system $E_7$, Weyl group order 2903040, 46 different weak del Pezzo surfaces), the GAP 3 program took about 1.5 hours to look for counterexamples for the 15 strongly admissible sequences. Among these sequences, we found only one sequence which admits counterexamples, detailed in the previous section.

The degree 1 program (root system $E_8$, Weyl group order 696729600, 76 different weak del Pezzo surfaces) takes considerably more time to run (examining 17 strongly admissible sequences). The GAP 3 program for this case is designed to run for each root subsystem separately, as we can then run multiple instances of GAP 3 on a single computer. For each of the 76 root subsystems, the runtime ranges from 2 to 10 hours.

\subsection{The effectiveness test}

\subsubsection{General effectiveness test}

We first present some lemmas.

\begin{lemma}\label{eff equiv}
Let $D$ be a divisor and let $C$ be an (irreducible reduced) curve on a surface $X$.
\begin{enumerate}
\item If $D\cdot C<0$, then $D$ is effective iff $D-C$ is.
\item In fact, if $D\cdot C=-b<0$ and $C^2=-c<0$, then $D$ is effective iff $D-\lceil b/c\rceil C$ is. 
\item Moreover, if $C_1,\dots,C_k$ are distinct curves on $X$ with $D\cdot C_i=-b_i<0$ and $(C_i)^2=-c_i<0$, then $D$ is effective iff $D-\sum_{i=1}^k \lceil b_i/c_i\rceil C_i$ is.
\end{enumerate}
\end{lemma}
 
\begin{proof}
We first prove Claim (1). If $D-C$ is effective, clearly, $D=(D-C)+C$ is effective. Conversely, if $D$ is effective, it can be written as a sum of curves $\sum_{i=1}^l C_i$. Since $D\cdot C<0$, $C_{i_0}\cdot C<0$ for some $1\le i_0\le l$, which implies that $C_{i_0}=C$. Therefore $D-C=\sum_{1\le i\le l,i\neq i_0}C_i$ is effective.

For Claims (2) and (3), the ``if" part is also clear. If $D\cdot C=-b$ and $C^2=-c$, we have $(D-iC)\cdot C=-b+ic<0$ for $i=0,1,\dots,\lceil b/c \rceil-1$. If $D$ is effective, by applying the first claim repeatedly, we see that $D-\lceil b/c\rceil C$ is effective, thus proving Claim (2).

Now we prove the ``only if" part of Claim (3). Let  $D_{k_0}:=D-\sum_{i=1}^{k_0}\lceil b_i/c_i \rceil C_i$. Assuming $D=D_0$ is effective, we shall show inductively that $D_1,D_2,\dots,D_k$ are all effective, thus proving Claim (3). Indeed, since 
\[D_{k_0}\cdot C_{k_0+1}=D\cdot C_{k_0+1}-\sum_{i=1}^{k_0}\lceil b_i/c_i\rceil C_i\cdot C_{k_0+1}\le-b_{k_0+1}\] because $C_i\cdot C_j\ge0$ for $i\neq j$, and since $D_{k_0}$ is effective, we see that $D_{k_0}-\lceil b_{k_0+1}/c_{k_0+1}\rceil C_{k_0+1}=D_{k_0+1}$ is effective by Claim (2).
\end{proof}

\begin{lemma}\label{D.K=0}
Let $D$ be a divisor on a weak del Pezzo surface $X$ with $D\cdot K=0$. Then $D$ is effective iff $D$ is a sum of $(-2)$-curves.
\end{lemma}
\begin{proof}
The ``if" part is clear. Conversely, if $D$ is effective, it can be written as $\sum_{i=1}^k C_i$ for not necessarily distinct curves $C_i$. Since $-K$ is nef on the weak del Pezzo surface $X$, we have $K\cdot C_i\le 0$ for all $i$; since $K\cdot D=0$, we conclude that $K\cdot C_i=0$ for all $i$. Since the intersection form is negative definite on $\{D\mid D\cdot K=0\}$, each $C_i\neq0$ must be a negative curve. By Lemma \ref{neg curve}, each $C_i$ is a $(-2)$-curve, since $K\cdot C=-1\neq0$ for $(-1)$-curves $C$.
\end{proof}

\begin{lemma}\label{nef crit}
Let $D$ be a divisor on weak del Pezzo surface $X$ of degree $K_X^2\leq 7$. Assume that $D.C\geq 0$ for each $(-1),(-2)$-curve on $X$. Then $D$ is nef divisor.
\end{lemma}

\begin{proof}
The effective cone of $X$ 
is generated by $(-1)$ and $(-2)$-curves. By assumption, $D$ has non-negative intersection with any non-negative real combinations of $(-1)$ and $(-2)$-curves. Thus $D$ has non-negative intersection with any effective divisor on $X$ and hence $D.C\geq 0$ for any irreducible curve on $X$, by definition, $D$ is nef.
\end{proof}

\begin{lemma} \label{nef=>eff}
Let $D$ be a nef divisor on weak del pezzo surface $X$, then $D$ is effective. 
\end{lemma}

\begin{proof}
We claim that $h^2(D)=0$, otherwise $h^2(D)=h^0(K-D)\neq 0$, then $(K_X-D)(-K_X)\geq 0$ since $-K_X$ is nef. Thus $(K_X-D)(K_X)\leq 0$, but $D$ is nef and $-K_X$ is effective, hence $D.(-K_X)\geq 0$. Thus $K_X^2-K_X.D\geq 1$. Hence $h^2(D)=0$. Thus, we have $\chi(D)=h^0(D)-h^1(D)=\frac{1}{2}D(D-K_X)+1$. Hence, $h^0(D)\geq\frac{1}{2}D(D-K_X)+1$. Again, since $D$ is nef, by the Nakai-Moishezon criterion, $D^2\geq 0$. Thus, $\frac{1}{2}(D^2-K_X.D)\geq 0$, hence $h^0(D)\geq 1$, thus $D$ is effective. 
\end{proof}

Using the above lemmas, we have the following algorithm to determine whether a divisor $D$ on a weak del Pezzo surface $X$ is effective.

(1) Test whether $D\cdot K\le0$: if no, $D$ cannot be effective since $-K$ is nef; if yes, proceed to the next step.

(2) Test whether $D\cdot K<0$: if no, then $D\cdot K=0$ and $D$ is effective or not according to whether it is a nonnegative integral linear combination of the (linearly independent) $(-2)$-curves (Lemma \ref{D.K=0}); if yes, proceed to the next step.

(3) Test whether $D\cdot C\ge 0$ for all $(-1)$-curves and $(-2)$-curves $C$: if no, replace $D$ by $D-\sum_C \lceil \frac{D\cdot C}{C\cdot C} \rceil C$ (see Lemma \ref{eff equiv}(3)) and return to Step (1); if yes, $D$ is nef by Lemma \ref{nef crit} and hence effective by Lemma \ref{nef=>eff}, and we conclude that the original $D$ is also effective.\\

The algorithm is guaranteed to terminate because of the following. Looking at Step (3), we see that we shall never return to the same $D$ during the algorithm, since a sum of curves cannot be linearly equivalent to 0. Moreover, each time we subtract a $(-2)$-curve $R$, we have $D\cdot R\le -1$, so $(D-R)^2=D^2-2D\cdot R+R^2\ge D^2-2(-1)+(-2)=D^2$, i.e. the new divisor $D'=D-R$ has larger or equal $D'^2$. But for each $k$ and $r$, there are only finitely many divisors $D$ satisfying $D\cdot K=k$ and $D^2\ge r$, so either $D\cdot K$ must eventually increase (by subtracting a $(-1)$-curve from $D$), or we must end up with a nef divisor and the algorithm terminates. It is clear that $D\cdot K$ can only increase finitely many times (up to 0), so the algorithm always terminates.

\begin{remark}
The integrality of the linear combination in Step (2) is crucial; on some weak del Pezzo surfaces there exist divisors which are nonnegative linear combinations of $(-2)$-curves but which are not effective, though such does not exist on genuine del Pezzo surfaces.
\end{remark}

\subsubsection{The simplified effectiveness test for conditionally effective anti-classes}

In the counterexample searching algorithm, we actually employ a simplified version of the effectiveness test which only work for the so-called ``conditionally effective anti-classes''. This is based on the observation that the anti-$r$-classes examined in Test (3) of the algorithm are all conditionally effective. In the simplified test, we no longer need to compute intersections with $(-1)$-classes (can be as many as 183), which saves a lot of work and helps reduce the runtime of the counterexample searching program.

\begin{definition}

\begin{enumerate}
\item A divisor $D$ on a weak del Pezzo surface is called absolutely effective if it lies in the cone generated by $(-1)$-classes. 
\item $D$ is conditionally effective if it is not absolutely effective. 
\item $D$ is called an anti-class if $D^2-D\cdot K=-2$, i.e. if $-D$ is a class.
\end{enumerate}
\end{definition}
\begin{remark}
Conditionally effectiveness and absolutely effectiveness is a property stable under the Weyl group action, since the action preserves the set of $(-1)$-classes.
\end{remark}

We now present some lemmas.

\begin{lemma} [See {\cite[Lemma 4.1]{HL}}]
The effective monoid of a rational surface with an effective anticanonical divisor and Picard rank $\ge3$ is generated by the set of prime divisors with negative self-intersections together with $-K$.
\end{lemma}

Since any weak del Pezzo surface $X$ is rational with nef hence effective anticanonical divisor and satisfies ${\rm rk\,Pic}(X)+d=10$, we have
\begin{lemma} \label{eff monoid}
The effective monoid of a weak del Pezzo surface of degree $\le7$ is generated by the $(-1)$-curves and $(-2)$-curves together with $-K$. 

(It can be shown that the generator $-K$ is not needed for weak del Pezzo surfaces except the genuine del Pezzo surface of $d=1$.)
\end{lemma}

\begin{remark}

On a genuine del Pezzo surface $X'$, there are no $(-2)$-curves and the $(-1)$-curves are exactly all the $(-1)$-classes. Therefore the $(-1)$-classes together with $-K$ generate the effective monoid, and this set of generators is clearly invariant under the Weyl group, so the effective monoid of $X'$ is invariant under the Weyl group action. The effective cone, generated as a cone by the $(-1)$-curves without $-K$, is exactly the set of absolutely effective divisors. Lemma \ref{rat eff} together with a standard argument shows that the effective monoid coincides with the effective cone.

For a weak del Pezzo surface $X$, it is well-known that all $(-1)$-classes are still effective, so the effective monoid can be generated by all $(-1)$-classes and $(-2)$-curves together with $-K$.
Recall that for any genuine del Pezzo surface $X'$ with the same degree as $X$, we can isometrically identify $\Pic(X)$ with $\Pic(X')$ preserving $K$.
Under such identification, the effective monoid of $X$ contains that of $X'$ with $(-2)$-curves as additional generators. Since the absolutely effective divisors coincide with the effective monoid of $X'$, they are effective on all weak del Pezzo surfaces. 
It is also true that a conditionally effective divisor $D$ with $D\cdot K\le0$ is effective on some weak del Pezzo surface. 
\end{remark}

\begin{lemma} \label{rat eff}
Let $X$ be a genuine del Pezzo surface with $d\le7$, let $D$ be a divisor on $X$, and let $k\in\Z_{>0}$. If $kD$ is effective, then $D$ is also.
\end{lemma}
Colloquially, the lemma says that ``rationally effective" divisors on a genuine del Pezzo surface are effective. Intuitively, it says that there are no holes in the effective cone, i.e. the effective monoid is exactly the intersection of the effective cone with the Picard group. 
\begin{proof}
We proceed by induction on $-K\cdot D$. If $-K\cdot D\le0$, then since $-K$ is ample, neither $D$ nor $kD$ can be effective, unless $D=0$. 

Now suppose that $-K\cdot D>0$. The effective cone of $X$ is generated by the $(-1)$-curves, thus if $D\cdot C\ge0$ for all $(-1)$-classes $C$, then $D$ is nef, hence effective by Lemma \ref{nef=>eff}. If $D\cdot C<0$ for some $(-1)$ class $C$, then $kD\cdot C\le -k$. If $kD$ is effective, then $kD-kC=k(D-C)$ is also by Lemma \ref{eff equiv}(3).

Since $-K\cdot(D-C)=-K\cdot D+K\cdot C=-K\cdot D-1<-K\cdot D$, we conclude that $D-C$ is effective by the induction hypothesis. Therefore, $D=(D-C)+C$ is also effective.
\end{proof}

This property is not true for weak Del pezzo surfaces, here is an example:

\begin{example}
Let $X$ be a weak del pezzo surface of degree $4$: $X_{4,4A_1}$, its irreducible $(-2)$-curves are $E_1-E_2,L_{123},E_4-E_5,L_{345}$, the sum of these curves is $2L_{235}$ which is effective divisor on $X$ but $L_{235}$ itself is a strong left-orthgonal divisor, hence not an effective divisor on $X$
\end{example}

\begin{lemma}\label{eff=>AE}
Let $D$ be a divisor on a weak del Pezzo surface $X$ such that $D\cdot R\ge0$ for all $(-2)$-curves $R$. If $D$ is effective, then $D$ is absolutely effective. In other words, if $D$ is conditionally effective, then $D$ is not effective.
\end{lemma}
\begin{proof}
Assuming that $D$ is effective, $D$ can be written as the sum of some generators of the effective monoid. 
By Lemma \ref{eff monoid}, we can therefore write $D=D_0+\gamma$, where $D_0=D-\gamma$ is a sum of $(-1)$-curves possibly with $-K$ and hence absolutely effective, and $\gamma=\sum R_i$ is a sum of $(-2)$-curves. 
Since $D\cdot R\ge0$ for all $(-2)$-curves $R$, in particular for those $R_i$ that appears in $\gamma$, the following lemma
shows that $D$ is absolutely effective.
\end{proof}

\begin{lemma}
Let $D$ be a divisor on a weak del Pezzo surface, and let $\gamma$ be a sum of (not necessarily distinct) $(-2)$-classes $R_1,R_2,\dots,R_n$ such that $D\cdot R_i\ge0$ for $1\le i\le n$. If $D-\gamma$ is absolutely effective, then $D$ is also.
\end{lemma}
\begin{proof}
Without loss of generality, we may assume that $\gamma=\sum_{i=1}^n R_i$ is the shortest possible expression, i.e. there is no way to write $\gamma=\sum_{i=1}^{n'}R'_i$ for $(-2)$-classes $R'_1,\dots,R'_{n'}$ with $n'<n$ such that $D\cdot R'_i\ge0$ for $1\le i\le n'$. I claim this implies that $R_i\cdot R_j\le0$ for all $1\le i,j\le n$. Otherwise, $R_i\cdot R_j=1$ or $2$: if $R_i\cdot R_j=1$, then $R_i+R_j$ is a $(-2)$-class and $D\cdot(R_i+R_j)\ge0$, so we get a shorter expression of $\gamma$, reducing $n$ by 1; if $R_i\cdot R_j=2$, then $R_i+R_j=0$, and when we remove them from the expression, $n$ is reduced by 2. Thus we must have $R_i\cdot R_j\le0$.

We proceed by induction on $n$. The $n=0$ case is trivial. If $n>0$, $(D-R_1)\cdot R_j\ge0$ for all $2\le j\le n$ since $D\cdot R_j\ge0$ and $R_1\cdot R_j\le0$. Since $\gamma-R_1=\sum_{j=2}^n R_j$, by the induction hypothesis, if $D-\gamma=(D-R_1)-(\gamma-R_1)$ is absolutely effective, then $D-R_1$ is also. Since the absolutely effective divisors are stable under the Weyl group, we conclude that $r_{R_1}(D-R_1)=(D-R_1)+k R_1=D+(k-1)R_1$ is also effective, where $k=(D-R_1)\cdot R_1=D\cdot R_1-R_1\cdot R_1\ge2$, and where $r_R$ is the reflection with respect to the hyperplane orthogonal to $R$, an element of the Weyl group. Therefore, $\frac{k-1}{k}(D-R_1)+\frac{1}{k}(D+(k-1)R_1)=D$ is absolutely effective.
\end{proof}

\begin{lemma}\label{D^2 eff}
Let $D$ be a divisor on a weak del Pezzo surface $X$ such that $D\cdot K\le0$ and $D^2\ge D\cdot K$, then $D$ is 
effective.
\end{lemma}
\begin{proof}
As in Lemma \ref{nef=>eff}, $h^2(D)=h^0(K-D)=0$ since $(K-D)\cdot(-K)=-d+D\cdot K<0$, so $h^0(D)\ge h^0(D)-h^1(D)=\chi(D)=\frac12D(D-K)+1\ge1$.
\end{proof}

In light of the above lemmas, we have the following simplified algorithm to determine whether a conditionally effective anti-class $D$ on a weak del Pezzo surface $X$ is effective. If $D\cdot K>0$, then $D$ cannot be effective since $-K$ is nef. If $D\cdot K\le0$, the algorithm is just one sentence:

{\bf Compute $D\cdot R$ for all $(-2)$-classes $R$. If $D\cdot R\ge0$ for all $R$, conclude ``not effective"; if $D\cdot R\le-2$ for some $R$, conclude ``effective"; otherwise, pick an $R$ with $D\cdot R=-1$, replace $D$ by $D-R$, and repeat.}

This algorithm is guaranteed to terminate by the same argument as before: we may only subtract $(-2)$-curves finitely many times before increasing $D^2$, but subtracting $R$ with $D\cdot R=-1$ will never change $D^2$, so we must fall into the other two cases at some point and then the algorithm terminates immediately. We are now just two steps away from the correctness of the algorithm.

\begin{lemma}
In the above algorithm, whenever we replace $D$ by $D-R$, we have
\begin{enumerate}
\item $(D-R)^2=D^2$ and $(D-R)\cdot K=D\cdot K$.
\item $D-R$ is conditionally effective iff $D$ is.
\item $D$ is effective iff $D-R$ is.
\end{enumerate}
\end{lemma}
\begin{proof}
$D-R$ is effective iff $D$ is effective by Lemma \ref{eff equiv}(1), since $D\cdot R<0$ whenever we replace $D$ by $D-R$. Since we only replace $D$ by $D-R$ when $D\cdot R=-1$, $D-R=D+(D\cdot R)R$ is just the image of $D$ under the simple reflection $r_R$, so $D$ is conditionally effective iff $D-R=r_R(D)$ is. Moreover, $(D-R)^2=D^2-2D\cdot R+R^2=D^2-2(-1)+(-2)=D^2$ and $(D-R)\cdot K=D\cdot K$ because $R\cdot K=0$.
\end{proof}

\begin{lemma}
The algorithm is correct.
\end{lemma}
\begin{proof}
By the previous lemma (2), the final $D$ (just before the termination of the algorithm), call it $D'$, is conditionally effective since the original $D$ is. If $D'\cdot R\ge0$ for all $R$, Lemma \ref{eff=>AE} shows that $D'$ is not effective, and hence the original $D$ is not effective.

Again by the previous lemma (1), $D'\cdot K\le0$ and $D'^2-D'\cdot K=-2$ since the original $D$ is an anti-class with $D\cdot K\le0$. If $D'\cdot R\le-2$ for some $R$, then $(D'-R)^2=D'^2-2D'\cdot R+R^2\ge (D'\cdot K-2)-2(-2)+(-2)=(D'-R)\cdot K$ and $(D'-R)\cdot K\le 0$, so by Lemma \ref{D^2 eff}, $D'-R$ is effective, and hence the original $D$ is effective.
\end{proof}

In the actual implementation of the above algorithm, we attempt to subtract multiple $(-2)$-curves at a time. Namely, instead of picking an $R$ with $D\cdot R=-1$, we define $D_0$ to be the sum of all $R_i$ with $D\cdot R_i=-1$, and subtract $D_0$ from $D$, with a caveat: this only works if these $R_i$ are pairwise disjoint, which guarantees that $(D-R_{i_1}-\dots-R_{i_k})\cdot R_{i_{k+1}}=-1$. If instead $R_i\cdot R_j>0$, then $(D-R_i)\cdot R_j\le -2$ and we should instead conclude ``effective" immediately. It can be shown that this happens iff $D_0\cdot R=-1$ for some $(-2)$-curve $R$ (think about the Dynkin diagram).

Moreover, since all input that we need to run this algorithm is the intersection numbers of $D$ with all $(-2)$-curves and the intersection numbers between the $(-2)$-curves (aka the Cartan matrix), we pre-compute in Mathematica the intersections of the anti-$r$-classes $D$ in the initial toric system with all $(-2)$-classes, and in GAP 3 compute  $w(D)\cdot R$ as $D\cdot w^{-1}(R)$, since any Weyl group element $w$ is an isometry.\\

Finally, we show that the anti-$r$-classes arising as sums of terms in toric systems of Types III--VI are conditionally effective.

\begin{lemma}
Let $\{D_i\}_{i=1}^n$ be some of the terms in an initial toric system of Types III--VI, and let $D=\sum_{i=1}^n D_i$. If we write $D=a_0 L+\sum_{j=1}^{9-d} a_j E_j$, then $a_0\ge0$, and if $a_0=0$, then there exists $1\le j\le 9-d$ such that $a_j>0$.
\end{lemma}
\begin{proof}
Any term in any initial toric system of Types III--VI has the form $L-\sum_{k=1}^m E_{j_k}$ or $E_j-\sum_{k=1}^m E_{j_k}$ with $j_k>j$ for $1\le k\le m$. This can be verified by inspection, but the general reason is that every term can be obtained as the proper transform of $L$ or some exceptional curve $E_j$ under iterated blowings-up (at a point every time) of $\mathbb{P}^2$. Since $E_j$ is the exceptional curve of the $j$th blow-up, it can only be proper-transformed in the future blowing-ups, yielding $j_k>j$. Since $D$ is a sum of such terms, it is clear that the $L$ coefficient $a_0$ is nonnegative. If $a_0=0$, then all $D_i$ has zero $L$ coefficient, hence is of the form $E_{j(i)}-\sum_{k=1}^{m(i)} E_{j(i)_k}$ with $j(i)_k>j(i)$. Let $j:=\min_{1\le i\le n} j(i)$, then the $E_j$ coefficient of $D$ equals the number of $i$'s with $j(i)=j$, which is a positive number. Thus the lemma is proved.
\end{proof}

\begin{lemma}
Let $\{D'_i\}_{i=1}^n$ be some of the terms in an arbitrary toric system of Types III--VI, then $-\sum_{i=1}^n D'_i$ is conditionally effective.
\end{lemma}
\begin{proof}
By Proposition \ref{Weyl trans}, there exists an element $w$ of the Weyl group and some terms $\{D_i\}_{i=1}^n$ in the initial toric system of the same type, such that $D'_i=w(D_i)$ for $1\le i\le n$, so that $-\sum_{i=1}^n D'_i=w(-D)$ where $D:=\sum_{i=1}^n D_i$. Since the conditionally effective divisors are stable under the Weyl group, it suffices to show that $-D$ is conditionally effective.

Suppose for contradiction that $-D$ is absolutely effective, so that $-D=\sum_{i=1}^m b_i C_i$ for some $(-1)$-classes and $b_i\in\R_{\ge0}$.
By the previous lemma, $-D=-a_0 L+\sum_{j=1}^{9-d}-a_jE_j$ with $-a_0\le0$. Since all $(-1)$-classes has nonnegative $L$-coefficients and the only ones with zero $L$-coefficients are $E_1,E_2,\dots,E_{9-d}$, we see that each $C_i$ is $E_{j(i)}$ for some $1\le j(i)\le 9-d$, so $a_0=0$ and all $E_j$-coefficients of $-D$ are nonnegative. But this leads to a contradiction since the previous lemma implies that $-a_j<0$ for some $1\le j\le 9-d$.
\end{proof}

\end{document}